\documentclass[12pt]{article}

\usepackage[margin=1in]{geometry}
\usepackage{xurl}
\usepackage{comment}
\usepackage{amsthm}
\usepackage{amssymb}
\usepackage{amsmath}
\usepackage{enumitem}
\usepackage{graphicx}
\usepackage{mathdots}
\usepackage{mathtools}
\usepackage{url}
\usepackage{color}
\usepackage{framed}
\usepackage{tikz}
\usepackage{tikz-cd}
\usepackage[ruled,vlined]{algorithm2e}
\usepackage{array}
\usepackage{hyperref}
\usepackage{cleveref}
\usepackage{mathtools}
\newcolumntype{L}{>{$}l<{$}} 

\newcommand{\llangle}{\langle\hspace{-2.5pt}\langle}

\newcommand{\Oname}{\operatorname{O}}
\newcommand{\Id}{\operatorname{Id}}

\newcommand{\SO}{\operatorname{SO}}

\newcommand{\E}{\operatorname{E}}
\renewcommand{\O}{\operatorname{O}}
\newcommand{\ev}{\operatorname{ev}}

\newcommand{\RR}{\mathbb{R}}

\newcommand{\rrangle}{\rangle\hspace{-2.5pt}\rangle}

\newcommand{\longto}[1][\ ]{\xrightarrow{\quad #1 \quad}}

\newcommand{\RotImplies}{\rotatebox[origin=c]{-90}{$\implies$}}
\newcommand{\subsmall}{\fontsize{10.25pt}{11pt}\selectfont}

\newlength{\claimstepvspace} 
\newlength{\theoremstepvspace} 
\newtheoremstyle{mystyle}
  {}
  {}
  {}
  {}
  {\bfseries}
  {.}
  { }
  {\thmname{#1}\thmnumber{ #2}\thmnote{ (#3)}}

\newtheorem{theorem}{Theorem}
\newtheorem{corollary}[theorem]{Corollary}
\newtheorem{lemma}[theorem]{Lemma}
\newtheorem{proposition}[theorem]{Proposition}

\theoremstyle{definition}

\newtheorem{example}[theorem]{Example}
\newtheorem{definition}[theorem]{Definition}
\newtheorem{remark}[theorem]{Remark}

\newcounter{claimcounter}
\theoremstyle{mystyle}

\crefname{claim}{Claim}{Claims}
\crefname{figure}{Figure}{Figures}
\crefname{table}{Table}{Tables}

\crefname{lemma}{Lemma}{Lemmas}
\crefname{example}{Example}{Examples}
\crefname{definition}{Definition}{Definitions}
\crefname{remark}{Remark}{Remarks}
\crefname{theorem}{Theorem}{Theorems}
\crefname{corollary}{Corollary}{Corollaries}
\crefname{proposition}{Proposition}{Propositions}
\crefname{section}{Section}{Sections}
\crefname{appendix}{Appendix}{Appendices}
\crefname{problem}{Problem}{Problems}
\crefname{claimcounter}{Claim}{Claims}

\newlist{propenum}{enumerate}{1} 
\setlist[propenum]{label=(\alph*), ref=\theproposition(\alph*)}
\crefalias{propenumi}{proposition} 

\newlist{claimenum}{enumerate}{1} 
\setlist[claimenum]{label=(\alph*), ref=\theclaim(\alph*)}
\crefalias{claimenumi}{claim}

\newlist{lemenum}{enumerate}{1} 
\setlist[lemenum]{label=(\alph*), ref=\thelemma(\alph*)}
\crefalias{lemenumi}{lemma} 

\newlist{thmenum}{enumerate}{1} 
\setlist[thmenum]{label=(\alph*), ref=\thetheorem(\alph*)}
\crefalias{thmenumi}{theorem} 

\newlist{corenum}{enumerate}{1} 
\setlist[corenum]{label=(\alph*), ref=\thecorollary(\alph*)}
\crefalias{corenumi}{corollary} 

\title{Estimating the Euclidean distortion of an orbit space}
\author{
Ben~Blum-Smith\footnote{Department of Applied Mathematics and Statistics, Johns Hopkins University, Baltimore, MD}
\quad
Harm~Derksen\footnote{Department of Mathematics, Northeastern University, Boston, MA}
\quad
Dustin~G.~Mixon\footnote{Department of Mathematics, The Ohio State University, Columbus, OH} \footnote{Translational Data Analytics Institute, The Ohio State University, Columbus, OH}\\
\quad
Yousef~Qaddura\footnotemark[3]
\quad
Brantley~Vose\footnotemark[3]
}

\date{}

\makeatletter
\newcommand{\bigperp}{%
  \mathop{\mathpalette\bigp@rp\relax}%
  \displaylimits
}

\newcommand{\bigp@rp}[2]{%
  \vcenter{
    \m@th\hbox{\scalebox{\ifx#1\displaystyle2.1\else1.5\fi}{$#1\perp$}}
  }%
}
\makeatother

\newcommand{\Plu}{\operatorname{Plu}}
\newcommand{\Tr}{\operatorname{Tr}}
\newcommand{\U}{\operatorname{U}}

\renewcommand{\Re}{\operatorname{Re}}
\DeclareMathOperator*{\argmin}{arg\,min}

\begin{document}
\maketitle
\begin{abstract}
Given a finite-dimensional inner product space $V$ and a group $G$ of isometries, we consider the problem of embedding the orbit space $V/G$ into a Hilbert space in a way that preserves the quotient metric as well as possible.
This inquiry is motivated by applications to invariant machine learning.
We introduce several new theoretical tools before using them to tackle various fundamental instances of this problem.
\end{abstract}

\tableofcontents

\section{Introduction}

In this paper, we construct low-distortion bilipschitz embeddings (into Euclidean spaces) of quotients of Euclidean spaces by various important group actions. 
In doing so, we contribute to a young but rapidly growing body of work aimed at understanding how, and to what extent, orbit spaces can be embedded into a Hilbert space with minimal metric distortion \cite{AgarwalRT:20, AmirBDE:25, BalanD:22, BalanHS:22, BalanT:23, CahillIM:24, CahillIMP:22, Eriksson:18, HavivR:10, HeimendahlLVZ:22, MixonP:22, MixonQ:22, Qaddura:25, VallentinM:23, Zolotov:19}; this research program may be termed {\em bilipschitz invariant theory}. 
In what follows, we present some motivation (\cref{sec.motivation}), we review the relevant literature (\cref{sec.related-work}), and we summarize our results (\cref{sec.results}).

\subsection{Motivation}\label{sec.motivation}

Many machine learning algorithms are designed to operate on data that sits in a Euclidean space.
Other types of data, such as text, shapes, or graphs, must undergo vectorization before one can bring the standard tools to bear. Even methods that theoretically apply to data in a general metric space can benefit from vectorization.
For instance, fast randomized nearest neighbor algorithms, such as the one in \cite{JonesOR:11}, can efficiently approximate nearest neighbors in large Euclidean datasets by avoiding explicit computation of all pairwise distances.

In many cases, the naive vector representation of an object is not unique.
For example, one might represent a point cloud of $n$ points in $\mathbb{R}^d$ as a member $x$ of $V:=\mathbb{R}^{d\times n}$, but the same point cloud can also be represented by any permutation of the columns of $x$.
Observe that this ambiguity arises from a group $G$ of isometries of $V$, namely, the group of column permutations.

To address such ambiguities, one might factor them out by identifying objects with members $[x]:=G\cdot x$ of the orbit space $V/G$.
In general, whenever a group $G$ acts isometrically on a metric space $X$, the orbit space $X/G$ inherits a pseudometric defined by
\begin{equation*}
\label{eq.distance def}
d([x],[y])
:=\inf_{\substack{p\in[x]\\q\in[y]}}\|x-y\|,
\end{equation*}
which is an honest metric in the usual case where the orbits of $G$ are closed.
In order to make use of the vast array of Euclidean-based machine learning algorithms, we are inclined to embed the orbit space into Euclidean space while simultaneously minimizing the resulting distortion of the quotient metric.
We elect to quantify distortion in the following manner.

\begin{definition}\label{def.distortion}
Given two metric spaces $X$ and $Y$ and a map $f\colon X\to Y$, take $\alpha,\beta\in[0,\infty]$ to be the largest and smallest constants (respectively) such that
\[
\alpha \cdot d_X(x,x') 
\leq d_Y\big(f(x),f(x')\big)
\leq \beta\cdot d_X(x,x') 
\qquad 
\forall x,y\in X.
\]
Then $\alpha$ and $\beta$ are called the \textbf{(optimal) lower and upper Lipschitz bounds} of $f$, respectively. 
The \textbf{distortion} of $f$ is given by $\kappa(f):= \frac{\beta}{\alpha}$, where we take $\kappa(f)=\infty$ if $\alpha=0$ or $\beta = \infty$.
A map with finite distortion is called \textbf{bilipschitz}.
Moreover, the \textbf{Euclidean distortion} of $X$, denoted $c_2(X)\in [1,\infty]$, is the infimum of $c$ for which there exists a Hilbert space $H$ and $f\colon X\to H$ such that $\kappa(f) = c$.
\end{definition}

We can think of the distortion of a map as a generalization of the \textit{condition number} of a matrix. 
Indeed, the distortion $\kappa(f)$ of a linear map $f$ is exactly its condition number. 
Hence, a metric space $X$ has low Euclidean distortion if it can be embedded in Euclidean space in a ``well-conditioned" way.

While there are many ways to quantify the failure of a map to preserve geometry, Section~2 of~\cite{CahillIM:24} shows how minimizing distortion is particularly useful in the context of transferring Euclidean data science algorithms (e.g., nearest-neighbor, clustering, and multidimensional scaling algorithms) to arbitrary metric spaces. 
To illustrate, we recall an example in the context of nearest-neighbor search.

\begin{example}[Example~1 in~\cite{CahillIM:24}] 
\label{ex.bilip nearest}
Given a metric space $(X,d_X)$, data $x_1,\ldots,x_m\in X$, and an approximation parameter $\lambda\geq 1$, the corresponding \textit{$\lambda$-approximate nearest neighbor problem} takes as input $x\in X$ and outputs $j\in\{1,\ldots,m\}$ such that
\[
d_X(x,x_j)
\leq \lambda\cdot \min_{1\leq i\leq m} d_X(x,x_i).
\]
Suppose it is relatively easy to solve this problem in another metric space $Y$ (e.g., when $Y$ is a Euclidean space~\cite{JonesOR:11}).
Given a map $f\colon X\to Y$ with lower and upper Lipschitz bounds $\alpha,\beta\in (0,\infty)$, one may pull back any solver in $Y$ through $f$, thereby allowing one to solve the $\lambda\kappa(f)$-approximate nearest-neighbor problem in $X$.
To see this, first use the solver in $Y$ to find $j\in \{1,\dots,m\}$ such that
\[
d_Y\big(f(x),f(x_j)\big)
\leq \lambda\cdot \min_{1\leq i\leq m} d_Y\big(f(x),f(x_i)\big).
\]
Then $x_j$ is also an approximate nearest neighbor of $x$ in $X$:
\[
d_X(x,x_j)
\leq\frac{1}{\alpha}\cdot d_Y\big(f(x),f(x_j)\big)
\leq\frac{\lambda}{\alpha}\cdot\min_{i\in I}d_Y\big(f(x),f(x_i)\big)
\leq \lambda\cdot\frac{\beta}{\alpha}\cdot\min_{i\in I} d_X(x,x_i).
\]
\end{example}        

In short, data science insists on Euclidean representations of objects.
As such, when objects are naturally represented by orbits, we seek a low-distortion embedding of the orbit space into Euclidean space.


\subsection{Related work}\label{sec.related-work}

\subsubsection{Exact distortions of quotients by groups of isometries}

\begin{table}[t]
\begin{center}
\begin{tabular}{|p{5cm}p{3cm}p{3cm}p{3.5cm}|}
        \hline
        \parbox[c][0.8cm][l]{3cm}{}
        $G$ & $V$ & $c_2(V/G)$ & reference\\
        \hline\hline
        \parbox[c][0.8cm][l]{3cm}{}
        $\O(1)$ & $\mathbb R^{n}$, ~$n\geq 2$ & $\sqrt{2}$ & Cor.~36 in~\cite{CahillIM:24}\\
        \hline
        \parbox[c][0.8cm][l]{3cm}{}
        $\U(1)$ & $\mathbb C^{n}$, ~$n\geq 2$ & $\sqrt{2}$ & Cor.~37 in~\cite{CahillIM:24}\\
        \hline
        \parbox[c][0.8cm][l]{3cm}{}
        $\langle e^{2\pi i/r}\rangle$ & $\mathbb C$ & $r\sin(\tfrac{\pi}{2r})$ & Cor.~38 in~\cite{CahillIM:24}\\
        \hline
        \parbox[c][0.8cm][l]{3cm}{}
        reflection group & $\mathbb R^n$ &  $1$ & Lem.~8 in~\cite{MixonP:22}\\
        \hline
        \parbox[c][0.8cm][l]{3cm}{}
        connected polar group & $\mathbb R^n$ & $1$ & Prop.~6 in~\cite{Dadok:85}\\
        \hline
        \parbox[c][0.8cm][l]{3cm}{}
        rectangular lattice & $\mathbb R^n$ & $\tfrac{\pi}{2}$ & Thm.~6.1 in~\cite{HeimendahlLVZ:22}\\
        \hline
        \parbox[c][0.8cm][l]{3cm}{}
        $A_2$ lattice & $\mathbb R^2$ & $\tfrac{\sqrt{8}\pi}{\sqrt{27}}$ & Thm.~4.1(2) in~\cite{VallentinM:23}\\
        \hline
        \parbox[c][0.8cm][l]{3cm}{}
        $E_8$ lattice & $\mathbb R^8$ & $\tfrac{\sqrt{15}\pi}{\sqrt{32}}$ & Thm.~4.1(3) in~\cite{VallentinM:23}\\
        \hline
    \end{tabular}
    \caption{Previously known Euclidean distortions for $V/G$.\label{table1}}
    \end{center}
    \end{table}

Given a finite-dimensional real Hilbert space $V$ and $G\leq\E(V)$ with closed orbits, what is the Euclidean distortion of the orbit space $V/G$?
We note that $c_2(V/G)< \infty$. To see this, assume without loss of generality that $G$ is closed, since $G$ and its closure $\overline G$ have the same (closed) orbits; see, for instance,~\cite[Lem.~2.1]{Kramer:22}. In particular, $V/G$ is a finite-dimensional \emph{Alexandrov space of nonnegative curvature}; see~\cite[Prop.~3.62, Rmk.~3.80, fn.~17]{AlexandrinoB:15} and the discussion following Theorem~1.6 in~\cite{GroveK:02}. For such a space $X$, Zolotov's work \cite{Zolotov:19} built on \cite{Eriksson:18} to establish $c_2(X) < \infty$. In the case where $G\leq \Oname(V)$, we provide in \cref{app.finite distortion} a streamlined proof of $c_2(V/G)< \infty$ which avoids the metric geometry machinery used in \cite{Zolotov:19} and instead highlights the power of the \textit{slice theorem} for isometric group actions.

\cref{table1} presents all previously known Euclidean distortions for spaces of the form $V/G$.\footnote{The only exceptions we know of are the further consequences that can be derived from the identity $c_2(V/G\times W/H) = \max\{c_2(V/G),c_2(W/H)\}$ established in~\cite{CahillIM:24} (see \cref{prop.max-product}).}
In all cases, the upper bound on $c_2(V/G)$ was obtained by an explicit embedding into a finite-dimensional Hilbert space.
For the first three rows of \cref{table1}, the underlying $G$-invariant maps exhibit a similar form:
\[
\begin{aligned}
\mathbb{R}^n&\to(\mathbb{R}^n)^{\otimes 2}\\[3pt]
\displaystyle x&\mapsto \frac{x\otimes x}{\|x\|}
\end{aligned}
\qquad\qquad
\begin{aligned}
\mathbb{C}^n&\to(\mathbb{C}^n)^{\otimes 2}\\[3pt]
\displaystyle z&\mapsto \frac{z\otimes \overline{z}}{\|z\|}
\end{aligned}
\qquad\qquad
\begin{aligned}
\mathbb{C}&\to\mathbb{R}\times\mathbb{C}\\[3pt]
\displaystyle z&\mapsto \bigg(\cos(\tfrac{\pi}{2r})\cdot|z|,~\sin(\tfrac{\pi}{2r})\cdot\frac{z^r}{|z|^{r-1}}\bigg)
\end{aligned}
\]
and $0\mapsto 0$ in each case.
Furthermore, the matching (nontrivial) lower bound on $c_2(V/G)$ was obtained by semidefinite programming (see Section~7 in~\cite{CahillIM:24}).

In the case of a reflection group, the optimal embedding sends each orbit to the unique representative in a fixed fundamental domain; a proof for finite reflection groups appears in~\cite{MixonP:22}, but it can be easily adapted to treat affine reflection groups as well.
A Lie group $G\leq \O(V)$ is said to be \textit{polar} if there exists a subspace $\Sigma\leq V$, called a \textit{section}, that orthogonally intersects every orbit of $G$.
Let $H\leq G$ denote the largest subgroup under which $\Sigma$ is invariant.
It turns out that if $G$ is connected, then $V/G$ is isometrically isomorphic to $\Sigma/H$, and furthermore, $H$ acts as a reflection group on $\Sigma$.
It follows that $c_2(V/G)=1$.
As an example, the conjugation action of $G=\SO(n)$ on the vector space $V$ of real symmetric $n\times n$ matrices is a connected polar action with section $\Sigma$ consisting of the diagonal matrices, on which $H=S_n$ acts as a reflection group.

We discuss the the last three rows of \cref{table1} in \cref{sec.sub flat tori}. 



\subsubsection{Phase retrieval}

Given $F\in\{\mathbb{R},\mathbb{C}\}$, let $V$ denote a finite-dimensional Hilbert space over $F$, and let $G$ denote the group of scalars in $F$ of unit modulus.
The \textit{phase retrieval problem} is to reconstruct any orbit $[x]\in V/G$ from a collection of $G$-invariant measurements of the form
\[
\Phi([x])=\big(|\langle x,a_i\rangle|\big)_{i=1}^n,
\]
where $a_1,\ldots,a_n\in V$ are known measurement vectors.
A recent line of work \cite{Alharbi:22,BalanW:15,BandeiraCMN:14,CahillCD:16} established that the map $\Phi\colon V/G\to \mathbb{R}^n$ is bilipschitz precisely when it is injective, and furthermore, \cite{BalanCE:06,ConcaEHV:15} established that $\Phi$ is injective for a generic choice of $a_1,\ldots,a_n$ whenever $n\geq 2\operatorname{dim}_{\mathbb{R}}(V/G)-\operatorname{dim}_{\mathbb{R}}(F)$.
More recent work has quantified the distortion of $\Phi$.
In particular, \cite{XiaXX:24} showed that
\[
\kappa(\Phi)
\geq\left\{\begin{array}{cl}
\displaystyle\sqrt{\frac{4}{4-\pi}}&\text{if } F = \mathbb R,\\[16pt]
\displaystyle\sqrt{\frac{\pi}{\pi-2}}&\text{if } F = \mathbb C.
\end{array}\right.
\]
Furthermore, if $a_1,\ldots,a_n$ are independent Gaussian vectors, then with high probability as $n\to\infty$, the distortion $\kappa(\Phi)$ concentrates towards this lower bound.
Meanwhile, this bound is strictly larger than the Euclidean distortion $c_2(V/G)=\sqrt{2}$ given in~\cref{table1}.

\subsubsection{Max filtering}

Max filtering was introduced in~\cite{CahillIMP:22} as a generalization of phase retrieval to arbitrary compact subgroups $G\leq \O(d)$. 
The \textit{max filtering map} $\llangle\cdot,\cdot\rrangle\colon \mathbb R^d/G\times \mathbb R^d/G\to \mathbb R$ is defined by
\[
\llangle [x],[y]\rrangle 
:= \sup_{\substack{p\in [x]\\q\in [y]}} \langle p,q\rangle.
\]
Given \textit{templates} $a_1,\ldots,a_n\in\mathbb R^d$, the \textit{max filter bank} $\Phi\colon \mathbb R^d/G \to \mathbb R^n$ is defined by
\[
\Phi([x]) 
:= \big(\llangle [x],[a_i]\rrangle\big)_{i=1}^n.
\]
For finite $G$, it was shown in~\cite{BalanT:23} that every injective max filter bank is bilipschitz. 
Moreover, \cite{CahillIMP:22} showed that generic templates yield injectivity whenever $n\geq 2d$. 
When the templates are drawn from a standard Gaussian distribution, \cite{MixonQ:22} established the distortion bound
\[
\kappa(\Phi) 
\leq \left(4e^{\frac{3}{2}}|G|^{\frac{5}{2}} \ln^{\frac{1}{2}}(e|G|)\right)^{1+\varepsilon}
\]
for arbitrary $\varepsilon > 0$ and large enough $n \geq N(\varepsilon)$, with high probability. 
This yields the best known general bound
\begin{equation}
\label{eq.max filter bank bound}
c_2(\mathbb R^d/G) 
\leq 4e^{\frac{3}{2}}|G|^{\frac{5}{2}}\ln^{\frac{1}{2}}(e|G|),
\end{equation}
which is valid for all finite subgroups $G\leq \O(d)$.
In the case where $G\leq \Oname(d)$ is a compact (and possibly infinite) subgroup, \cite{MixonQ:22} again proved generic injectivity under the condition $n\geq 2d$. 
More recently, \cite{Qaddura:25} established bilipschitz behavior for generic $\Phi$, provided that all nonzero orbits of $G$ have equal dimension. 
It remains open whether every injective max filter bank is bilipschitz when $G$ is compact.

\subsubsection{Flat tori}
\label{sec.sub flat tori}

For nonorthogonal subgroups $G\leq \operatorname{E}(d)$, the only nontrivial bounds on $c_2(\mathbb R^d/G)$ that we know of take $G$ to be a translation lattice $T$, meaning $\mathbb R^d/T$ is a flat torus. 
In this setting, the sharpest known asymptotic upper bound
\[
c_2(\mathbb R^d/T) 
\leq \O(\sqrt{d\log d})
\]
was established for all lattices in \cite{AgarwalRT:20} using Gaussian densities and building on earlier work by~\cite{HavivR:10}.
In dimension two, the proof of~\cite[Theorem~3]{HavivR:10} provides the following explicit upper bound for all lattices:
\[
c_2(\mathbb R^2/T) 
< 8.
\]
In the other direction, \cite{HavivR:10} used a Fourier-based lower bound from~\cite{KhotN:06} to show that for sufficiently large $d$, there exists a translation lattice $T$ such that
\[
c_2(\mathbb R^d/T)
\geq \Omega(\sqrt d). 
\]
This demonstrates that high-dimensional flat tori can be highly non-Euclidean.

Recently, \cite{VallentinM:23,HeimendahlLVZ:22} used semidefinite programming to obtain more detailed estimates. 
The strongest general lower bound to date appears in~\cite[Theorem~5.1]{HeimendahlLVZ:22}:
\[
c_2(\mathbb R^d/T) 
\geq \frac{\pi}{\sqrt d}\lambda^*(T)\mu(T),
\]
where $\mu(T)$ denotes the covering radius of the lattice and $\lambda^*(T)$ denotes the length of the shortest nonzero vector in the dual lattice. 
A simpler proof of this bound was later given in~\cite{VallentinM:23}.
For certain specific lattices, exact Euclidean distortions have been computed. 
If $T$ is rectangular, i.e., generated by an orthogonal basis, then
\[
c_2(\mathbb R^d/T)
=\frac{\pi}{2},
\]
since $\mathbb R^d/T$ is a product of circles, and $c_2(S^1) = \frac{\pi}{2}$ by Theorem~6.1 in~\cite{HeimendahlLVZ:22} (see \cref{prop.max-product,prop.known-distortions}). 
In dimension two, the Euclidean distortion of a torus has been expressed as the solution to a specific nonconvex optimization problem over a Voronoi cell in $\mathbb R^2$; see~\cite[Theorem~8.3]{HeimendahlLVZ:22}. 
Building on this, \cite[Theorem~4.1(2)--(3)]{VallentinM:23} established that
\[
c_2(\mathbb R^2/A_2) 
= \frac{\sqrt{8}\pi}{\sqrt{27}} \qquad \text{and} \qquad c_2(\mathbb R^8/E_8) 
= \frac{\sqrt{15}\pi}{\sqrt{32}},
\] 
where $A_2\subseteq \mathbb R^2$ denotes the hexagonal lattice, and $E_8\subseteq \mathbb R^8$ denotes the root lattice corresponding to the exceptional root system in $\mathbb R^8$.

To date, the rectangular lattices in $\mathbb R^d$, the hexagonal lattice in $\mathbb R^2$, the $E_8$ root lattice in $\mathbb R^8$, and products thereof are the only translation lattices for which the Euclidean distortion of the associated orbit space is known in closed form.

\subsubsection{Quotients of infinite-dimensional Hilbert spaces}
\label{sec.sub hilbert quotients}

To our knowledge, \cite{CahillIM:24} is the only work that has investigated the distortion of quotients of infinite-dimensional Hilbert spaces by subgroups of their linear isometries. 
In this context, special care must be taken when defining the quotient space, as orbits may not be closed; the equivalence class of a point in $V$ is instead defined to be the \emph{closure} of its orbit.
With this setup, \cite{CahillIM:24} studied three particular classes of quotient spaces.
First, they studied the action of the base field's unit-scalar group on infinite-dimensional Hilbert spaces, determining an exact Euclidean distortion of $\sqrt 2$ for the resulting quotients.
Second, they considered the quotient of $ \ell^2(\mathbb N\to \mathbb R^d)$ by the group $S_\infty$ of bijections of $\mathbb N$ acting via precomposition.
They showed that the distortion of the quotient is $1$ if $d=1$ and $\infty$ if $d\geq 3$; the case $d=2$ remains open.
Finally, they analyzed the quotient of $ \ell^2(\mathbb Z \to \mathbb R)$ by the group $\mathbb Z$ of translations acting via precomposition. 
In this case, the distortion gives a lower bound on the distortion of many other related quotients, but it remains open whether this distortion is infinite.

\subsection{Summary of results}\label{sec.results}

This paper presents two different flavors of contributions. 
In \cref{sec.tools}, we develop general tools for bounding the distortion of metric quotients, and in \cref{sec.applications}, we apply these tools to various families of quotients by groups of Euclidean isometries.

\subsubsection{General results}

In \cref{sec.symmetry lemmas}, we present two particularly useful \textit{equivariant embedding lemmas}. 
For any metric space $X$ equipped with an isometric action by a compact group $G$, the first lemma states that any embedding of $X$ into a Hilbert space can be promoted to a $G$-equivariant bilipschitz map with the same (or better) distortion. 
The second lemma states that a $G$-equivariant map $f\colon X\to Y$ descends to a map between quotient spaces $f_{/G}\colon X/G\to Y/G$ with the same (or better) distortion.
    
In \cref{sec.23-groups}, we introduce the \textit{Euclidean contortion} $\Upsilon(G)$ of an abstract group $G$ of finite order, which is the largest possible Euclidean distortion of a quotient $V/G$ where $V$ is a finite-dimensional orthogonal $G$-representation.
For any metric space $X$ equipped with an isometric action by $G$, we establish in \cref{thm.upsilon-bound} the inequality
\[
c_2(X/G)
\leq \Upsilon(G)\cdot c_2(X).
\]
We furthermore compute $\Upsilon(G)$ exactly for the groups $G$ of order at most three and provide tools for bounding the contortions of larger finite groups.

In \cref{sec.reflection like}, we introduce a mechanism that promotes an embedding of $X/G$ to an embedding of $X$.
The resulting \textit{quotient--orbit embedding} induces an upper bound on $c_2(X)$ in terms of $c_2(X/G)$; see \cref{thm.align preserve}. 
As an application, we establish in \cref{corr.glued space} the inequality
\[
c_2(Y)^2
\leq 2\cdot c_2(X)^2 +2,
\]
where $Y$ is the metric space formed by gluing two copies of $X$ along a closed subset $Z\subseteq X$.

In \cref{sec.isotropy}, we use a local approximation of a metric space to obtain a lower bound on the Euclidean distortion of that metric space.
In particular, for every Riemannian manifold $M$ equipped with a wandering\footnote{An action of a group $G$ on a topological space $X$ is called \emph{wandering} if for each $x\in X$, there exists an open neighborhood $U\subseteq X$ such that $\{g\in G:g\cdot U\cap U\neq \varnothing\}$ is finite.} isometric action by a discrete group $G$, where the stabilizer $G_p$ of $p$ acts on $T_pM$ via the differential, we prove that
\[
c_2(M/G)
\geq c_2(T_pM/G_p).
\]
Next, given a finite group $G\leq \O(W)$ and a subspace $V\leq W$ with pointwise stabilizer $G_V$, we use the above inequality to prove the lower bound
\begin{equation}
\label{eq.wed ineq}
c_2(W/G)
\geq c_2(V^\perp/G_V).
\end{equation}

In \cref{sec.euclidean tool}, we consider any Euclidean isometry group $\Gamma\leq \E(V)$ whose translation subgroup $T$ is a subspace and whose point group $G\leq \O(V)$ is closed. 
We show that $T^\perp$ is $G$-stable and $V/\Gamma$ is isometric to $T^\perp/G$.

\subsubsection{Applications}

\begin{table}[t]
\begin{center}
\begin{tabular}{|p{5cm}p{3.5cm}p{3.5cm}p{2cm}|}
        \hline
        \parbox[c][0.8cm][l]{3cm}{}
        $G$ & $V$ & $c_2(V/G)$ & reference\\
        \hline\hline
        \parbox[c][0.8cm][l]{3cm}{}
        $\langle e^{2\pi i/r}\rangle$ & $\mathbb C^n$ & $r\sin(\tfrac{\pi}{2r})$ & Thm.~\ref{thm.root unity} \\
        \hline
        \parbox[c][0.8cm][l]{3cm}{}
        $\O(r)$ & $\mathbb R^{r\times n}$, ~$n\geq 2$ & $\sqrt{2}$ & Prop.~\ref{prop.gram bounds} \\
        \hline
        \parbox[c][0.8cm][l]{3cm}{}
        $\U(r)$ & $\mathbb C^{r\times n}$, ~$n\geq 2$ & $\sqrt{2}$ & Rmk.~\ref{rk.U(r)} \\
        \hline
        \parbox[c][0.8cm][l]{3cm}{}
        $\SO(r)$ & $\mathbb R^{r\times n}$, ~$n\geq r\geq 2$ &  $[\,\sqrt{2},\,2\sqrt{2}\,]$ & Thm.~\ref{thm.so psi lower bound} \\
        \hline
        \parbox[c][0.8cm][l]{3cm}{}
        $\SO(n)$ $\cap$ reflection group & $\mathbb R^n$ & $[\,\sqrt{2},\,2\,]$ & Thm.~\ref{thm.alternating subgroup} \\
        \hline
        \parbox[c][0.8cm][l]{3cm}{}
        wallpaper group $**$ & $\mathbb R^2$ & $\frac{\pi}{2}$ & Thm.~\ref{thm.wallpaper groups} \\
        \hline
        \parbox[c][0.8cm][l]{3cm}{}
        wallpaper group $2{*}22$ & $\mathbb R^2$ & $\sqrt{2}$ & Thm.~\ref{thm.wallpaper groups} \\
        \hline
        \parbox[c][0.8cm][l]{3cm}{}
        wallpaper group $4{*}2$ & $\mathbb R^2$ & $2\sqrt{2-\sqrt{2}}$ & Thm.~\ref{thm.wallpaper groups} \\
        \hline
        \parbox[c][0.8cm][l]{3cm}{}
        $\E(r)$ & $\mathbb R^{r\times n}$, ~$n\geq 3$ & $\sqrt{2}$ & Ex.~\ref{cor.euclidean-orthogonal-isometry} \\
        \hline
    \end{tabular}
    \caption{New bounds on Euclidean distortions for $V/G$.\label{table2}}
    \end{center}
    \end{table}

In \cref{sec.root unity}, we consider the orbit space $\mathbb C^n/C_r$, where $C_r := \langle e^{2\pi i/r}\rangle$. 
In the case where $n=1$, the Euclidean distortion was previously determined by Corollary~38 in \cite{CahillIM:24} (see \cref{prop.known-distortions}).
In \cref{thm.root unity}, we establish that a quotient--orbit embedding delivers the same distortion in general:
\[
c_2(\mathbb C^n/ C_r) 
= r\sin(\tfrac{\pi}{2r}).
\]
    
In~\cref{sec.so}, we start by enunciating a lower bound on the Euclidean distortions of the orbit spaces\footnote{For each $F\in\{\mathbb R,\mathbb C\}$, we view the matrix space $F^{r\times n}$ as a real Hilbert space with Frobenius inner product $\langle X,Y\rangle := \operatorname{Re}(\operatorname{Tr}(X^* Y))$, and given $G\leq\operatorname{GL}(r,F)$, we assume without mention that $G$ acts on $F^{r\times n}$ by left matrix multiplication.} $\mathbb{R}^{r\times n}/\O(r)$ and $\mathbb{C}^{r\times n}/\operatorname{U}(r)$ that matches the upper bound given by~\cite{BalanD:22}:
\[
c_2\big(\mathbb{R}^{r\times n}/\O(r)\big)
=\sqrt{2},
\qquad
c_2\big(\mathbb{C}^{r\times n}/\operatorname{U}(r)\big)
=\sqrt{2},
\]
provided $n\geq 2$.
These in turn generalize the $r=1$ cases, which were treated in Corollaries~36 and~37 in~\cite{CahillIM:24} (see \cref{prop.known-distortions}).
We then use a quotient--orbit embedding to extend this analysis to the related orbit space $\mathbb R^{r\times n}/\SO(r)$.
In particular, a careful modification of the Gram matrix and Pl\"ucker coordinates leads to the two-sided bounds
\[
\sqrt 2 
\leq c_2\big(\mathbb R^{r\times n}/\SO(r)\big) 
\leq 2\sqrt 2
\]
for all $n\geq r\geq 2$; see \cref{thm.so psi lower bound}.    

In \cref{sec.alternating}, we consider a finite-group analogy to the above relationship between $\O(r)$ and $\SO(r)$.
In particular, we recall that for a reflection group $G\leq \O(V)$, the mapping that sends each orbit in $V/G$ to the unique representative in a fixed Weyl chamber produces an embedding $V/G\to V$ of distortion $c_2(V/G)=1$~\cite{MixonP:22}.
We extend this analysis to the orbit space of the subgroup $G^+ = G\cap \SO(V)$ by viewing $V/G^+$ as a glued space consisting of two copies of $V/G$:
\[
\sqrt 2
\leq c_2(V/G^+)
\leq 2,
\]
whenever $G^+$ is nontrivial. 
Currently, this Euclidean distortion is exactly known only in the case where $V$ is the plane and $G=D_n$, i.e., $G^+=C_n$, in which case Corollary~38 in \cite{CahillIM:24} (see \cref{prop.known-distortions}) gives
\[
c_2(V/G^+)
=n\sin(\tfrac{\pi}{2n})
\in[\sqrt 2,\tfrac{\pi}{2}),
\]
which increases towards $\frac{\pi}{2}$ as $n\to \infty$.


In \cref{sec.tran klien}, we analyze quotients of the plane by the wallpaper groups.
In particular, we determine the exact Euclidean distortion for a handful of these quotients, and we obtain two-sided bounds in all other cases.
Overall, we find that
\[
c_2(\mathbb R^2/G)
\leq 8\sqrt 2
\]
for every wallpaper group $G$.

In \cref{sec.euclidean}, we apply the result of \cref{sec.euclidean tool} by considering subgroups of the form $G\ltimes V\leq \E(V^n)$, where $G\leq \O(V)$ is a compact group and the action is given by the diagonal action on the $n$-fold direct sum. 
In \cref{thm.EG}, we establish that
\[
c_2\big(V^n/(G\ltimes V)\big) 
= c_2(V^{n-1}/G)
\]
for all $n\geq 2$. 
In the special case where $G=\O(r)$, this yields the explicit formula:
\[
c_2(\RR^{r\times n}/\E(r)) 
= \left\{\begin{array}{cl}
        \sqrt{2} &\text{if } n\geq 3,\\
        1 &\text{if } n=2.
    \end{array}\right.
\]

In \cref{sec.wednesday theorem}, we demonstrate the utility of the general bound \eqref{eq.wed ineq} by exhibiting families of group quotients whose distortions grow unboundedly with the size of the group. 
These groups arise from real-world instances where an object of interest is represented as a matrix only after selecting an arbitrary labeling of sorts, thereby introducing a permutation ambiguity.

\section{Toolbox of general results}
\label{sec.tools}

In this part, we introduce several tools we developed to bound the Euclidean distortion of a metric space quotiented by isometries.

\subsection{A pair of equivariant embedding lemmas}
\label{sec.symmetry lemmas}

In this section, we present two \emph{equivariant embedding lemmas} that allow one to convert between bilipschitz embeddings that interact with group actions in different ways.
These lemmas are illustrated by the following diagrams:
\[
\begin{aligned}
G\curvearrowright X&\xlongrightarrow{\phantom{G}} H\\[0pt]
&~~\,\RotImplies\text{\subsmall{~~ 1st EE Lemma~}}\\[0pt]
X&\xlongrightarrow{G} L^2(G,H)
\end{aligned}
\hspace{1in}
\begin{aligned}
X&\xlongrightarrow{G} Y\\[0pt]
&~~\,\RotImplies\text{\subsmall{~~ 2nd EE Lemma~}}\\[0pt]
X/G&\xlongrightarrow{\phantom{G}} Y/G
\end{aligned}
\]

The first lemma concerns a situation in which a compact group $G$ acts isometrically on a metric space $X$, and we have a bilipschitz embedding of $X$ into a Hilbert space $H$ which need not respect the group action. We show that one can convert such an embedding into a $G$-equivariant bilipschitz embedding $X\to L^2(G,H)$ without increasing the distortion. The following proposition clarifies notation for general locally compact groups, and its proof can be found in \cref{app.bochner}.

\begin{proposition}
    \label{prop.hilbert L2}
    Let $H$ be a Hilbert space and $G$ a locally compact group equipped with a right-invariant Haar measure $\mu$ on its Borel $\sigma$-algebra. Define $L^2(G,H)$ to be the space of Borel-measurable functions $f\colon G\to H$ such that
    \[\int_G \|f(g)\|_H^2\, d\mu(g) < \infty,\]
    with functions identified if they agree $\mu$-almost everywhere. Then the following hold:
    \begin{propenum}
    \item $L^2(G,H)$ is a Hilbert space under the inner product
    \[\langle f_1,f_2\rangle_{L^2(G,H)} := \int_G\big\langle f_1(g),f_2(g)\big\rangle_H\, d\mu(g).\]
    \item The group action of $G$ on $L^2(G,H)$ defined by
    \[(g\cdot f)(h) = f(hg)\]
    induces a unitary representation $\rho\colon G\to \U(L^2(G,H))$ that is strongly continuous, meaning $\rho$ is continuous with respect to the strong operator topology on $\U(L^2(G,H))$.
    \end{propenum}
\end{proposition}

With the above notational setup, we specialize to the setting of compact groups and enunciate our first lemma.

\begin{lemma}[first equivariant embedding lemma]
\label{lem.enforcing symmetry}
Let $X$ be a metric space equipped with an isometric action by a compact group $G$, and let $\phi\colon X\to H$ be a map into a Hilbert space. Equip $G$ with its unique normalized bi-invariant Haar measure, and define $L^2(G,H)$ as in \cref{prop.hilbert L2}.
Then there exists a map $\psi\colon X\to L^2(G,H)$ such that
    \begin{lemenum}
        \item $\psi$ is $G$-equivariant, i.e., $\psi\circ g = g\circ \psi$ for each $g\in G$, and
        \item $\kappa(\psi)\leq \kappa(\phi)$.
    \end{lemenum}
\end{lemma}
\begin{proof}
The result is trivial if $\kappa(\phi)=\infty$, since then we may take $\psi\equiv 0$.
Thus, we may assume $\kappa(\phi) < \infty$. 
In particular, $\phi$ is continuous. On account of $G$'s action on $X$, each $g\in G$ can be interpreted as a function $g\colon X\to X$, and with this notation, we define the embedding $\psi\colon X\to L^2(G,H)$ by
\begin{align*}
\psi(x)(g) = \phi(g(x)),
\end{align*}
that is, $\psi(x) = \phi\circ \ev_x$, where $\ev_x(g) := g(x)$ denotes the evaluation map. 
The map $\psi$ is well defined with the codomain $L^2(G,H)$, as it maps each $x\in X$ to a continuous (and hence bounded) function.
        
For (a), i.e., the $G$-equivariance of $\psi$, let $x\in X$ and $g,h\in G$.
Then
    \begin{align*}
        \psi(gx)(h) = \phi\circ \ev_{gx}(h) = \phi(hgx) = \phi\circ \ev_x(hg) = \psi(x)(hg) = (g\cdot (\psi(x)))(h),
    \end{align*}
and so $\psi(gx) = g\cdot (\psi(x))$.
    
For (b), i.e., $\kappa(\psi)\leq\kappa(\phi)$, note that
    \begin{align*}
        \|\psi(x) - \psi(y)\|_{L^2(G,H)}^2 
        & = \int_G\|\phi\circ \ev_x(h) - \phi\circ \ev_y(h)\|_H^2\,d\mu(h)
        \\ & = \int_G\|\phi\circ h(x) - \phi\circ h(y)\|_H^2\,d\mu(h).
    \end{align*}
Let $\alpha$ and $\beta$ denote the optimal lower and upper Lipschitz bounds for $\phi$, respectively. Applying the lower Lipschitz inequality for $\phi$ to the integrand, we obtain
    \begin{align*}
        \|\psi(x) - \psi(y)\|_{L^2(G,H)}^2 \geq \int_G\alpha^2 d_X^2(hx,hy)\,d\mu(h)
        = \alpha^2 d_X^2(x,y)\int_Gd\mu(h)
        = \alpha^2 d_X^2(x,y),
    \end{align*}
and we may similarly obtain an upper bound of $\beta^2 d_X^2(x,y)$. 
Taking square roots, we get a two-sided Lipschitz inequality:
    \begin{align*}
        \alpha \cdot d_X(x,y)\leq \|\psi(x) - \psi(y)\|_{L^2(G,H)} \leq \beta \cdot d_X(x,y),
    \end{align*}
    and hence a distortion bound of $\kappa(\psi) \leq \beta/\alpha = \kappa(\phi)$.
\end{proof}

Our second lemma allows one to descend an equivariant bilipschitz map to the quotient without increasing distortion.
(In what follows, the assumption that orbits are closed may be dropped if \cref{def.distortion} is generalized to pseudometric spaces.)

\begin{lemma}[second equivariant embedding lemma]\label{lem:quotient-distortion}
Let $G$ be an abstract group acting by isometries on metric spaces $X$ and $Y$ in such a way that every $G$-orbit in each space is closed. 
Let $f\colon X\to Y$ be a $G$-equivariant map with optimal lower and upper Lipschitz bounds $\alpha$ and $\beta$, respectively. 
Then the induced map $f_{/G}\colon X/G\to Y/G$ has lower Lipschitz bound at least $\alpha$ and upper Lipschitz bound at most $\beta$. 
In particular, $\kappa(f_{/G})\leq \kappa(f)$.
\end{lemma}

\begin{proof}
Let $x,x'\in X$ and $g\in G$ be arbitrary. 
Using the Lipschitz constants of $f$, we write
\[
\alpha\cdot d_X(gx,x')
\leq d_Y\big(f(gx),f(x')\big) 
\leq \beta\cdot d_X(gx,x').
\]
Using the equivariance of $f$ and minimizing all three expressions over the choice of $g$ gives
\[
\alpha\cdot\inf_{g\in G} d_X(gx,x') 
\leq\inf_{g\in G} d_Y\big(gf(x),f(x')\big) 
\leq \beta\cdot \inf_{g\in G} d_X(gx,x') ,
\]
so that
\[
\alpha\cdot d_{X/G}(Gx,Gx')
\leq d_{Y/G}\big(Gf(x),Gf(x')\big)
\leq \beta\cdot d_{X/G}(Gx,Gx'),
\]
and finally,
\[
\alpha\cdot d_{X/G}(Gx,Gx')
\leq d_{Y/G}\big(f_{/G}(Gx),f_{/G}(Gx')\big)
\leq \beta\cdot d_{X/G}(Gx,Gx').
\tag*{\qedhere}
\]
\end{proof}

\subsection{The Euclidean contortion of a finite group}
\label{sec.23-groups}
In this section, we study a new quantity that we assign to any abstract group of finite order.

\begin{definition}\label{def.upsilon}
The \textbf{(Euclidean) contortion} of a finite group $G$ is given by
\[
\Upsilon(G) 
:= \sup_{V\in \operatorname{Rep}(G)}c_2(V/G),
\]
where $\operatorname{Rep}(G)$ denotes the set of (equivalence classes of) all finite-dimensional orthogonal representations of $G$.
\end{definition}

Every finite group has finite contortion as a consequence of \eqref{eq.max filter bank bound}. 
The contortion of a group allows one to estimate the Euclidean distortion of the quotient of \textit{any} metric space by \textit{any} isometric action of that group.

\begin{theorem}\label{thm.upsilon-bound}
If a finite group $G$ acts by isometries on a metric space $X$, then
\[
c_2(X/G) 
\leq \Upsilon(G)\cdot c_2(X).
\]
\end{theorem}

\begin{proof}
By Proposition~31 in~\cite{CahillIM:24} (See \cref{prop.finitely-determined}), it suffices to prove
\[
c_2(B) 
\leq \Upsilon(G)\cdot c_2(X)
\]
for an arbitrary finite sub-metric space $B\subseteq X/G$. 
Let $A\subseteq X$ be the (finite) preimage of $B$ through the quotient map, so that $G$ acts on $A$ by isometries and $A/G = B$. 
Now take any $\varepsilon > 0$, and fix a map $f\colon A \to H$ into a Hilbert space with distortion $\kappa(f) < c_2(A) + \varepsilon$. 
By the first equivariant embedding lemma (\cref{lem.enforcing symmetry}), we can select $f$ to be equivariant with respect to $G$, with $H$ being a finite-dimensional orthogonal representation of $G$.
    
By the second equivariant embedding lemma (\cref{lem:quotient-distortion}), the map $f_{/G}\colon B \to H/G$ has distortion at most $\kappa(f)<c_2(A) + \varepsilon$. 
Next, by the definition of $\Upsilon(G)$, there exists a map $h\colon H/G\to H'$ into a Hilbert space such that $\kappa(h)\leq \Upsilon(G)+\varepsilon$. 
Then the composition
\[
B 
\longto[f] H/G
\longto[h] H'
\]
has distortion at most $(c_2(A) + \varepsilon)(\Upsilon(G)+\varepsilon)$. 
Since $\varepsilon>0$ was arbitrary, it follows that
\[
c_2(B) 
\leq \Upsilon(G)\cdot c_2(A) 
\leq \Upsilon(G)\cdot c_2(X).
\tag*{\qedhere}
\]
\end{proof}

Since smaller groups have easier representation theories, we can determine their contortions exactly.
(The contortion of the order-$2$ group will be particularly relevant in \cref{sec.applications}.)

\begin{lemma}\label{lem.upsilon compute}
The Euclidean contortion of a group $G$ of order at most $3$ is given by
\[
\Upsilon(G) 
= \begin{cases}
        1 & \text{ if } |G|=1,\\
        \sqrt 2 & \text{ if }|G|=2,\\
        3/2 & \text{ if }|G|=3.
\end{cases}
\]
\end{lemma}

\begin{proof}
The case $|G|=1$ is immediate. 
In the other cases, $G$ has two irreducible representations over $\mathbb R$: the trivial representation, and a real representation that underlies complex multiplication by $\zeta := e^{2\pi i/r}$. 
When $|G|=2$, this nontrivial representation is the $1$-dimensional sign representation, while when $|G|=3$, it is the $2$-dimensional representation in which a generator acts as an order-$3$ rotation.
Then every $V\in \operatorname{Rep}(G)$ decomposes as $V=V_1\oplus V_\zeta$, where $V_1$ is the trivial component and $V_\zeta$ is a direct sum of copies of the nontrivial irreducible representation. 
Since $G$ fixes $V_1$, Lemma~39 in~\cite{CahillIM:24} (see \cref{prop.max-product}) gives
\[c_2(V/G) 
= \max\{c_2(V_1),c_2(V_\zeta/G)\} 
= c_2(V_\zeta/G).
\]
Meanwhile, Corollary~36 in~\cite{CahillIM:24} (see \cref{prop.known-distortions}) gives $c_2(V_\zeta/G) = \sqrt 2$ in the case $|G|=2$, whereas \cref{thm.root unity} in \cref{sec.applications} implies $c_2(V_\zeta/G) = 3/2$ in the case $|G|=3$.
\end{proof}




Next, we observe that contortion interacts nicely with normal subgroups and quotient groups.
\begin{theorem}\label{thm.upsilon subnormal}
Let $G$ be a finite group with normal subgroup $N$. 
Then
\[
\Upsilon(G/N)\leq \Upsilon(G)\leq \Upsilon(G/N)\cdot \Upsilon(N).
\]
Consequently, if $G$ admits a subnormal series
$\{1\}
=G_0
\lhd G_1
\lhd \dots 
\lhd G_r 
= G$,
then
\[
\Upsilon(G)
\leq \prod_{i=0}^{r-1} \Upsilon(G_{i+1}/G_i).
\]
\end{theorem}

\begin{proof}
Consider the quotient map $\pi\colon G\to G/N$. Since every orthogonal representation $\rho$ of $G/N$ determines an orthogonal representation $\rho\circ\pi$ of $G$, the first inequality follows.

We now focus on the second inequality. Fix $V\in \operatorname{Rep}(G)$, and note that since $N$ is normal in $G$, the action of $G$ on $V$ descends to an isometric action of $G/N$ on $V/N$, and $(V/N)/(G/N)$ is isometric to $V/G$. 
Hence, \cref{thm.upsilon-bound} gives
\[
c_2(V/G) 
= c_2\big((V/N)/(G/N)\big) 
\leq \Upsilon(G/N)\cdot c_2(V/N) 
\leq \Upsilon(G/N)\cdot\Upsilon(N).
\]
The rest of the lemma follows from induction on the length of a subnormal series for $G$. To elaborate, the base case of $r=0$ is trivial since $\Upsilon(\{1\})= 1$ and the empty product has value one. Next, we use right hand inequality in the above result to get $\Upsilon(G_r)\leq \Upsilon(G_r/G_{r-1})\cdot \Upsilon(G_{r-1})$, and we apply the induction hypothesis on $G= G_{r-1}$.
\end{proof}

As an example application, one may combine Burnside's $p^aq^b$ theorem~\cite{Burnside:04} with \cref{lem.upsilon compute,thm.upsilon subnormal} to obtain a contortion bound for any group $G$ of order $2^a3^b$:
\[
\Upsilon(G)
\leq\big(\sqrt{2}\big)^a \big(\tfrac{3}{2}\big)^b
\leq\sqrt{|G|}.
\]
Notably, this implies that both groups of order~$4$ have contortion at most~$2$.
These are the smallest groups for which we do not have exact contortions.
Considering $\Upsilon(C_4)$ is at least the distortion of $\mathbb{R}^2$ modulo the order-$4$ rotation group, Corollary~38 in~\cite{CahillIM:24} (see \cref{prop.known-distortions}) gives
\[
\textstyle 2\sqrt{2-\sqrt{2}}
\leq\Upsilon(C_4)
\leq 2.
\]
Meanwhile, we get
\[
\sqrt{2}
\leq \Upsilon(C_2\times C_2)
\leq 2
\]
as a consequence of \cref{lem.upsilon compute,thm.upsilon subnormal}, applied to the normal subgroup $C_2 \lhd (C_2\times C_2)$.

In addition to finding exact contortions of small groups, it would be interesting to better understand the contortions of large groups.
As a consequence of results in \cref{sec.wednesday theorem}, there are groups with arbitrarily large contortion.
Specifically, $\Upsilon(S_n)\to \infty$ as $n\to \infty$.
One can still ask how contortion scales with the size of the group.
For instance, what is the smallest $\alpha$ such that for every $\varepsilon>0$, there exists a constant $c_\varepsilon$ such that $\Upsilon(G)\leq c_\varepsilon|G|^{\alpha+\varepsilon}$ for every finite group $G$?
(The bound~\eqref{eq.max filter bank bound} implies $\alpha\leq\frac{5}{2}$.)
Also, for each $n$, which groups exhibit the maximum contortion among all groups of order at most $n$?

\subsection{Quotient--orbit embeddings}
\label{sec.reflection like}

Let $X$ be a metric space equipped with an isometric action by a compact group $G$. Intuitively, a bilipschitz embedding of the quotient space $\phi\colon X/G\to\mathbb{R}^m$ can serve as a stepping stone towards a bilipschitz embedding of $X$ itself.
Since $\phi$ already bilipschitzly separates the $G$-orbits of $X$, if we find another map $\psi\colon X\to\mathbb{R}^n$ that bilipschitzly separates points within each $G$-orbit, then we might expect the map $x\mapsto(\phi([x]),\psi(x))$ to bilipschitzly embed $X$ into $\mathbb{R}^{m+n}$.
In this section, we show that this actually works under suitable conditions on $\psi$, and we call the resulting map $X\to\mathbb{R}^{m+n}$ a \textbf{quotient--orbit embedding}.
Throughout \cref{sec.applications}, we use this approach to find bilipschitz embeddings of $X:=V/N$ with $N\leq\operatorname{Isom}(V)$ in cases where we already know how to bilipschitzly embed $V/K$ for some intermediate group $N\unlhd K\leq\operatorname{Isom}(V)$.
(Notably, if we put $G:=K/N$, then $V/K=X/G$.)

In what follows, we first present general conditions on $\psi$ that allow us to quantitatively control the bilipschitz bounds of $x\mapsto(\phi([x]),\psi(x))$.
Our bounds are stronger when $|G|=2$, which is our most common use case in \cref{sec.applications}.
As an important instance of this case, we show in \cref{corr.glued space} that
\[
c_2(X)^2
\leq 2\cdot c_2(Y)^2 +2
\]
when $X$ is the metric space formed by gluing two copies of $Y$ along a closed subset of $Y$ (and $G$ acts by swapping these copies of $Y\cong X/G$ in $X$).

\subsubsection{Main result}
To obtain a bilipschitz quotient--orbit embedding $x\mapsto(\phi([x]),\psi(x))$ as described above, it suffices for $\psi$ to satisfy the following technical conditions.
\begin{definition}
\label{def.align preserve}
Let $X$ be a metric space equipped with an isometric action by a compact group $G$, and let $\psi\colon X\to H$ be a map into a Hilbert space carrying a unitary action by $G$.
\begin{itemize}
\item[(a)]
We say $\psi$ is \textbf{orbit expanding} if
for every $x\in X$ and $g\in G$, it holds that
\[
\|\psi(x) - g\cdot \psi(x)\|
\geq d_X(x,gx).
\]
\item[(b)]
We say $\psi$ is \textbf{alignment preserving} if for every $x,y\in X$, there exists $g\in G$ that simultaneously minimizes $d_X(x,gy)$ and $\| \psi(x) - g\cdot \psi(y)\|$.
\end{itemize}
\end{definition}

While our theory does not require it, the above properties are perhaps most intuitive when $\psi$ is $G$-equivariant.
(This is the case in all our applications in \cref{sec.applications}.)
As a simple example, if $X=H$ is a Hilbert space and $G\leq \U(H)$ is compact and fixes a closed subspace $S\leq H$, then the orthogonal projection $\psi$ onto $S^\perp$ is orbit expanding and alignment preserving (and $G$-equivariant).

We now state the main result of this section:
\begin{theorem}
    \label{thm.align preserve}
    Let $X,G,H,\psi$ be as in \cref{def.align preserve}, with $\psi$ being $\gamma$-Lipschitz in addition to orbit expanding and alignment preserving. Let $\phi\colon X/G\to H'$ be a map into a Hilbert space with optimal bilipschitz bounds $\alpha_\phi,\beta_\phi\in[0,\infty]$, and define
    \[c:=\begin{cases}
        1 &\text{if $G=\langle \sigma\rangle$ with $\sigma v=-v$ for $v\in H$ and $\dim(H) = 1$,}\\
        \sqrt{2} &\text{if $G=\langle \sigma\rangle$ with $\sigma v=-v$ for $v\in H$ and $\dim(H) > 1$,}\\
        2 &\text{otherwise.}
    \end{cases}\]
    Then the map $\Psi\colon X\to H'\times H$ defined by $\Psi(x) = \big(\,\phi([x]),\ c\cdot \alpha_\phi\cdot \psi(x)\,\big)$ satisfies 
\[\kappa(\Psi)\leq\sqrt{2}\cdot\sqrt{\kappa(\phi)^2 + (c\gamma)^2}.\]
In particular, we have $c_2(X)\leq\sqrt{2}\cdot\sqrt{c_2(X/G)^2 + (c\gamma)^2}$.
\end{theorem}
\begin{proof}
The bound is trivial if $\kappa(\phi)=\infty$ or $\gamma = \infty$, so assume $\alpha_\phi,\beta_\phi\in(0,\infty)$ and $\gamma <\infty$. 
It suffices to prove that for all $x,y\in X$,
\[\frac{\alpha_\phi}{\sqrt 2}\cdot d_X(x,y) \leq \big\|\Psi(x)-\Psi(y)\big\|\leq \sqrt{\beta_\phi^2+\left(\alpha_\phi c\gamma\right)^2}\cdot d_X(x,y).\]
The upper bound follows immediately from concatenation and the fact that the orbit map $[\cdot]\colon X\to X/G$ is $1$-Lipschitz.
For the lower bound, fix $x,y\in X$, and assume without loss of generality that $\|\psi(x)\|\leq\|\psi(y)\|$.
Since $\psi$ is alignment preserving, there exists $g\in G$ minimizing both $d_X(x,gy)$ and $\|\psi(x)-g\cdot\psi(y)\|$. 
We claim that
\[
d_X(x,gx)
\leq c \cdot \| \psi(x) - \psi(y) \|.
\]
Assuming the claim, the result follows from the following chain of inequalities:
\[
    \begin{aligned}
    (\alpha_\phi^2/2)\cdot d_X(x,y)^2 &= (\alpha_\phi^2/2)\cdot d_X(gx,gy)^2&& \text{($g$ is isometric)}\\
    &\leq (\alpha_\phi^2/2) \cdot\big(d_X(x,gy)+d_X(x,gx)\big)^2 && \text{(triangle inequality)}\\
    &\leq \alpha^2_\phi\cdot d_{X/G}([x],[y])^2 + \alpha^2_\phi\cdot d_X(x,gx)^2 &&\text{(AM--QM inequality)}\\
    &\leq \|\phi([x])-\phi([y])\|^2 + (c\cdot\alpha_\phi)^2\cdot \|\psi(x)-\psi(y)\|^2 && \text{(def.\ of $\alpha_\phi$; claim)}\\
    &= \|\Psi(x)-\Psi(y)\|^2 && \text{(def.\ of $\Psi$).}
    \end{aligned}
    \]
It remains to prove the claim. 
In general, we have
\[
\begin{aligned}
    d_X(x,gx)
&\leq \| \psi(x) - g\cdot \psi(x) \| &&\text{($\psi$ is orbit expanding)}\\
&\leq \| \psi(x) - g\cdot \psi(y) \| + \| g\cdot \psi(x) - g \cdot \psi(y) \|&& \text{(triangle inequality)}\\
&= \| \psi(x) - g \cdot \psi(y) \| + \| \psi(x) - \psi(y) \| && \text{($g\in \U(H)$)}\\
&\leq 2 \| \psi(x) - \psi(y) \| && \text{(choice of $g$).}
\end{aligned}
\]
Now suppose $G = \langle\sigma\rangle$ with $\sigma v=-v$ for $v\in H$. 
Since $\psi$ is orbit expanding, we get that $\langle \sigma^2\rangle$ acts trivially on $X$, and the $G$-orbit of each $x\in X$ is given by $\{x,\sigma x\}$. 
If $g \in \langle \sigma^2\rangle$, then the claim is immediate, as $d_X(x,gx) = 0 \leq c \|\psi(x)-\psi(y)\|$.
Otherwise, $g\in \sigma\cdot \langle\sigma^2\rangle$, and
\[
\|\psi(x)+\psi(y)\|
=\|\psi(x)-g\cdot \psi(y)\|
\leq\|\psi(x)-\psi(y)\|,
\]
so $\langle \psi(x),\psi(y)\rangle\leq0$.
In the special case $\dim(H)=1$, this further implies that
\[
\langle \psi(x),\psi(y)\rangle
= - \|\psi(x)\|\|\psi(y)\|.
\]
Along with the assumption $\|\psi(x)\|\leq\|\psi(y)\|$, this yields
\begin{align*}
\|\psi(x)-\psi(y)\|^2
&=\|\psi(x)\|^2+\|\psi(y)\|^2-2\langle\psi(x),\psi(y)\rangle\\
&\geq\|\psi(x)\|^2+\|\psi(y)\|^2+2\|\psi(x)\|\|\psi(y)\|\cdot\mathbf{1}_{\{\dim H=1\}}\\
&\geq\left\{\begin{array}{cl}
4\|\psi(x)\|^2 &\text{if }\dim H=1\\
2\|\psi(x)\|^2 &\text{if }\dim H>1
\end{array}\right.\\
&=\frac{4}{c^2}\cdot\|\psi(x)\|^2.
\end{align*}
Combining this with the orbit expanding property, we get
\[
d_X(x,gx)
=d_X(x,\sigma x)
\leq \|\psi(x)-\sigma\cdot\psi(x)\|
=2\|\psi(x)\|
\leq c\|\psi(x)-\psi(y)\|.
\tag*{\qedhere}
\]
\end{proof}

\subsubsection{Application to glued spaces}
\label{sec.pushout}

\begin{figure}
    \centering
    \includegraphics[width=0.25\linewidth]{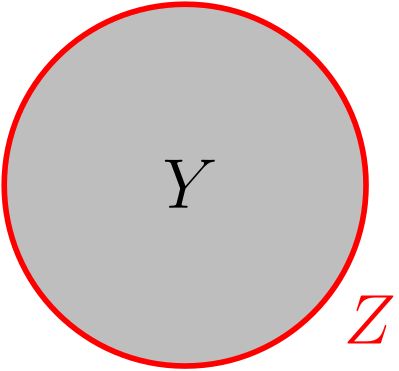}%
    \qquad\qquad%
    \includegraphics[width=.4\linewidth]{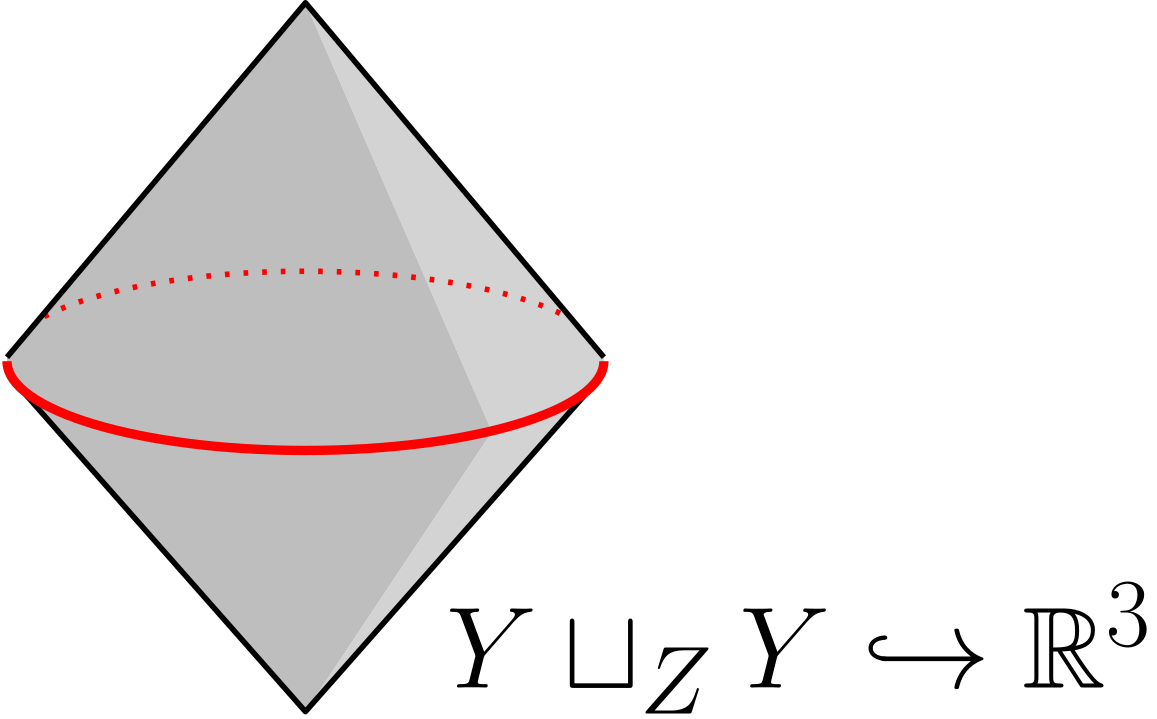}
    \caption{An illustration of a glued space and a corresponding embedding of the type described in Section~\ref{sec.pushout}. \textbf{(left)} We take $Y$ to be the unit disk with the usual metric, and we take $Z$ to be its boundary. \textbf{(right)} The resulting glued space $Y\sqcup_Z Y$ is topologically a sphere, though the metric agrees with the flat metric on $Y$ when restricted to the northern or southern hemispheres. This figure depicts $Y\sqcup_Z Y$ embedded into 3-dimensional space according to the embedding described in \cref{corr.glued space}. That is, for each embedded point, the horizontal coordinates are given by the usual embedding of the unit disk into $\mathbb R^2$, while the vertical coordinate is proportional to its distance in $X$ from the boundary $Z$.}
    \label{fig:pushout}
\end{figure}

Let $(Y,d_Y)$ be a metric space, and let $Z\subseteq Y$ be closed. 
One can construct a new metric space by gluing two copies of $Y$ along $Z$. 
Formally, define $Y\sqcup_Z Y := (Y\times \{\pm 1\})/{\sim_Z}$, where $\sim_Z$ identifies $(z,1)$ with $(z,-1)$ for each $z\in Z$. 
Equip this space with the metric
\begin{align*}
    d_{Y\sqcup_Z Y}\big((y,\delta),(y',\delta')\big) = \left\{\begin{array}{cl}
        d_Y(y,y') & \text{if }\delta = \delta'\\
        \displaystyle\inf_{z\in Z} \big(d_Y(y,z)+d_Y(z,y')\big) & \text{if }\delta = -\delta'.
    \end{array}\right.
\end{align*}
The resulting space $(Y\sqcup_Z Y,d_{Y\sqcup_Z Y})$, true to our notation, is the pushout
\[
\begin{tikzcd}
    Z \ar[r,hook] \ar[d,hook] & Y \ar[d,dashed]\\
    Y \ar[r,dashed] & Y\sqcup_Z Y
\end{tikzcd}
\]
in the category of metric spaces with 1-Lipschitz morphisms. 
In what follows, we present a bilipschitz embedding of all such glued spaces, and we apply \cref{lem.upsilon compute,thm.align preserve} to estimate their Euclidean distortions. 
See Figure~\ref{fig:pushout} for an illustration of a glued space and the proposed embedding.

\begin{corollary}
\label{corr.glued space}
Let $Y$, $Z$, and $X:=Y\sqcup_Z Y$ be as above, let $\phi\colon Y\to H'$ be a map into a Hilbert space, and define $\Psi\colon X\to H'\times \mathbb R$ by
\[
\Psi(y,\delta) 
:= \big(\,\phi(y),\,\delta\inf_{z\in Z}d_Y(y,z)\,\big).
\]
Then $\kappa(\Psi) 
\leq \sqrt 2 \cdot\sqrt{\kappa(\phi)^2 + 1}$.
Moreover, we obtain the two-sided bound:
\[
c_2(Y)
\leq c_2(X)
\leq \sqrt{2}\cdot\sqrt{c_2(Y)^2 + 1}.\]
\end{corollary}
\begin{proof}
The lower bound on $c_2(X)$ follows from observing that $Y$ is a submetric space of $X$.

It remains to prove our upper bound on $\kappa(\Psi)$. Consider the group $G:=\{\operatorname{id},\sigma\}$ with $\sigma$ acting isometrically on $X$ by 
\[\sigma(y,\delta):=(y,-\delta).\]
We first claim that $X/G$ is isomorphic to $Y$ as a metric space.
Indeed, for every $y,y'\in Y$ and $\delta\in\{\pm1\}$, the triangle inequality implies
\[
d_X\big((y,\delta),(y',\delta)\big)
=d_Y(y,y')
\leq \inf_{z\in Z} \big(d_Y(y,z)+d_Y(z,y')\big)
= d_X\big((y,\delta),(y',-\delta)\big),
\]
and so $d_{X/G}([(y,\delta)],[(y',\delta')])=d_Y(y,y')$.
Letting $\sigma$ act on $\mathbb{R}$ by $\sigma v=-v$, then by \cref{thm.align preserve}, it suffices to verify that the map $\psi\colon X\to \mathbb{R}$ defined by
\[
\psi(y,\delta)
:=\delta\inf_{z\in Z}d_Y(y,z)
\]
is $1$-Lipschitz, orbit expanding, and alignment preserving.
In what follows, we write $d_Z(y):=\inf_{z\in Z}d(y,z)$ for convenience.

For $1$-Lipschitz, we aim to show that
\[
|\delta d_Z(y)-\delta' d_Z(y')|
\leq d_X\big((y,\delta),(y',\delta')\big).
\]
On account of the piecewise definition of $d_X$, we proceed in cases.
In the case where $\delta=\delta'$, our task is to show
\begin{equation}
\label{eq.case 1 of lipschitz}
|d_Z(y)-d_Z(y')|
\leq d_Y(y,y').
\end{equation}
Without loss of generality, assume $d_Z(y)\geq d_Z(y')$.
Given any $\varepsilon>0$, we may select $z\in Z$ in such a way that $d_Y(y',z)\leq d_Z(y')+\varepsilon$.
Thus, the triangle inequality gives
\[
|d_Z(y)-d_Z(y')|
=d_Z(y)-d_Z(y')
\leq d_Y(y,z)-d_Y(y',z)+\varepsilon
\leq d_Y(y,y')+\varepsilon.
\]
Since $\varepsilon>0$ was arbitrary, the validity of~\eqref{eq.case 1 of lipschitz} follows.
In the remaining case where $\delta=-\delta'$, our task is to show
\[
|d_Z(y)+d_Z(y')|
\leq \inf_{z\in Z} \big(d_Y(y,z)+d_Y(z,y')\big).
\]
Considering $d_Z(y)+d_Z(y')\geq0$, this is equivalent to the bound
\[
\inf_{z\in Z}d_Y(y,z)+\inf_{z\in Z}d_Y(y',z)
\leq\inf_{z\in Z} \big(d_Y(y,z)+d_Y(y',z)\big),
\]
which immediately follows from the definition of infimum.

For orbit expanding, take any $x:=(y,\delta)\in X$ and $g\in G$.
If $g=\operatorname{id}$, we trivially have $\|\psi(x)-g\cdot\psi(x)\|\geq 0=d_X(x,gx)$.
Otherwise, $g=\sigma$, in which case
\[
\|\psi(x)-g\cdot\psi(x)\|
=2\|\psi(x)\|
=2d_Z(y).
\]
Given any $\varepsilon>0$, we may select $z\in Z$ in such a way that $d_Y(y,z)\leq d_Z(y)+\frac{\varepsilon}{2}$, in which case
\[
\|\psi(x)-g\cdot\psi(x)\|
=2d_Z(y)
\geq 2d_Y(y,z)-\varepsilon
\geq d_X\big((y,\delta),(y,-\delta)\big)-\varepsilon
= d_X(x,gx)-\varepsilon.
\]
Since $\varepsilon>0$ was arbirary, it follows that $\|\psi(x)-g\cdot\psi(x)\|\geq d_X(x,gx)$.

Finally, for alignment preserving, take any $x:=(y,\delta)$ and $x':=(y',\delta')$ in $X$, and for convenience, denote 
\[
A:=\argmin_{g\in G} d_X(x,gx'),
\qquad
B:=\argmin_{g\in G} \|\psi(x)-g\cdot\psi(x')\|.
\]
As a consequence of the triangle inequality, $\delta=\delta'$ implies $\operatorname{id}\in A$ and $\delta=-\delta'$ implies $\sigma\in A$.
Additionally, $\delta=\delta'$ implies that $\psi(x)$ and $\psi(x')$ have the same sign (or at least one is zero), and so $\operatorname{id}\in B$.
On the other hand, $\delta\neq \delta'$ implies that they have opposite sign (or at least one is zero), and so $\sigma\in B$.
Either way, there is a member of $G$ that simultaneously minimizes $d_X(x,gx')$ and $\|\psi(x)-g\cdot\psi(x')\|$.
\end{proof}



\subsection{Local-to-global distortion bounds}
\label{sec.isotropy}

Recall that we can lower bound $c_2(X)$ by passing to a sub-metric space $Y$ of $X$:
\[
c_2(X)
\geq c_2(Y).
\]
In this section, we identify conditions under which we can instead pass to a metric space $\tilde{Y}$ that in some sense linearly approximates a small neighborhood $Y$ of a point in $X$.
In service to our applications in \cref{sec.applications}, we focus on a setting in which $X$ is the quotient of a Riemannian manifold by a suitable isometric action.

Let $M$ be a Riemannian manifold equipped with distance function $d\colon M\times M\to \mathbb R_{\geq 0}$, which induces the given topology on $M$ and restricts to the geodesic distance on each connected component; for its existence, see, e.g.,~\cite[Prob.~2.30]{Lee:18}. 
Suppose that, as a metric space, $(M,d)$ is equipped with a \textit{wandering} isometric action by a discrete group $G$. 
That is, for each $p\in M$, there exists an open neighborhood $U\subseteq M$ such that the set
\[
\{g\in G:g\cdot U\cap U
\neq \varnothing\}
\]
is finite. 
Then all $G$-orbits are closed, and the stabilizer subgroup 
\[
G_p
:= \{g\in G:g\cdot p=p\}
\]
is finite \cite[Lem.~13]{Kapovich:24}.
On each connected component of $M$, every element $g\in G$ restricts to a Riemannian isometry between connected components; see, for example,~\cite[Prob.~6.7]{Lee:18}. 
The differential
\[
d_p\colon G_p\to \Oname(T_pM),
\quad 
g\mapsto d_pg
\]
then defines a representation of $G_p$ on $T_pM$, known as the \textit{isotropy representation} at $p$.

\begin{figure}
    \centering
    \includegraphics[width=0.8\linewidth]{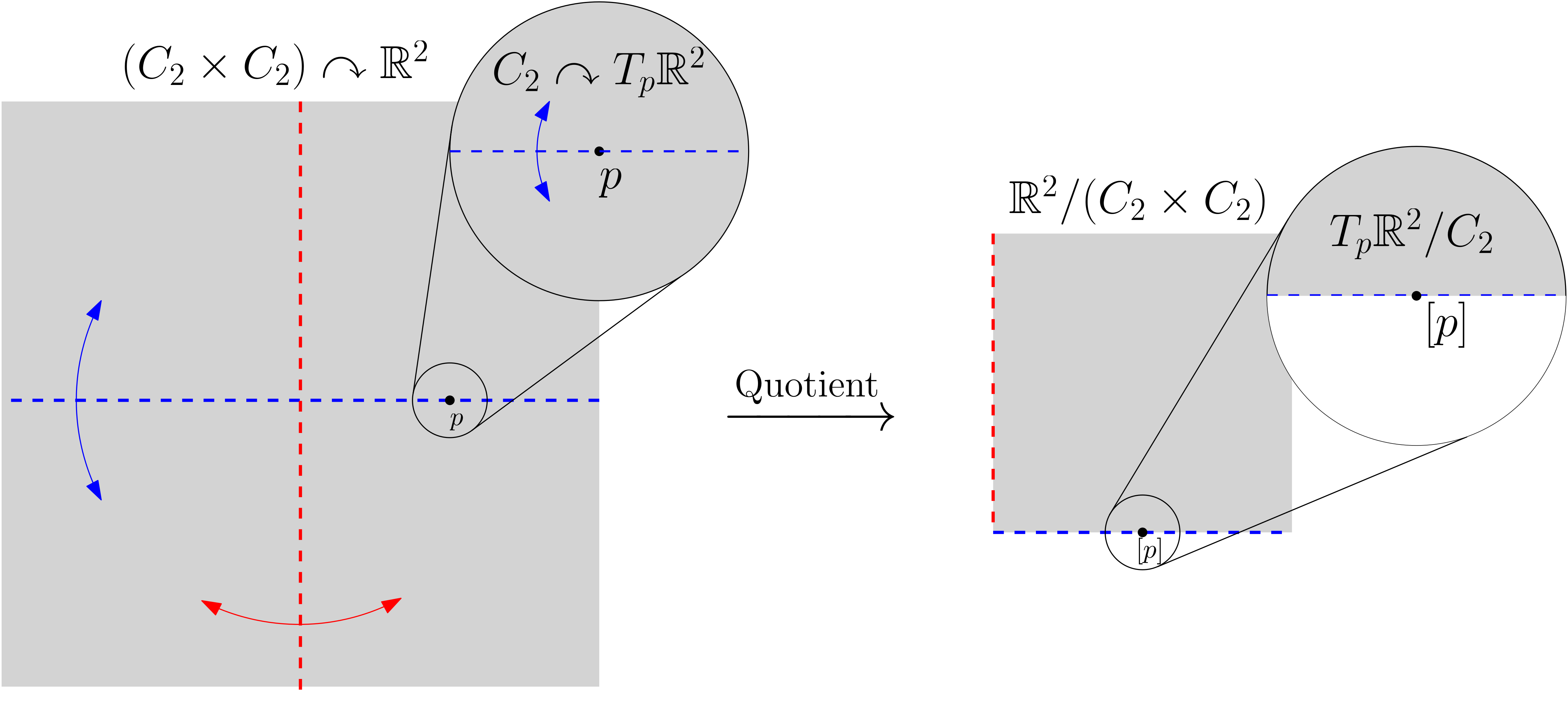}
    \caption{A depiction of the tangent space analysis described in Section~\ref{sec.isotropy}. \textbf{(left)} We consider the action of $C_2\times C_2$ on $\mathbb R^2$ generated by a pair of reflections. The axes of reflection are shown with dashed red and blue lines. Selecting a point $p$ on one of these axes, we note the stabilizer of $p$ is $C_2$. The image depicts the induced isotropy representation of $C_2$ on the tangent space $T_p\mathbb R^2$. \textbf{(right)} The quotient space $\mathbb R^2/(C_2\times C_2)$. The boundaries arising from each reflection are shown in dashed red and blue lines. We see that an infinitesimal neighborhood of $[p]$ looks like the quotient $T_p\mathbb R^2/C_2$ of the isotropy representation at $p$.
    \label{fig:isotropy}}
\end{figure}

One often thinks of the tangent space $T_pM$ as an ``infinitesimal neighborhood'' of $p$, and correspondingly, the quotient $T_pM/G_p$ as an ``infinitesimal neighborhood'' of  $[p]\in M/G$.
Figure~\ref{fig:isotropy} shows a simple example when $M=\mathbb R^2$.
While the containment of $T_pM/G_p$ in $M/G$ is purely figurative, the following theorem shows that a distortion inequality holds as if the containment were literal. 
It provides a local-to-global principle for bounding the distortion of the quotient space. 
We postpone the proof to the end of this subsection.

\begin{theorem}
    \label{thm.isotropy}
    Let $M$ and $G$ be as above.
    For each $p\in M$,
    $c_2(T_pM/G_p) \leq c_2(M/G)$.
\end{theorem}

In \cref{sec.tran klien}, we apply \cref{thm.isotropy} to obtain lower bounds on the Euclidean distortions of quotients of the plane by various wallpaper groups.
In what follows, we specialize \cref{thm.isotropy} to the special case where $M = W$ is an orthogonal representation of a finite group $G$. This is used in \cref{sec.wednesday theorem} to identify several parameterized families of important quotient spaces whose Euclidean distortions diverge to infinity.

\begin{corollary}
\label{thm.wednesday}
Let $W$ be a finite-dimensional real Hilbert space, and let $G\leq \O(W)$ be a finite group. 
Then for every subspace $V\leq W$, we have
\[
c_2(V^\perp/G_V) 
= c_2(W/G_V)
\leq c_2(W/G),
\]
where $G_V$ denotes the pointwise stabilizer of $V$ in $G$. 
Here, we abuse notation slightly by identifying $G_V\leq \O(V^\perp)$ with its restriction $G_V\leq \O(W)$ since $G_V$ acts trivially on $V$.
\end{corollary}

\begin{proof}
The first equality follows from Lemma~39 in~\cite{CahillIM:24} (see \cref{prop.max-product}) and the fact that $c_2(V)=1$:
\[
c_2(W/G_V) 
= c_2(V\times (V^\perp/G_V)) 
= \max\{c_2(V),c_2(V^\perp/G_V)\}
= c_2(V^\perp/G_V).
\]    
For the inequality $c_2(W/G_V)\leq c_2(W/G)$, the case $G=G_V$ is trivial, so we assume $G\neq G_V$. 
A generic point $p\in V$ satisfies $G_p=G_V$, since for each of the finitely many $g\in G\setminus G_V$, the fixed points in $V$ form a proper subspace of $V$.
To conclude, we apply \cref{thm.isotropy} at $p$, and the inequality follows since $T_pW = W$ and $d_pg = g$ for each $g\in G_p\leq \Oname(W)$.
\end{proof}

We conclude this subsection with a proof of \cref{thm.isotropy}.
\begin{proof}[Proof of \cref{thm.isotropy}]
Fix $p\in M$, and let $B\subseteq T_pM$ be an open ball centered at the origin. 
Define $U:= \exp_p(B)$, where $\exp_p$ denotes the Riemannian exponential map. 
We select $B$ with sufficiently small radius so that
    \begin{itemize}
        \item[] $U$ is geodesically convex,
        \item[] $0<\operatorname{diam}(U)<\frac{1}{3}\cdot\min_{g\in G\setminus G_p}d_M(p,gp)$, and
        \item[] $\exp_p|_B\colon B\to U$ is a diffeomorphism with inverse $\log_p\colon U\to B$.
    \end{itemize}
Since $G$ acts isometrically on $M$, we have the identity 
\[
\exp_p\circ \, d_pg 
= g\circ \exp_p
\]
for each $g\in G_p$. 
Thus, $U$ is $G_p$-invariant and $\exp_p|_B$ is $G_p$-equivariant. 
By the second equivariant embedding lemma (\cref{lem:quotient-distortion}), the induced map $\exp^G_p|_B\colon B/G_p \to U/G_p$ satisfies
\begin{equation}
\label{eq.isotropy inter}
c_2(B/G_p)
\leq \kappa(\exp^G_p|_B)\cdot c_2(U/G_p) 
\leq \kappa(\exp_p|_B)\cdot c_2(U/G_p).
\end{equation}
We claim that
\begin{enumerate}[label=(\roman*)]
    \item $c_2(B/G_p)= c_2(T_pM/G_p)$,
    \item $c_2(U/G_p) = c_2(GU/G)$, and
    \item $\kappa(\exp_p|_B)\to 1$ as the radius of $B$ shrinks to zero.
\end{enumerate}
Considering \eqref{eq.isotropy inter} and the fact that $GU/G$ is a sub-metric space of $M/G$, (i) and (ii) imply
\[c_2(T_pM/G_p) \leq \kappa(\exp_p|_B)\cdot c_2(GU/G) \leq \kappa(\exp_p|_B)\cdot c_2(M/G),\]
and then the theorem follows from (iii) by taking $\kappa(\exp_p|_B) \to 1$.

For (i), we apply Proposition~31 in~\cite{CahillIM:24} (see \cref{prop.finitely-determined}), which gives that $c_2(T_pM/G_p)$ is the supremum of $c_2(G_pF/G_p)$ over all finite $F\subseteq T_pM$. 
Since scaling a metric space does not change its distortion, we have $c_2(G_pF/G_p) = c_2(G_p(rF)/G_p)$ for every $r > 0$. 
Meanwhile, for small enough $r$, we get $G_p(rF)\subseteq B$, and so
\[c_2(G_pF/G_p)\leq c_2(B/G_p) \leq c_2(T_pM/G_p).\]
Taking the supremum over all finite $F\subseteq T_pM$ proves (i).

For (ii), fix $q,q'\in U$.
It suffices to show
\[
d_{U/G_p}(G_p\cdot q,G_p\cdot q') 
= d_{GU/G}(G\cdot q, G\cdot q').
\]
Since the right-hand side minimizes over a larger group, the inequality $\geq$ holds.
Suppose, for the sake of contradiction, that the inequality is strict, meaning there exists $g\in G\setminus G_p$ such that $d_M(gq,q') = d_{M/G}(G\cdot q,G\cdot q')$. Then
\[d_M(gq,q') \leq d_M(q,q')\leq \operatorname{diam}(U).\]
On the other hand, using the triangle inequality, the fact that $g$ is an isometry, the fact that $p,q,q'\in U$, and our assumption on $\operatorname{diam}(U)$, we find that
\begin{align*}
d_M(gq,q')
&\geq d_M(gp,p) - d_M(gp,gq) - d_M(q',p)\\
&= d_M(gp,p) - d_M(p,q) - d_M(q',p)\\
&\geq d_M(gp,p) - 2\operatorname{diam}(U)\\
&> \operatorname{diam}(U),
\end{align*}
a contradiction.

For (iii), fix $\varepsilon> 0$ and choose $\delta > 0$ so that $(1+\delta)^2 < 1+\varepsilon$. 
Since both $\exp_p|_B$ and its inverse $\log_p$ are smooth and satisfy $d_p(\exp_p|_B) = d_p(\log_p) = \Id_{T_pM}$, we can shrink the radius of $B$ so that the operator $2$-norms of the differentials satisfy
\[
\|d_b(\exp_p|_B)\|_2 
< 1+\delta
\qquad \text{and } \qquad
\|d_q(\log_p)\|_2
<1+\delta
\] 
for all $b\in B$ and $q\in U$. 
It follows that the upper Lipschitz bound of $\exp_p|_B$ is at most $1+\delta$, and its lower Lipschitz bound is at least $(1+\delta)^{-1}$, so
\[
\kappa(\exp_p|_B) 
\leq (1+\delta)^2 
< 1+\varepsilon.
\]
Since $\varepsilon>0$ was arbitrary, the claim follows.
\end{proof}

\subsection{Euclidean isometry groups with a subspace of translations}
\label{sec.euclidean tool}

For a finite-dimensional real inner product space $V$, we are interested in bilipschitzly embedding quotients of $V$ by various subgroups of $\E(V)$. 
To date, the vast majority of attention has been given to quotients by subgroups that consist of linear transformations, motivated in part by the assortment of tools available from algebraic invariant theory.
In this section, we relate this special case to a broader class of quotients.

The Euclidean group $\E(V)$ consists of the affine orthogonal transformations of $V$.
By identifying $V$ with its translation group, it follows that
\[
\E(V)
=\O(V)\ltimes V.
\]
In particular, the canonical homomorphism $\pi:\E(V)\rightarrow \O(V)$ has kernel $V$.

Suppose we wish to quotient $V$ by a subgroup $\Gamma\leq\E(V)$.
We will assume that the \textit{point group} $G:=\pi(\Gamma)\leq \O(V)$ is closed so that $\Gamma\leq\E(V)$ is also closed, and hence $V/\Gamma$ is a metric space. 
Notice that the kernel of the restriction of $\pi$ to $\Gamma$ is the \textit{translation subgroup} $T := \Gamma \cap V$, and so $G\cong\Gamma/T$.
It turns out that if $T$ is a linear subspace of $V$, then our quotient $V/\Gamma$ is isometrically isomorphic to a quotient by a linear group:
\[
V/\Gamma
\cong \frac{V/T}{\Gamma/T}
\cong T^\perp/G,
\]
where $T^\perp$ denotes the orthogonal complement of $T$ in $V$.
In \cref{sec.euclidean}, we use this reduction to estimate the Euclidean distortions of $(\mathbb{R}^r)^n/\E(r)$ and $(\mathbb{R}^r)^n/\operatorname{SE}(r)$.

\begin{theorem}\label{thm.euclidean-subgroup}
Let $\Gamma$, $G$, $T$ be as above, with $G\leq \O(V)$ closed and $T$ a subspace of $V$.
Then
\begin{thmenum}
\item $T$ and $T^\perp$ are $G$-stable, and
\item the inclusion $T^\perp \hookrightarrow V$ induces an isometry of  $T^\perp / G$ with $V/\Gamma$.
\end{thmenum}
\end{theorem}
\begin{proof}
We begin with (a). 
By the commutativity of $T$, the conjugation action of $\Gamma$ on $T$ factors through $\Gamma/T\cong G$. 
An arbitrary $\gamma\in \Gamma$ may be expressed as $tg$ with $t\in V$ and $g\in G$.
It follows from the commutativity of $V$ that the induced conjugation action $G\curvearrowright T$ coincides with the linear action of $G\leq \O(V)$ on $T\subseteq V$. 
Thus, $T$ is $G$-stable.
Since $G\leq\O(V)$, it follows that $T^\perp$ is also $G$-stable, meaning $G$ is a subgroup of $\O(T)\times \O(T^\perp)$.

Next, we prove (b). 
By composing the inclusion $T^\perp \hookrightarrow V$ with the projection $V\to V/T$, we obtain an isometry $T^\perp \rightarrow V/T$ that is $\O(T)\times \O(T^\perp)$-equivariant, and thus $G$-equivariant.
By the second equivariant embedding lemma (\cref{lem:quotient-distortion}), this then induces an isometry of $T^\perp/G$ with $(V/T)/G$.

To relate these metric spaces to $V/\Gamma$, we consider another group action.
Specifically, if we first quotient by the action $T\curvearrowright V$, the action $\Gamma\curvearrowright V$ induces an action $G\curvearrowright V/T$. 
This does not necessarily coincide with $G$'s linear action on $V/T$ above, but we claim that these actions are necessarily conjugate to each other by an isometry. 
This then implies that $V/\Gamma \cong (V/T)/(\Gamma/T)$ is isometric with the metric space $(V/T)/G\cong T^\perp / G$ discussed above.

To prove our claim, note that the action $G\curvearrowright V/T$ induced by $\Gamma\curvearrowright V$ is an action by isometries on the Euclidean space $V/T$.
Meanwhile, since $G\leq \O(V)$ is compact, it has a fixed point.
Given any Euclidean isometry with a fixed point, we may conjugate it by the translation that sends that fixed point to the origin to obtain an orthogonal isometry.
As such, this action $G\curvearrowright V/T$ is conjugate by an isometry to a linear action. 
By writing elements of $\Gamma$ in the form $tg$ with $t\in V$, $g\in G$, this linear action is necessarily the action $G\curvearrowright V/T$ discussed above. 
\end{proof}

In the above proof of (b), we established that there exists $p\in T^\perp$ such that every element of $\Gamma$ has the form $x\mapsto A(x-p)+p+c$, where $A\in G$ and $c\in T$. 
This shows that the homomorphism $\varphi\colon G\to \Gamma$ defined by $\varphi(A)(x) = A(x-p)+p$ splits the short exact sequence $0\to T\to\Gamma\to G\to 1$. 
In \cref{app.split}, we give alternative proofs of this fact using ideas from group cohomology.

As a final remark, note that the subspace assumption on $T$ is essential in \cref{thm.euclidean-subgroup}. 
If $T$ contains a lattice summand $\Lambda$, then we expect estimating $c_2(V/\Gamma)$ to be at least as difficult as estimating $c_2(V/\Lambda)$, which is already nontrivial; see \cite{HavivR:10, HeimendahlLVZ:22,VallentinM:23}. 
This raises a natural question: given $\Gamma\leq \E(V)$, with translation lattice $T=V\cap \Gamma$ and closed point group $G=\pi(\Gamma)\leq\O(V)$, can $c_2(V/\Gamma)$ be bounded in terms of $c_2(V/T)$ and $c_2(V/G)$? Does the answer depend on whether the extension $0\rightarrow T\rightarrow\Gamma\rightarrow G\rightarrow 1$ splits?

\section{Portfolio of applications}
\label{sec.applications}

In this part, we use the tools developed in the previous section to bound the Euclidean distortion of several important families of orbit spaces.

\subsection{Scalar actions of cyclic groups}
\label{sec.root unity}

In this section, we determine the Euclidean distortion of the metric space $\mathbb C^n / C_r$ for each $r \geq 2$, where the cyclic group $C_r$ acts on $\mathbb{C}^n$ via multiplication by $r$th roots of unity.

Considering $C_r$ is a subgroup of the unit circle $\mathbb{T}\leq\mathbb{C}^\times$, one might start by collecting invariants of this larger group.
To this end, Corollary~37 in~\cite{CahillIM:24} (see \cref{prop.known-distortions}) gives that $c_2(\mathbb{C}^n/\mathbb{T})$ is achieved by the map induced by the invariant $\phi\colon\mathbb{C}^n\to(\mathbb{C}^n)^{\otimes2}$ defined by
\[
\phi(u)
=\left\{\begin{array}{cl}
\displaystyle\frac{u\otimes\overline{u}}{\|u\|}&\text{if }u\neq0\\[16pt]
0&\text{else}.
\end{array}\right.
\]
It remains to find invariants to $C_r$ that bilipschitzly separate $C_r$-orbits within $\mathbb{T}$-orbits.
In fact, one may show that the invariant $\psi\colon\mathbb{C}^n\to(\mathbb{C}^n)^{\otimes r}$ defined by
\[
\psi(u)
=\left\{\begin{array}{cl}
\displaystyle\frac{u^{\otimes r}}{\|u\|^{r-1}}&\text{if }u\neq0\\[16pt]
0&\text{else}
\end{array}\right.
\]
induces a bilipschitz map $\psi_{/C_r}\colon\mathbb{C}^n/C_r\to(\mathbb{C}^n)^{\otimes r}$, though with sub-optimal distortion.

We aim to achieve an even better distortion by combining $\phi$ and $\psi$ as a quotient--orbit embedding of $X:=\mathbb{C}^n/C_r$ with $G:=\mathbb{T}/C_r$.
Consider the action of $\mathbb{T}/C_r$ on $(\mathbb{C}^n)^{\otimes r}$ defined by taking $g\cdot x^{\otimes r}:=(gx)^{\otimes r}$ and extending linearly.
Then since $\psi_{/C_r}$ is lower Lipschitz, a suitable multiple of $\psi_{/C_r}$ is orbit expanding with respect to this action.
One can also verify that any scalar multiple of $\psi_{/C_r}$ is alignment preserving with respect to this action.
Unfortunately, a direct application of \cref{thm.align preserve} in this setting is rather weak, producing a quotient--orbit embedding with distortion bound
\[
2\cdot\sqrt{1+2\kappa(\psi_{/C_r})^2},
\]
notably worse than the distortion of $\psi_{/C_r}$.
A more careful analysis reveals that a particular mixture of $\phi$ and $\psi$ delivers a quotient--orbit embedding that achieves the Euclidean distortion of $\mathbb{C}^n/C_r$.

\begin{theorem}
\label{thm.root unity}
Consider $\phi$ and $\psi$ as above, and define $F\colon \mathbb C^n/C_r\to (\mathbb C^n)^{\otimes 2}\times (\mathbb C^n)^{\otimes r}$ by
\[
F([u])
= \big(~
\cos(\tfrac{\pi}{2r})\phi(u),~
\sin(\tfrac{\pi}{2r})\psi(u)
~\big).
\]
Then for all $u,v\in\mathbb C^n$, it holds that
\[
d_{C_r}([u],[v])
\leq \|F([u])-F([v])\|
\leq r\sin(\tfrac{\pi}{2r})\cdot d_{C_r}([u],[v]).
\]
In particular, $c_2(\mathbb C^n/C_r)
= \kappa(F) 
= r\sin(\tfrac{\pi}{2r})$.
\end{theorem}

We note that the $n=1$ case coincides with the optimal embedding of $\mathbb{C}/C_r$ given in \cite{CahillIM:24}.

\begin{proof}[Proof of \cref{thm.root unity}]
Consider any $1$-dimensional complex subspace $V$ of $\mathbb{C}^n$.
Since $V$ is invariant under the action of $C_r$, it follows that $V/C_r$ is a sub-metric space of $\mathbb{C}^n/C_r$, and so
\[
\kappa(F)
\geq c_2(\mathbb{C}^n/C_r)
\geq c_2(V/C_r)
=r\sin(\tfrac{\pi}{2r}),
\]
where the last step is by Corollary~38 in~\cite{CahillIM:24} (see \cref{prop.known-distortions}).
It remains to show that the lower and upper Lipschitz bounds of $F$ are $1$ and $r\sin(\frac{\pi}{2r})$, respectively.

Note that $F$ maps the sphere $S$ in $\mathbb C^n$ to the sphere in $(\mathbb C^n)^{\otimes 2}\times (\mathbb C^n)^{\otimes r}$.
Furthermore, $F$ is \textit{positively homogeneous} in the sense that $F([ru])=rF([u])$ for every $r\geq0$ and $u\in\mathbb{C}^n$.
By Theorem~13 in~\cite{CahillIM:24}, it suffices to verify our bilipschitz bounds for $u,v\in S$.   
To this end, fix $u,v\in S$ and put $z:= u^* v$.
Then
\begin{align*}
\|F([u])-F([v])\|^2
&=\cos^2(\tfrac{\pi}{2r})\cdot\|\phi(u)-\phi(v)\|^2+\sin^2(\tfrac{\pi}{2r})\cdot\|\psi(u)-\psi(v)\|^2\\
&=\cos^2(\tfrac{\pi}{2r})\cdot\big(2 - 2|z|^2\big)+\sin^2(\tfrac{\pi}{2r})\cdot\big(2 - 2\Re(z^r)\big)\\
&=2-2\cdot\Big(\cos^2(\tfrac{\pi}{2r})\cdot|z|^2+\sin^2(\tfrac{\pi}{2r})\cdot\Re(z^r)\Big).
\end{align*}
Meanwhile, if we let $\omega$ denote a primitive $r$th root of unity, we have
\[
d_{C_r}([u],[v])^2 
= 2 - 2\max_{j\in[r]}\Re(\omega ^j z).
\]
If $z=0$, then both quantities above equal $2$, and so equality is achieved in our putative lower Lipschitz bound of $1$.
Thus, we may assume $z\neq 0$.
Take $t:= |z|\in (0,1]$ and $\theta:= \arg(z)$. By invariance under the actions of $C_r$ and complex conjugation, we may further assume $\theta \in [0,\frac{\pi}{r}]$ so that $d_{C_r}([u],[v])^2=2-2t\cos(\theta)$. 
    
Put $a:=\sin(\frac{\pi}{2r})$.
For the lower Lipschitz bound, we aim to show
\[
2-2t\cos(\theta)
\leq 2 - 2\Big( (1-a^2)t^2+a^2t^r\cos(r\theta)\Big),
\]
which rearranges to
\[
(1-a^2)t+a^2t^{r-1}\cos(r\theta)
\leq \cos(\theta).
\]
Letting $L(t)$ denote the left-hand side, it suffices to show
\begin{itemize}
\item[(i)] $L(t)\leq L(1)$, and
\item[(ii)] $L(1)\leq \cos(\theta)$.
\end{itemize}
We prove (i) by multiple applications of the mean value theorem.
First,
\[
L'(t)
=1-a^2+a^2(r-1)t^{r-2}\cos(r\theta)
\geq 1-a^2-a^2(r-1)
=1-r\sin^2(\tfrac{\pi}{2r}).
\]
Thus, it suffices to show $1-r\sin^2(\tfrac{\pi}{2r})\geq0$.
Since equality holds when $r=2$, it suffices to show that $r\mapsto r\sin^2(\tfrac{\pi}{2r})$ is decreasing over $r\geq2$.
Taking the derivative, it then suffices to show $\tan(\tfrac{\pi}{2r})\leq\frac{\pi}{r}$ for all $r>2$.
In fact, we get $\tan(x)-2x\leq 0$ for all $x\in[0,\frac{\pi}{4}]$ by noting that equality holds at $x=0$, and observing that the derivative is $\sec^2(x)-2\leq0$ over $[0,\frac{\pi}{4}]$.

For (ii), we view $L(1)-\cos(\theta)$ as a function of $\theta$, and our task is to show nonpositivity:
\[
g(\theta)
:=L(1)-\cos(\theta)
=1-a^2+a^2\cos(r\theta)-\cos(\theta)
\leq 0
\]
for $\theta\in[0,\frac{\pi}{r}]$.
Notably, $g(0)=g(\frac{\pi}{r})=0$ and $g(\frac{\pi}{2r})=-2\sin^2(\frac{\pi}{4r})\cos(\frac{\pi}{2r})<0$.
By the mean value theorem, there necessarily exist $\theta_-\in(0,\tfrac{\pi}{2r})$ and $\theta_+\in(\tfrac{\pi}{2r},\tfrac{\pi}{r})$ for which $g'(\theta_-)<0$ and $g'(\theta_+)>0$.
We claim that $g'$ has a unique root $\theta_0\in(0,\frac{\pi}{r})$, meaning $g'$ is negative over $(0,\theta_0)$ and positive over $(\theta_0,\frac{\pi}{r})$, and so our claim that $g(\theta)\leq 0$ for all $\theta\in[0,\frac{\pi}{r}]$ follows from the mean value theorem.
Write
\[
g'(\theta)
=\underbrace{\sin(\theta)}_{f_1(\theta)}-\underbrace{a^2r\sin(r\theta)}_{f_2(\theta)}.
\]
See \cref{fig:scalar g prime proof aid} for an illustration.
Observe that $f_1(0)=0$, and then $f_1$ is increasing over $[0,\frac{\pi}{r}]$ with $f_1(\theta)\leq\theta$.
Meanwhile, $f_2(0)=0$, and then $f_2$ is increasing over $[0,\frac{\pi}{2r}]$ with $f_2(\theta)\geq a^2r\cdot\frac{2r}{\pi}\cdot\theta$, and then decreasing over $[\frac{\pi}{2r},\frac{\pi}{r}]$ until $f_2(\frac{\pi}{r})=0$.
Thus, if we can show that $a^2r\cdot\frac{2r}{\pi}>1$, then $g'$ is negative over $\theta\in(0,\frac{\pi}{2r}]$ and increasing over $[\frac{\pi}{2r},\frac{\pi}{r}]$ with $g'(\frac{\pi}{r})>0$, and so $g'$ indeed has a root in $(0,\frac{\pi}{r})$ by the intermediate value theorem, and it is unique by the mean value theorem.
Finally, $a^2r\cdot\frac{2r}{\pi}>1$ follows from the fact that $r\mapsto r^2\sin^2(\frac{\pi}{2r})$ is increasing in $r\geq2$ and equals $2>\frac{\pi}{2}$ when $r=2$. 

\begin{figure}[t]
  \centering
  \includegraphics[scale = 1]{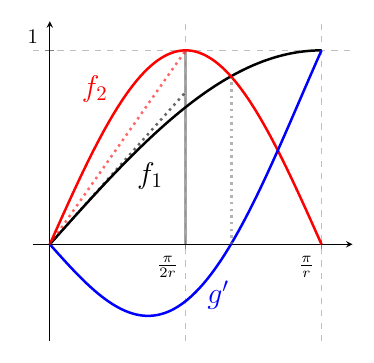}
  \caption{Plot of the functions $f_1$, $f_2$, and $g'=f_1-f_2$ analyzed in the proof of \cref{thm.root unity}.}
  \label{fig:scalar g prime proof aid}
\end{figure}

For the upper Lipschitz bound, recall that we only need to consider $z:=u^*v \in D$, where 
\[
D
:=\big\{z\in\mathbb C:0<|z|\leq 1,\, \arg(z)\in [0,\pi/r]\big\}.
\]
In particular, again taking $a:=\sin(\frac{\pi}{2r})$, it suffices to show that every $z\in D$ satisfies
\[
f(z)
:= \underbrace{a^2r^2(1 - \operatorname{Re} z)}_{\frac{1}{2}(ar)^2 d_{C_r}([u],[v])^2} - \underbrace{( 1 - (1-a^2)|z|^2 - a^2 \operatorname{Re}z^r)}_{\frac{1}{2}\|F([u])-F([v])\|^2} 
\geq 0,
\]
For $|z|=1$, this was established by Example~18 in~\cite{CahillIM:24}.
We will show $\frac{\partial f(z)}{\partial\Re(z)}\leq 0$ for all $z\in D$, since then extrapolating leftward from the $|z|=1$ edge gives $f(z)\geq0$ for all $z\in D$. 

To this end, we start by computing a Wirtinger derivative:
\begin{align*}
\frac{\partial}{\partial z} f(z)
&=\frac{\partial}{\partial z}\bigg[a^2r^2\bigg(1 - \frac{z+\overline z}{2}\bigg)
-\bigg(1-(1-a^2)z\overline{z}-a^2\cdot\frac{z^r+\overline{z}^r}{2}\bigg)\bigg]\\[6pt]
&=-\frac{1}{2}a^2r^2+(1-a^2)\overline{z}+\frac{1}{2}a^2rz^{r-1}.
\end{align*}
Writing $z = te^{i\theta}$ with $\theta\in [0,\frac{\pi}{r}]$, we obtain
\[
\frac{\partial f(z)}{\partial\Re(z)} 
= \Re \frac{\partial }{\partial z}f(z) 
= -\frac{1}{2}a^2r^2+(1-a^2)t\cos(\theta)+\frac{1}{2}a^2rt^{r-1}\cos((r-1)\theta).
\]
This expression is maximized at $t=1$ and $\theta = 0$, so it suffices to verify
\[
-\frac{1}{2}a^2r^2+(1-a^2)+\frac{1}{2}a^2r
\leq0
\]
for every integer $r\geq 2$.
This inequality rearranges to $a^2
\geq\frac{2}{r^2-r+2}$, which holds for small $r$ by direct computation.
Meanwhile, for large $r$, this inequality follows from the facts that
\[
a^2
=\sin^2\Big(\frac{\pi}{2r}\Big)
\sim \frac{\pi^2}{4}\cdot\frac{1-o(1)}{r^2},
\qquad
\frac{2}{r^2-r+2}
\sim 2\cdot\frac{1+o(1)}{r^2},
\qquad
\frac{\pi^2}{4}>2.
\]
We conclude by making these bounds effective.
Letting $x:=\frac{1}{r}$, we compute the Taylor series expansions of both functions about $x=0$, and then apply Taylor's theorem to obtain uniform bounds of both functions over $x\in (0,\frac{1}{2}]$:
\[
\sin^2\Big(\frac{\pi}{2}x\Big)
\geq \frac{\pi^2}{4}x^2-\frac{\pi^4}{ 48}x^4
> 2x^2+\frac{12}{3!}x^3
\geq \frac{2x^2}{1-x+2x^2},
\]
where the strict inequality holds when $x\in(0,0.195)$.
Thus, $a^2
\geq\frac{2}{r^2-r+2}$ for all $r\geq 6$.
\end{proof}
    

\subsection{Special orthogonal groups}
\label{sec.so}

In this section, we construct a bilipschitz embedding of $\mathbb R^{r\times n}/\SO(r)$.
Considering this space reduces to $\mathbb R^{r\times n}/\O(r)$ when $n<r$, we focus on the setting in which $n\geq r$.
Our bilipschitz map is directly inspired by the classical generating family of polynomial invariants, which, given $X\in\mathbb{R}^{r\times n}$, pairs the $\O(r)$-invariant Gram matrix $X^T X\in\mathbb{R}^{n\times n}$ with the $\SO(r)$-invariant \textbf{Pl\"{u}cker coordinates} $\operatorname{Plu}(X)\in \mathbb{R}^{\binom{n}{r}}$, namely, the determinants of all $r\times r$ submatrices of $X$.
First, we modify the Gram matrix so that it bilipschitzly separates $\O(r)$-orbits.
For this, it suffices to take the positive semidefinite square root of the Gram matrix.
Next, we perform a careful spectral modification of the Pl\"{u}cker coordinates to obtain a Lipschitz map that induces an orbit-expanding and alignment-preserving map $\mathbb{R}^{r\times n}/\SO(r)\to\mathbb{R}^{\binom{n}{r}}$ with respect to the group $\O(r)/\SO(r)$.
By combining these maps, we obtain a quotient--orbit embedding of $\mathbb R^{r\times n}/\SO(r)$ with distortion at most $2\sqrt{2}$. 



We start by considering the map $\phi\colon \mathbb{R}^{r\times n}\to\mathbb{R}^{n\times n}$ defined by
\[
\phi(X) 
= \sqrt{X^TX}.
\]
It turns out that $\phi_{/\O(r)}$ achieves the Euclidean distortion of $\mathbb R^{r\times n}/\O(r)$.

\begin{proposition}
\label{prop.gram bounds}
Given $X,Y\in \mathbb R^{r\times n}$, it holds that
\[
d_{\O(r)}([X],[Y]) 
\leq \|\phi(X) - \phi(Y)\|_F
\leq \sqrt{2} \cdot d_{\O(r)}([X],[Y]).
\]
The lower Lipschitz bound is tight for all $n\geq 1$, and the upper Lipschitz bound is tight for all $n\geq 2$. 
Moreover, for all $n\geq 2$, we have
\[
c_2(\mathbb R^{r\times n}/\Oname(r)) 
= \kappa(\phi_{/\O(r)}) 
= \sqrt 2.
\]
\end{proposition}

\begin{proof}
The inequality and tightness of bounds follow by adapting the proof of Theorem~3.7(i) in~\cite{BalanD:22}. 
Meanwhile, $c_2$ optimality is implied by the following chain of inequalities:
\[
\sqrt 2
\geq\kappa(\phi_{/\O(r)})
\geq c_2(\mathbb R^{r\times n}/\O(r)) 
\geq c_2(\mathbb R^{1\times n}/\O(1)) 
= \sqrt 2,
\] 
where the final equality follows from Corollary~36 in~\cite{CahillIM:24} (summarized in \cref{prop.known-distortions}), and final inequality holds since the zero-padding map $F\colon \mathbb R^{1\times n}\hookrightarrow \mathbb R^{r\times n}$ induces an isometric embedding of quotient spaces $\mathbb R^{1\times n}/\O(1)\to \mathbb R^{r\times n}/\O(r)$:
\begin{align*}
d_{\O(r)}([F(X)],[F(Y)])^2 
&= \|F(X)\|_F^2 + \|F(Y)\|_F^2 - 2\|F(X)F(Y)^T\|_*\\
&= \|X\|_F^2 + \|Y\|_F^2 - 2 \|XY^T\|_*\\
&= d_{\O(1)}([X],[Y])^2.
\tag*{\qedhere}
\end{align*}
\end{proof}

\begin{remark}
\label{rk.U(r)}
A simple modification of this proof gives $c_2(\mathbb C^{r\times n}/\operatorname{U}(r)) = \sqrt 2$ when $n\geq 2$.
\end{remark}

Now that we can bilipschitzly separate $\O(r)$-orbits, it remains to bilipschitzly separate $\SO(r)$-orbits within $\O(r)$-orbits.
We accomplish this by modifying the Pl\"ucker coordinates.
Given $X\in\mathbb{R}^{r\times n}$, take any thin singular value decomposition $X=U_X\Sigma_XV_X$ with $U_X\in \SO(r)$ and $V_X\in \mathbb R^{r\times n}$ such that $V_XV_X^T=\Id_r$.
Then we define the \textbf{scaled Pl\"ucker invariant} of $X$ to be
\[
\psi(X)
:=\sigma_{\min}(X)\cdot\operatorname{Plu}(V_X),
\]
where $\sigma_{\min}(X)$ is the smallest (i.e., $r$th largest) singular value of $X$.
Note that while $V_X$ is not uniquely determined by $X$, all possible choices of $V_X$ reside in a common $\SO(r)$-orbit, and so $\psi$ is well defined.
It turns out that $\psi$ bilipschitzly separates $\SO(r)$-orbits within $\O(r)$-orbits as intended:

\begin{lemma}
\label{thm.min sing scale upper}
Let $G:=\O(r)/\SO(r)$. 
Then $\psi_{/\SO(r)}$ is a $1$-Lipschitz map that is orbit expanding and alignment preserving with respect to the isometric action of $G$ on $\mathbb R^{r\times n}/\SO(r)$ and the sign representation of $G$ on $\mathbb R^{\binom{n}{r}}$.
\end{lemma}

We defer the proof of \cref{thm.min sing scale upper} to the end of the section.
What follows is the main result of this section.

\begin{theorem}
\label{thm.so psi lower bound}
The map $F\colon \mathbb R^{r\times n}/\operatorname{SO}(r)\to \mathbb R^{n\times n}\times \mathbb R^{\binom{n}{r}}$ defined by
\[
F([X]) 
= \big(~\phi(X),~\sqrt 2\cdot\psi(X)~\big)
\]
has distortion $\kappa(F)\leq 2\sqrt{2}$.
Moreover, for all $n\geq r\geq 2$, we have
\[
\sqrt{2}
\leq c_2(\mathbb R^{r\times n}/\SO(r))
\leq 2\sqrt 2.
\]
\end{theorem}

\begin{proof}
We may combine \cref{thm.align preserve,prop.gram bounds,thm.min sing scale upper} to obtain the distortion bound $\kappa(F)\leq 2\sqrt{2}$, which in turn implies the upper bound on $c_2(\mathbb R^{r\times n}/\SO(r))$.
For the lower bound, note that the zero-padding map $\mathbb R^{(r-1)\times n}\hookrightarrow \mathbb R^{r\times n}$ induces an isometric embedding of quotient spaces $\mathbb R^{(r-1)\times n}/\Oname(r-1)\to \mathbb R^{r\times n}/\SO(r)$, and so \cref{prop.gram bounds} gives
\[
c_2(\mathbb R^{r\times n}/\SO(r))
\geq c_2(\mathbb R^{(r-1)\times n}/\Oname(r-1))
= \sqrt{2}.
\tag*{\qedhere}
\]
\end{proof}

The remainder of this section is dedicated to the proof of \cref{thm.min sing scale upper}.
The Cauchy--Binet formula implies
\begin{equation*}
\langle \psi(X), \psi(Y) \rangle 
= \sigma_{\min}(X)\sigma_{\min}(Y)\det(V_XV_Y^T),
\end{equation*}
and so the squared distance between scaled Pl\"{u}cker invariants is given by
\begin{equation}
\label{eq.plucker real normdiff}
\|\psi(X)-\psi(Y)\|^2 
= \sigma_{\min}(X)^2 + \sigma_{\min}(Y)^2 - 2\,\sigma_{\min}(X)\sigma_{\min}(Y) \det(V_XV_Y^T).
\end{equation}
We will compare this to the squared orbital distance
\begin{equation*}
\label{eq.dist so eq}
d_{\SO(r)}([X],[Y])^2 
= \min_{Q\in \operatorname{SO}(r)}\|QX-Y\|_F^2 
= \|X\|_F^2 + \|Y\|_F^2 - 2\max_{Q\in\operatorname{SO}(r)}\operatorname{Tr}\big(QXY^T\big).
\end{equation*}
Conveniently, the maximum on the right-hand side has a closed-form expression:

\begin{proposition}[Theorem~1 of~\cite{MirandaT:93}]
\label{prop.so max filter}
For every $M\in\mathbb R^{r\times r}$, it holds that
\begin{equation*}
\label{eq.max filter so eq}
\max_{Q\in \operatorname{SO}(r)}\operatorname{Tr}(QM) 
= \|M\|_* - 2\sigma_{\min}(M) \cdot\mathbf{1}_{\{\det(M)<0\}},
\end{equation*}
where $\|\cdot \|_*$ denotes the nuclear norm.
\end{proposition}

We provide a streamlined proof of \cref{prop.so max filter} in \cref{app.calc proof}.
In what follows, we make continual use of the resulting expression for the squared orbital distance:
\begin{equation}
\label{eq.dist so eq full}
d_{\SO(r)}([X],[Y])^2  
= \|X\|_F^2 + \|Y\|_F^2 - 2\|XY^T\|_* + 4\sigma_{\min}(XY^T) \cdot\mathbf{1}_{\{\det(XY^T)<0\}}.
\end{equation}

\begin{proof}[Proof of \cref{thm.min sing scale upper}]
For orbit expanding, take any $B \in \operatorname{O}(r)$ such that $\det(B)=-1$. 
Then $\det(V_XV_{BX}^T)=-1$, and so \eqref{eq.plucker real normdiff} gives
\[
\|\psi(X)-\psi(BX)\|^2
=\sigma_{\min}(X)^2 + \sigma_{\min}(BX)^2 + 2\sigma_{\min}(X)\sigma_{\min}(BX)
=4\sigma_{\min}(X)^2.
\]
Meanwhile, \eqref{eq.dist so eq full} gives
\[
d_{\SO(r)}([X],[BX])^2 
= \|X\|_F^2 + \|BX\|_F^2 - 2\|X(BX)^T\|_* + 4\sigma_{\min}(X(BX)^T)
=4\sigma_{\min}(X)^2.
\]
Next, for alignment preserving, consider any $B\in\O(r)$.
Then \eqref{eq.dist so eq full} implies
\begin{align*}
d_{\SO(r)}([X],[BY])\leq d_{\SO(r)}([X],[Y])
&\iff \det(XY^T) \leq 0 \\
&\iff \det(V_XV_Y^T)\leq 0 \\
&\iff \|\psi(X) - \psi(BY)\|
\leq \|\psi(X) - \psi(Y)\|,
\end{align*}
where the last step follows from \eqref{eq.plucker real normdiff}.

It remains to show that $\psi_{/\SO(r)}$ is $1$-Lipschitz.
Fix any $X,Y\in\mathbb{R}^{r\times n}$.
Then letting $\sigma_i(Z)$ denote the $i$th largest singular value of a matrix $Z$, \labelcref{eq.plucker real normdiff,eq.dist so eq full} give that it suffices to prove
\begin{align*}
&\sigma_r(X)^2 + \sigma_r(Y)^2 - 2\sigma_r(X)\sigma_r(Y)\det(V_XV_Y^T) \\
&\qquad\leq \sum_{i=1}^r \big(\sigma_i(X)^2 + \sigma_i(Y)^2\big) - 2\|XY^T\|_* + 4\sigma_r(XY^T) \cdot\mathbf{1}_{\{\det(XY^T)<0\}}.
\end{align*}
We will do so with the help of the von Neumann trace inequality for multiple matrices:
\[
\|XY^T\|_* 
= \max_{Q\in \O(r)}\operatorname{Tr}(Q\Sigma_X V_XV_Y^T \Sigma_Y) 
\leq \sum_{i=1}^r \sigma_i(X)\sigma_i(Y)\sigma_i(V_XV_Y^T),
\]
as well as the following multivariate version of Bernoulli's inequality:
\[
\prod_{i=1}^r(1+x_i)
\geq1+\sum_{i=1}^r x_i,
\]
which holds when all the $x_i$'s have the same sign and are at least $-1$.
First,
\[
\begin{aligned}
&\|XY^T\|_*-\frac{1}{2}\sum_{i=1}^{r-1}\big(\sigma_i(X)^2+\sigma_i(Y)^2\big)\\
&\leq\sum_{i=1}^r \sigma_i(X)\sigma_i(Y)\sigma_i(V_XV_Y^T)-\sum_{i=1}^{r-1}\sigma_i(X)\sigma_i(Y)&&\text{(von Neumann and AM--GM)}\\
&=\sigma_r(X)\sigma_r(Y)+\sum_{i=1}^r \sigma_i(X)\sigma_i(Y)\big(\sigma_i(V_XV_Y^T)-1\big)&&\text{(add zero)}\\
&\leq\sigma_r(X)\sigma_r(Y)\bigg(1+\sum_{i=1}^r \big(\sigma_i(V_XV_Y^T)-1\big)\bigg)&&\text{($\sigma_i\geq\sigma_r$ and $\sigma_i(V_XV_Y^T)-1\leq0$)}\\
&\leq\sigma_r(X)\sigma_r(Y)\prod_{i=1}^r\sigma_i(V_XV_Y^T)&&\text{(mulivariate Bernoulli)}\\
&=\sigma_r(X)\sigma_r(Y)|\det(V_XV_Y^T)|.
\end{aligned}
\]
If $\det(XY^T)\geq0$, then $\det(V_XV_Y^T)=|\det(V_XV_Y^T)|$, and so rearranging gives the desired inequality in this case.
If $\det(XY^T)<0$, then $\det(V_XV_Y^T)=-|\det(V_XV_Y^T)|$, and so
\begin{align*}
\|XY^T\|_*-\frac{1}{2}\sum_{i=1}^{r-1}\big(\sigma_i(X)^2+\sigma_i(Y)^2\big)
&\leq\sigma_r(X)\sigma_r(Y)|\det(V_XV_Y^T)|\\
&=\sigma_r(X)\sigma_r(Y)\det(V_XV_Y^T)+2\sigma_r(X)\sigma_r(Y)|\det(V_XV_Y^T)|\\
&\leq\sigma_r(X)\sigma_r(Y)\det(V_XV_Y^T)+2\sigma_r(XY^T),
\end{align*}
where the last step uses the facts that $\sigma_r(X)\sigma_r(Y)\leq\sigma_r(XY^T)$ and $|\det(V_XV_Y^T)|\leq1$.
Rearranging then gives the desired inequality in this case.
\end{proof}

\subsection{Alternating subgroups of reflection groups}
\label{sec.alternating}

Given a finite-dimensional real inner product space $V$, consider any finite group $R\leq\O(V)$ that is generated by reflections, i.e., orthogonal transformations that fix a hyperplane of codimension $1$. 
It is known that $c_2(V/R)=1$; see for example~\cite{MixonQ:22}.
In this section, we instead mod out by the \textbf{alternating subgroup} $R^+ := R \cap \SO(V)$. 

\begin{theorem}
\label{thm.alternating subgroup}
Given a reflection group $R\leq\O(V)$, let $n\in\mathbb{N}$ denote the largest entry in the corresponding Coxeter matrix.
Then
\[
n\sin(\tfrac{\pi}{2n})
\leq c_2(V/R^+)
\leq 2.
\]
In particular, the lower bound is at least $\sqrt{2}$ when $|R|>2$, and at least $\frac{3}{2}$ when $R$ is also irreducible as a reflection group.
\end{theorem}

For example, if $V=\mathbb{R}^2$, then $R=D_n$ and $R^+=C_n$, and it follows from Corollary~38 in~\cite{CahillIM:24}~(see \cref{prop.known-distortions}) that $c_2(V/R^+)=n\sin(\tfrac{\pi}{2n})$.
In the general setting, the lower bound in \cref{thm.alternating subgroup} follows from reducing to this $2$-dimensional case by an application of \cref{thm.wednesday}.
It is unclear whether $c_2(V/R^+)$ always equals this lower bound, but we do not believe it ever equals the upper bound in \cref{thm.alternating subgroup}.
We derive this upper bound as a consequence of \cref{corr.glued space} after identifying $V/R^+$ with the space obtained by gluing two copies of a Weyl chamber $C\subseteq V$ of $R$ along the boundary.

In particular, let $\phi\colon V \to C$ denote the map that sends each $v\in V$ to the unique member of $(R\cdot v)\cap C$, and let $V_+$ (resp.\ $V_-$) denote the subset of $v\in V$ for which there exists $Q\in R^+$ (resp.\ $R\setminus R^+$) such that $Qv=\phi(v)$.
Both $V_+$ and $V_-$ are $R^+$-invariant, and we have $V = V_+\cup V_-$ with 
\[
V_+\cap V_- 
= \{v\in V: \phi(v)\in \partial C\}.
\]
Define an $R^+$-invariant map $\tilde F\colon V\to C\sqcup_{\partial C}C$ by setting $\tilde F|_{V_{\varepsilon}}(v) := (\phi(v), \varepsilon)$, and let $F\colon V/R^+\to C\sqcup_{\partial C}C$ denote the induced map.

\begin{lemma}
\label{lem.weyl chamber glued space}
The map $F\colon V/R^+\to C\sqcup_{\partial C}C$ is an isometry.
\end{lemma}

We prove \cref{lem.weyl chamber glued space} after first using it to prove \cref{thm.alternating subgroup}.

\begin{proof}[Proof of \cref{thm.alternating subgroup}]
The lower bound is trivial when $n=1$, so we may assume $n\geq 2$.
Consider a rotation $Q\in R^+$ of order $n$, expressed as the composition of two distinct reflections across hyperplanes $H$ and $H'$ that bound a common Weyl chamber $C$ and intersect at an angle of $\frac{\pi}{n}$ radians. 
Then the $1$-eigenspace $U:= H\cap H'$ of $Q$ has codimension $2$, so its pointwise stabilizer $G_U\leq R^+$ restricts isomorphically to a subgroup of $\operatorname{SO}(U^\perp)\cong\operatorname{SO}(2)$, and is therefore generated by a rotation with angle $\frac{2\pi}{|G_U|}$ radians. 
Moreover, the intersection $U^\perp \cap C$ is a two-dimensional closed cone with angle $\frac{\pi}{n}$ radians, and it contains at most one representative from each $G_U$-orbit. 
In particular, $|G_U|<2n$. 
Since $Q\in G_U$ has order $n$, it follows that $G_U = \langle Q\rangle$.
Then \cref{thm.wednesday} and Corollary~38 in~\cite{CahillIM:24}~(see \cref{prop.known-distortions}) together imply
\[
c_2(V/R^+) 
\geq c_2(U^\perp/\langle Q\rangle) 
= n\sin(\tfrac{\pi}{2n}).
\]
For the upper bound, \cref{lem.weyl chamber glued space,corr.glued space} together imply
\[
c_2(V/R^+)
\leq c_2(C\sqcup_{\partial C}C)
\leq \sqrt{2}\cdot\sqrt{c_2(C)^2+1}
= 2.
\]
Finally, the ``in particular'' claims follow from Coxeter's classification.
Indeed, $|R|>2$ implies $n\geq 2$, and any irreducible $R$ with $|R|>2$ has $n\geq 3$.
\end{proof}



    

\begin{figure}
    \centering
    \includegraphics[width=0.3\linewidth]{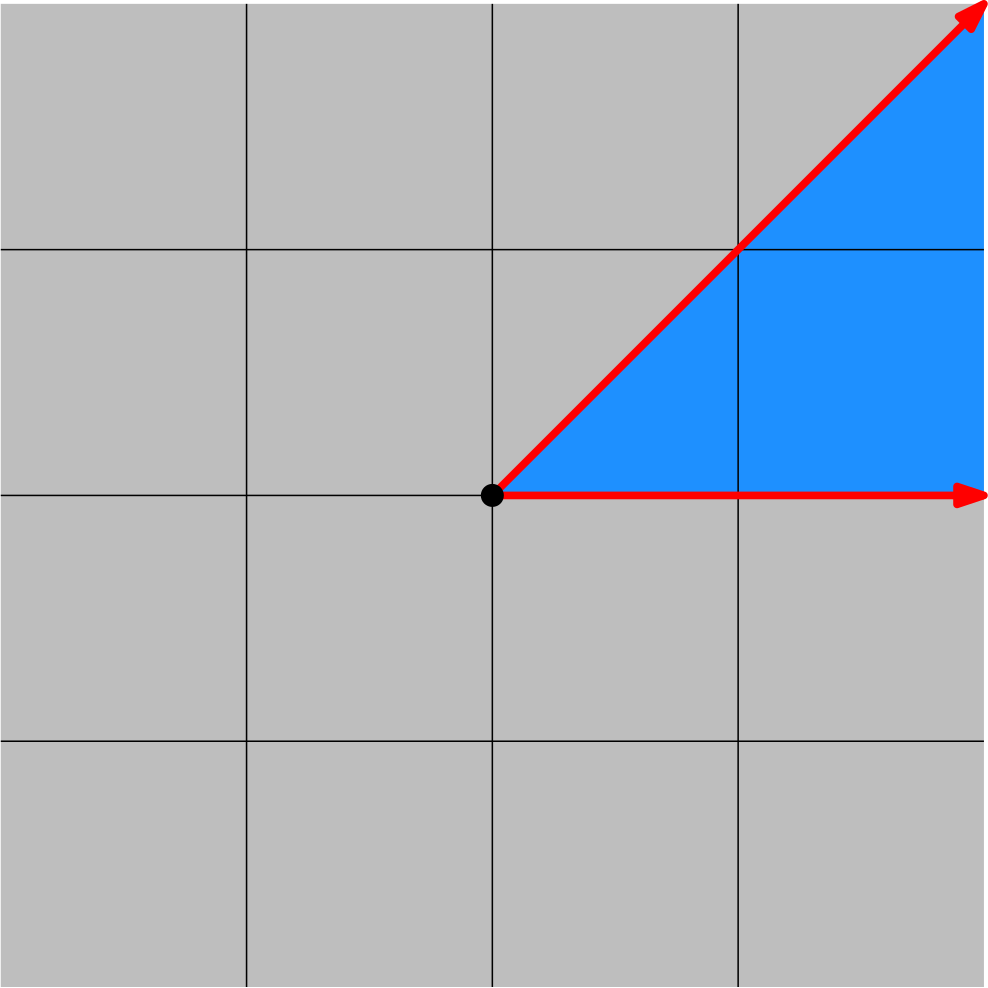}%
    \qquad\qquad\qquad%
    \includegraphics[width=0.4\linewidth]{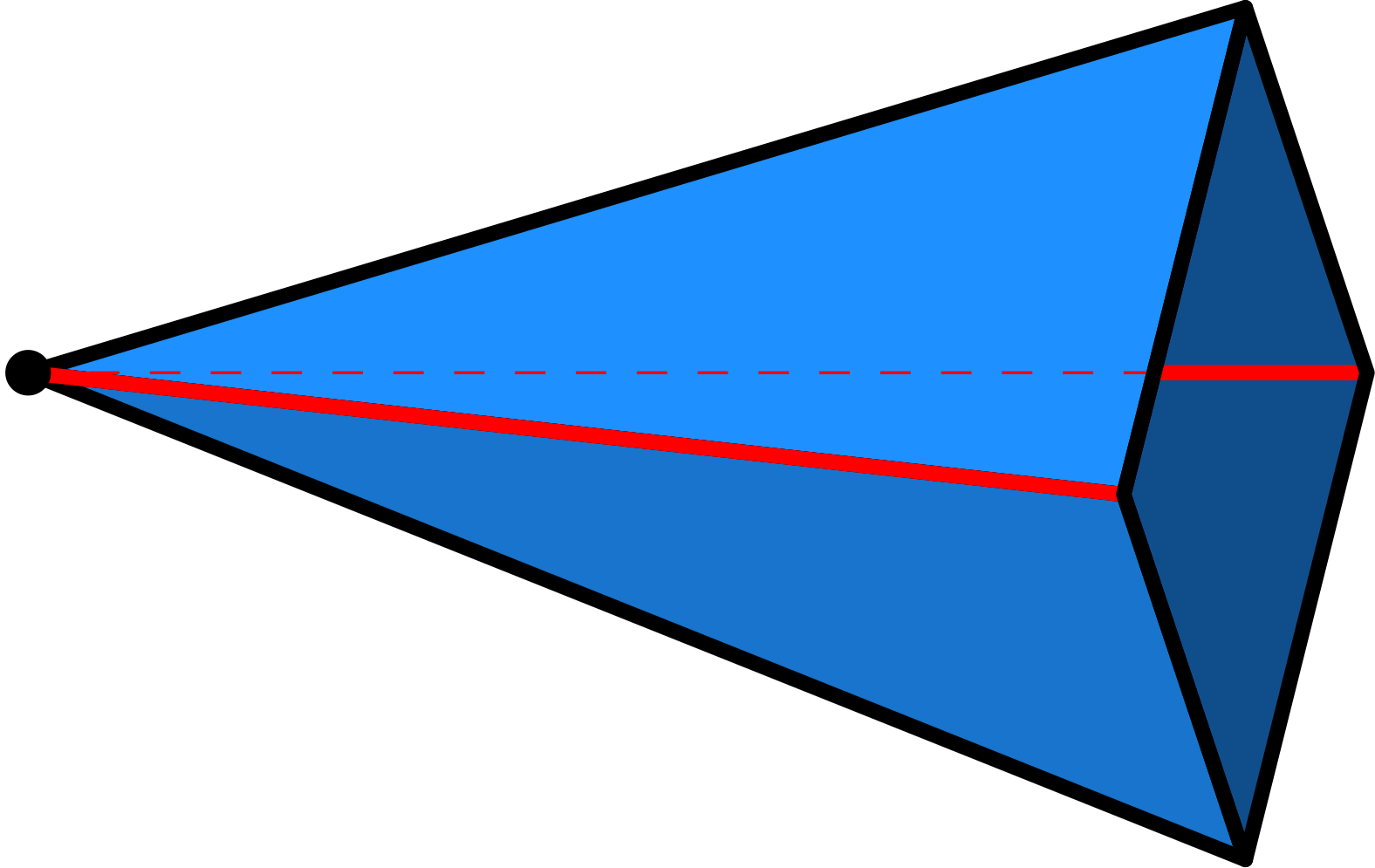}
    \caption{\textbf{(left)} A Weyl chamber $C$ for an action of $D_4$, the dihedral group with 8 elements, on $\mathbb R^2$. We highlight the boundary of the Weyl chamber in red. \textbf{(right)} An embedding of the quotient $\mathbb R^2/D_4^+$ in $\mathbb{R}^3$. This embedding is obtained by identifying $\mathbb R^2/D_4^+ \cong C\sqcup_{\partial C} C$ and applying \cref{corr.glued space}, as we do in Section~\ref{sec.alternating}. The boundary $\partial C$ along which the two copies of $C$ are glued is again highlighted in red. While the true embedded quotient is unbounded, the embedding in the figure is cut to show the shape of the cross-section.}
    \label{fig:glass}
\end{figure}

\begin{proof}[Proof of \cref{lem.weyl chamber glued space}]
Fix $u,v\in V$. 
If $u,v\in V_+$ or $u,v\in V_-$, then there exist $w,w'\in R$ with $\det(w)=\det(w')$ such that $\phi(u)=wu$ and $\phi(v) = w'v$.
Since $R\cdot x\mapsto\phi(x)$ defines an isometric embedding of $V/R$, it follows that
\[
\begin{aligned}
\min_{g\in R}\|u-gv\| 
= \|\phi(u)-\phi(v)\|
&= \|wu-w'v\| \\
&= \|u - w^{-1}w'v\| 
\geq \min_{g\in R^+}\|u-gv\| 
\geq \min_{g\in R}\|u-gv\|,
\end{aligned}
\]
meaning the inequalities are equalities, i.e.,
\[
d_{C\sqcup_{\partial C}C}\big(F([u]),F([v])\big)
= \|\phi(u)-\phi(v)\| 
= d_{V/R^+}([u],[v]).
\]

Now suppose $u\in V_+\setminus V_-$ and $v\in V_-\setminus V_+$. 
Consider the unit vectors $\alpha_1,\ldots,\alpha_m$ that are orthogonal to the facets of $C$ and point into $C$, and let $s_i:=\operatorname{Id}_V-2\langle \cdot,\alpha_i\rangle\alpha_i\in R\setminus R^+$ denote the corresponding fundamental reflections.
Taking $x:= \phi(u)$, $y:= \phi(v)$, then $x,y\in \operatorname{int}(C)$, and it suffices to show the following equalities:
\begin{equation*}
\label{eq.proof alternating}
\underbrace{\min_{g\in R\setminus R^+}\|x - g\cdot y\|}_{d_{V/R^+}([u],[v])}
\overset{\text{(a)}}{=} \min_{i\in [m]}\|x - s_i(y)\| 
\overset{\text{(b)}}{=} \underbrace{\inf_{c\in\partial C}\big(\|x-c\|+\|c-y\|\big)}_{d_{C\sqcup_{\partial C}C}(F([u]),F([v]))}.
\end{equation*}
We will establish (a) and (b) by first reducing to two claims that are perhaps more interpretable.
Given any $z\in V$, consider the index subset 
\[
k(z)
:= \{i\in [m]:\langle z,\alpha_i\rangle < 0\}.
\]
Notably, $k(z)=\varnothing$ if and only if $z\in C$. We claim the following:
\begin{enumerate}[label=(\roman*)]
\item For each $z\in V$, if $i\in k(z)$, then $\|x-s_i(z)\|<\|x-z\|$.
\item For each $i\in [m]$, the set $C\cup s_i(C)$ is convex.
\end{enumerate}
In what follows, we demonstrate (i)$\Rightarrow$(a) and (ii)$\Rightarrow$(b), and then (i) and (ii).

For (i)$\Rightarrow$(a), since $s_i\in R\setminus R^+$ for every $i\in[m]$, we have
\[
\min_{g\in R\setminus R^+}\|x - g\cdot y\| 
\leq \min_{i\in [m]}\|x - s_i(y)\|.
\]
For the reverse inequality, take
\begin{equation}
\label{eq.g star}
g_*
\in \argmin_{g\in R\setminus R^+}\|x - g\cdot y\|.
\end{equation}
Then $g_*y\in \operatorname{int}(g_*C)\subseteq C^c$, and so $k(g_*y) \neq \varnothing$. 
Pick any $i\in k(g_*y)$. 
Then $k(s_i(g_*y))=\varnothing$, since otherwise, assuming $j\in k(s_i(g_*y))$, we use \eqref{eq.g star} followed by two applications of (i) to obtain the contradiction:
\begin{align*}
\|x - g_*\cdot y\|\leq \|x-(s_js_i g_*)\cdot y\|< \|x-(s_i g_*)\cdot y\|< \|x-g_*\cdot y\|.
\end{align*}
Hence, $s_i(g_*y) = y\in \operatorname{int}(C)$, and since $R$ acts freely on $\bigsqcup_{g\in R}\operatorname{int}(gC)$ (e.g., see \cite[Theorem~4.2.4]{GroveB:96}), it follows that $s_ig_* = \Id_V$.
Thus, $g_* = s_i^{-1}=s_i$, as desired.

For (ii)$\Rightarrow$(b), fix $i\in [m]$, and let $H_i:= \operatorname{span}\{\alpha_{i}\}^\perp$ denote the reflection hyperplane of $s_{i}$.
Then $C\cap H_i = s_i(C)\cap H_i =\partial C\cap H_i$, and as a consequence of (ii) and $\langle x,\alpha_i\rangle > 0 > \langle s_i(y),\alpha_i\rangle$, we have
\[
\|x-s_{i}(y)\| 
= \min_{h\in \partial C\cap H_i}\big(\|x-h\|+\|h-s_i(y)\|\big).
\]
Furthermore, $\|h-s_i(y)\|=\|s_i^{-1}(h)-y\|=\|h-y\|$, and so
\[
\min_{i\in[m]}\|x-s_{i}(y)\| 
= \min_{i\in[m]}\min_{h\in \partial C \cap H_i}\big(\|x-h\|+\|h-y\|\big).
\]
Then (b) follows since $\bigcup_{i\in[m]}(\partial C\cap H_i)=\partial C$.

For (i), fix $z\in V$ and $i\in k(z)$. 
Then $s_i(z)-z=-2\langle z,\alpha_i\rangle\alpha_i$, and so
\[
\langle x,s_i(z)\rangle - \langle x,z\rangle 
=\langle x,s_i(z)-z\rangle
= -2\langle z,\alpha_i\rangle\langle x,\alpha_i\rangle 
> 0,
\]
where the last step uses the assumptions that $i\in k(z)$ and $x\in \operatorname{int}(C)$.
Since $\|s_i(z)\|=\|z\|$, it follows that $\|x-s_i(z)\|<\|x-z\|$, as claimed.

For (ii), fix $i\in[m]$, and consider the convex polyhedral cone
\[
K
:= \big\{z\in V:\langle z,\alpha_j\rangle\geq 0 \text{ and } \langle z,s_i(\alpha_j)\rangle \geq 0,~\forall\, j\in [m]\setminus\{i\}\big\}.
\]
It suffices to show that $C\cup s_i(C)=K$.
To this end, recall that
\begin{align*}
C
&=\{z\in V:\langle z,\alpha_j\rangle\geq 0,~\forall\, j\in [m]\},\\
s_i(C) 
&= \{z\in V:\langle z,s_i(\alpha_j)\rangle\geq 0,~\forall \, j\in[m]\}.
\end{align*}
For $K\subseteq C\cup s_i(C)$, note that $s_i(\alpha_i) = -\alpha_i$, and so each $z\in K$ either resides in $C$ or $s_i(C)$, depending on whether $\langle z,\alpha_i\rangle$ is nonnegative or nonpositive.
For $C\cup s_i(C)\subseteq K$, it suffices to show $s_i(C)\subseteq K$ since $K$ is invariant to the action of $s_i$. 
Given $z\in s_i(C)$, we have $\langle z,s_i(\alpha_j)\rangle\geq 0$ for all $j\in[m]$. Moreover, the walls of $C$ meet at acute angles, i.e., $\langle\alpha_i,\alpha_j\rangle\leq 0$ for each $j\in[m]\setminus\{i\}$ (e.g., see~\cite[Proposition~4.1.5]{GroveB:96}). 
Thus, for each $j\in[m]\setminus\{i\}$, we have
\[
0
\leq \langle z,s_i(\alpha_j)\rangle  
= \langle z,\alpha_j\rangle - 2\langle \alpha_i,\alpha_j\rangle \langle z,\alpha_i\rangle,
\]
and so $\langle z,\alpha_j\rangle \geq 2\langle \alpha_i,\alpha_j\rangle \langle z,\alpha_i\rangle=-2\langle \alpha_i,\alpha_j\rangle \langle z,s_i(\alpha_i)\rangle \geq 0$, as desired.
\end{proof}

\subsection{Wallpaper groups}
\label{sec.tran klien}

In this section, we embed quotients of the plane by the wallpaper groups, namely, the discrete subgroups $G\leq \E(2)$ such that $\mathbb R^2/G$ is compact. 
There are several familiar spaces among these quotients, including flat tori, flat Klein bottles, and the M\"obius strip.
To estimate the Euclidean distortions of these spaces, we make frequent use of \cref{lem.upsilon compute,corr.glued space,thm.isotropy}.
While there are multiple notation systems available for describing the wallpaper groups, we use the \textit{orbifold signature notation} introduced by MacBeath \cite{MacBeath:67}, in the form popularized by Conway \cite{Conway:02, Conway:16}.

In any wallpaper group $G$, the set $T_G:=G\cap\mathbb{R}^2$ of translation vectors forms a discrete rank-$2$ lattice in $\mathbb R^2$. 
All but two of the $17$ classes of wallpaper groups exhibit a key structural constraint: the translation lattice of any group in any of these classes is constrained in shape---either having fixed angles (as in rectangular and hexagonal lattices) or generators of equal length (as in rhombic lattices). 
The two exceptional classes are type $o$, in which the group consists of translations by an arbitrary lattice (yielding a flat torus quotient), and type $2222$, which is generated over a group of type $o$ by a $\pi/2$ rotation (producing a quotient of the flat torus by a group of order two).

In \cref{sec.sub flat tori}, we reviewed known results on the Euclidean distortions of quotients by type $o$ wallpaper groups. The best known general upper bound and the only known exact distortion values are summarized in the following proposition.
\begin{proposition}\label{thm:torus-distortion}
Let $G\leq E(2)$ be a wallpaper group of type $o$. Then the following hold:
\begin{itemize}
\item[(a)]
In general, $c_2(\mathbb R^2/G) < 8$.
\item[(b)] 
If $T_G$ is a rectangular lattice, then $c_2(\mathbb R^2/G) = \frac{\pi}{2}$.
\item[(c)] 
If $T_G$ is a hexagonal lattice, then $c_2(\mathbb R^2/G) = \frac{\sqrt{8}\pi}{\sqrt{27}}$.
    \end{itemize}
\end{proposition}

For each of the remaining $16$ wallpaper group types, we either precisely determine or explicitly bound the Euclidean distortions of the corresponding quotient spaces. 
In the case of type $2222$ groups, the underlying translation lattice may be oblique, so we resort to bounding the distortion using \cref{thm:torus-distortion}(a).
For the remaining $15$ types, we obtain sharper bounds or exact values for $c_2(\mathbb R^2/G)$, all of which are at most $\pi$ and expressed in more explicit terms.

Many of our lower bounds for distortion will come from the application of \cref{thm.isotropy} to a rotation center, so we quickly codify this strategy as a lemma.

\begin{lemma}\label{lem.rotation_center}
Let $G$ be a wallpaper group. 
If the action of $G$ has a rotation center of order $r$ that is not on a reflection line (that is, if the orbifold signature of $G$ has an $r$ which is not to the right of any $\ast$), then
\[
c_2(\mathbb R^2/G) 
\geq r\sin(\tfrac{\pi}{2r}).
\]
\end{lemma}

\begin{proof}
Suppose the action of $G$ has a rotation center of order $r$ at $p\in\mathbb R^2$ which is not on any reflection line. 
Then the stabilizer $G_p$ is isomorphic to $C_r$, and it acts on $T_p\mathbb R^2$ by rotations by integer multiples $\pi/r$.
Corollary 36 in \cite{CahillIM:24} (see \cref{prop.known-distortions}) and \cref{thm.isotropy} together give the desired lower bound.
\end{proof}

\begin{theorem}
\label{thm.wallpaper groups}
Let $G\leq \E(2)$ be a wallpaper group.
    \begin{enumerate}[label=(\alph*)]
        \item\label{it:reflections} If $G$ is of type ${\ast}333$, ${\ast}442$, $*632$, or ${\ast}2222$, then $c_2(\mathbb R^2/G) = 1$.
        \item\label{it:cylinder} If $G$ is of type ${\ast}{\ast}$, then $c_2(\mathbb R^2/G) = \frac{\pi}{2}$.
        \item\label{it:half-spin} If $G$ is of type $2{\ast}22$, then $c_2(\mathbb R^2/G) = \sqrt{2}$.
        \item\label{it:quarter-spin} If $G$ is of type $4{\ast}2$, then $c_2(\mathbb R^2/G) = 2\sqrt{2-\sqrt{2}}$.
        \item\label{it:klein} If $G$ is of type ${\times}{\times}$ (that is, if $\mathbb R^2/G$ is a flat Klein bottle), then $\frac{\pi}{2} \leq c_2(\mathbb R^2/G) \leq \frac{\pi}{\sqrt{2}}$.
        \item\label{it:football} If $G$ is of type $22{\times}$, then $\pi/2\leq c_2(\mathbb R^2/G) \leq \pi$.
        \item\label{it:torus-cover} If $G$ is of type $2222$, then $\sqrt{2}\leq c_2(\mathbb R^2/G)\leq \sqrt{2}\cdot c_2(\mathbb R^2/T_G) < 8\sqrt{2}$. 
        \item\label{it:mobius} If $G$ is of type ${\ast}{\times}$, (that is, if $\mathbb R^2/G$ is a M\"obius strip) then $\frac{\pi}{2}\leq c_2(\mathbb R^2/G)\leq \frac{\pi}{\sqrt{2}}$.
        \item\label{it:glass-2} If $G$ is of type $22{\ast}$, then $\sqrt{2}\leq c_2(\mathbb R^2/G)\leq 2$.
        \item\label{it:glass-3} If $G$ is of type $333$ or $3{\ast}3$, then $\frac{3}{2}\leq c_2(\mathbb R^2/G)\leq 2$.
        \item\label{it:glass-4} If $G$ is of type $442$, then $2\sqrt{2-\sqrt{2}}\leq c_2(\mathbb R^2/G)\leq 2$.
        \item\label{it:glass-6} If $G$ is of type $632$, then $3\sqrt{2-\sqrt{3}}\leq c_2(\mathbb R^2/G)\leq 2$.
    \end{enumerate}
\end{theorem}

\begin{proof} \cref{lem.rotation_center} provides the lower bounds for cases~\ref{it:half-spin}, \ref{it:quarter-spin},  \ref{it:torus-cover}, \ref{it:glass-2}, \ref{it:glass-3}, \ref{it:glass-4}, and \ref{it:glass-6}. We tackle the remaining bounds case by case.

\medskip

\noindent\ref{it:reflections}: 
Groups of these types are affine reflection groups, so the quotients $\mathbb R^2/G$ are all isometric to fundamental domains of their actions. 
Hence, they embed isometrically in $\mathbb R^2$.

\medskip
\noindent\ref{it:cylinder}: 
A quotient of this type is a flat cylinder, which means
\[
c_2(\mathbb R^2/G) 
= \max\{c_2([0,1]), c_2(S^1)\} 
= \max\{1,\tfrac{\pi}{2}\} = \tfrac{\pi}{2},
\]
where the first equality uses Lemma~39 in~\cite{CahillIM:24} (see \cref{prop.max-product}) and the second uses Theorem~6.1 in~\cite{HeimendahlLVZ:22} (see \cref{prop.known-distortions}).

\medskip
\noindent\ref{it:half-spin}: 
Note that the reflections in $G$ generate an index-$2$ normal subgroup $N$ that is a wallpaper group of type $*2222$. 
The quotient $\mathbb R^2/N$ is isometric to a rectangle, and $G/N\cong C_2$ acts on the rectangle by a half rotation.
By the second equivariant embedding lemma (\cref{lem:quotient-distortion}), there is an isometric embedding $\mathbb R^2/G \to \mathbb R^2/\{\pm \Id\}$ that descends from the equivariant isometric embedding of the rectangle $\mathbb R^2/N \to \mathbb R^2$, so $c_2(\RR^2/G)$ is at most $c_2(\RR^2/\{\pm \Id\})$. The distortion of the latter space is $\sqrt{2}$ by \cref{thm.root unity}, matching our lower bound of $\sqrt{2}$.

\medskip
\noindent\ref{it:quarter-spin}: 
The same argument that applied in case \ref{it:half-spin} applies here as well; the subgroup of $G$ generated by reflections is an index-$4$ normal subgroup $N$ that is a wallpaper group of type $*2222$. 
Since $G$ is of type $4{\ast}2$, the translation lattice $T_G = T_N$ must be square, and so the quotient $\mathbb R^2/N$ is isometric to a flat square. 
The quotient group $G/N\cong C_4$ acts on $\mathbb R^2/N$ by quarter rotations. 
As in the previous case, the upper bound is obtained by isometrically embedding $\mathbb R^2/G \to \mathbb R^2/C_4$,
the Euclidean distortion of which is $4\sin(\frac{\pi}{8}) = 2\sqrt{2-\sqrt{2}}$ by \cref{thm.root unity}.

\medskip
\noindent\ref{it:klein}:
Note that $T_G$ must be a rectangular lattice. 
The translation subgroup $N$ of $G$ is a normal subgroup of index $2$. 
Since $T_N = T_G$, $N$ is a wallpaper group of type $o$ with a rectangular underlying lattice. 
\cref{lem.upsilon compute,thm:torus-distortion}(b) together imply $c_2(\mathbb R^2/G) \leq \sqrt{2}\cdot c_2(\mathbb R^2/N) = \frac{\pi}{\sqrt{2}}$. 
For the lower bound, let $\pi:\mathbb R^2\to \mathbb R^2/G$ be the quotient map, and let $L\subseteq \mathbb R^2$ be any glide axis of $G$. 
Then $\pi(L)$ is topologically a circle in the quotient, and we claim that the distance on the circle is geodesic\footnote{A \emph{geodesic circle} is any metric space isometric to a circle equipped with the standard arc-length metric, or equivalently, to a quotient space of the form $\mathbb R/c\mathbb Z$ with its induced quotient metric.}. 
Indeed, if $x,y\in L$, then the orbit $G\cdot y$ is a rectangular grid one of whose rows is contained in $L$, and the shortest path from $x$ to a point in $G\cdot y$ remains within $L$. 
The Euclidean distortion of a geodesic circle is $\frac{\pi}{2}$, giving the desired lower bound.

\medskip
\noindent\ref{it:football}:
A group $G$ of this type is generated by a pair of glides $\gamma_1,\gamma_2$ in orthogonal axes. To prove the upper bound, note that $G$ contains an index-$2$ normal subgroup $N$ of type ${\times}{\times}$, generated by $\gamma_1^2$ and $\gamma_2$. (It is generated over $N$ by either the glide $\gamma_1$ or the order-2 rotation $\gamma_1\gamma_2$.) Case~\ref{it:klein}, \cref{thm.upsilon-bound}, and \cref{lem.upsilon compute} together give $c_2(\mathbb R^2/G) \leq \pi$. Meanwhile, for the lower bound, we claim $\RR^2/G$ contains a geodesic circle, so that $c_2(\RR^2/G)\geq \pi/2$. The argument is similar to part \ref{it:klein}. Say the magnitude of the glide $\gamma_i$ is $\alpha_i>0$ for $i=1,2$. We can assume without loss of generality that $\alpha_1 \leq 2\alpha_2$. Then the image in $\RR^2/G$ of the glide axis $L$ of $\gamma_1$ is the desired geodesic circle, as we now argue. If $y\in L$ then $G\cdot y \subseteq G\cdot L = \langle \gamma_2 \rangle \cdot L$, the latter being the union of the glide axes parallel to $L$, which are $\alpha_2$ apart from each other. If $x\in L$ too, then there is an element of $
(G\cdot y)\cap L$ at a distance of at most $\alpha_1 / 2$ from $x$. Since $\alpha_1 / 2 \leq \alpha_2$, there is a representative of $G\cdot y$ of minimum distance to $x$ lying in $L$, so the image of $L$ in $\RR^2/G$ is isometric to $L/\langle \gamma_1\rangle$, a geodesic circle.

\medskip
\noindent\ref{it:torus-cover}:
As in case~\ref{it:klein}, the translation subgroup $N$ of $G$ is normal of index $2$. 
The desired upper bound follows from \cref{thm:torus-distortion}(a) and \cref{lem.upsilon compute}.

\medskip
\noindent\ref{it:mobius}:
A group $G$ of this type contains an index-$2$ normal subgroup $N$ of type ${\ast}{\ast}$, generated by a reflection and the square of a glide of minimal magnitude. 
Case~\ref{it:cylinder} and \cref{lem.upsilon compute} together give $c_2(\mathbb R^2/G) \leq \frac{\pi}{\sqrt{2}}$. 
The lower bound comes from noting that the M\"obius strip contains a geodesic circle, the distortion of which is $\frac{\pi}{2}$.

\medskip
\noindent\ref{it:glass-2}--\ref{it:glass-6}:
For a group $G$ of any of these five types, the quotient $\mathbb R^2/G$ consists of a pair of flat pieces glued together along some portion of the boundary. 
\begin{itemize}
\item If $G$ is of type $333$, $442$, or $632$, then $G$ is the alternating subgroup of an affine reflection group of type $\ast 333$, $\ast 442$, or $\ast 632$ respectively. Thus, the quotient $\mathbb R^2/G$ is isometric to a pair of equilateral, isosceles right, or semi-equilateral triangles respectively, glued together at their boundaries. 

\item If $G$ is of type $3{\ast}3$, then it has a rotation center $p$ not on a reflection line; it is of index $2$ inside a group $\langle G, f\rangle$ of type $\ast 632$ generated by a reflection $f$ in a line through $p$. The quotient $\RR^2/\langle G,f\rangle$ is isometric to a fundamental domain $\Delta$ which is a semi-equilateral triangle; we can take it to have its $\pi/3$-vertex at $p$. Meanwhile, the reflections in $\langle G,f \rangle$ across a line through $p$ all lie in the nontrivial coset $fG$ of $G$ in $\langle G,f\rangle$ (while the reflection in the third edge of $\Delta$ already lies in $G$). Thus, identifying $\RR^2/\langle G,f\rangle$ with $\Delta$, the natural map $\RR^2/G\rightarrow \RR^2/\langle G,f\rangle$ is a degree-2 branched cover, branched along both edges of $\Delta$ that are incident to $p$. Furthermore, $\RR^2/G$ is disconnected by the ramification locus because the axes of the reflections in $Gf$ disconnect a fundamental domain for $G$. Thus, the quotient $\RR^2/G$ is isometric to two copies of $\Delta$ glued along the two sides incident to $p$, i.e., to a pair of semi-equilateral triangles glued along the short leg and the hypotenuse.

\item If $G$ is of type $22{\ast}$, the argument is similar to that for type ${3{\ast}3}$. In this case $G$ has a pair of non-conjugate order-$2$ rotation centers $p,q$ such that the line $\ell$ through $p,q$ is parallel to the reflection axes but not equal to any of them. If $f$ is the reflection in $\ell$, then the group $\langle G,f\rangle$ is of type $\ast 2222$. Its fundamental domain is a rectangle $\Omega$, and if in the first place we choose $p$ and $q$ to be as close to each other as possible, then we can choose $\Omega$ to have $p$ and $q$ as vertices. Meanwhile, all the reflections in $\langle G,f\rangle$ through either $p$ or $q$ lie in the nontrivial coset $fG$ of $G$ in $\langle G,f\rangle$ (while the reflection in the other edge of $\Omega$ lies in $G$). Thus, identifying $\RR^2/\langle G,f\rangle$ with $\Omega$, the natural map $\RR^2/G\rightarrow \RR^2/\langle G,f\rangle$ is a degree-$2$ branched cover of $\Omega$, branched over the edges incident to either $p$ or $q$. By the same reasoning as for type $3{\ast}3$, $\RR^2/G$ is disconnected by the branch locus. Thus, the quotient $\RR^2/G$ is a pair of rectangles glued along three of the four sides.
\end{itemize}
In all cases, the quotient is a pair of identical spaces of distortion $1$ glued together along a closed subset, so \cref{corr.glued space} gives an upper bound of $2$.
\end{proof}

\begin{figure}
    \centering
    \includegraphics[width=0.4\linewidth]{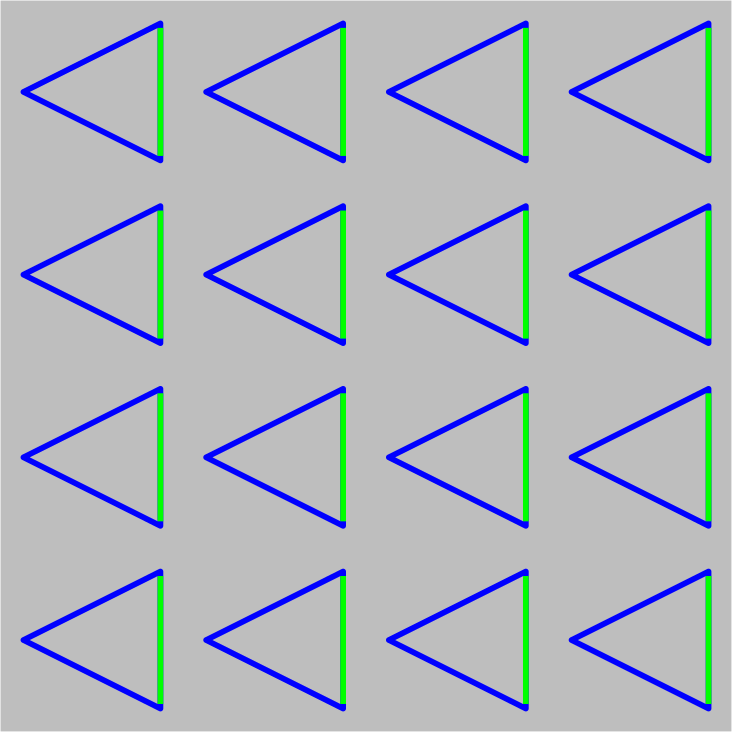}%
    \qquad\qquad%
    \includegraphics[width=0.4\linewidth]{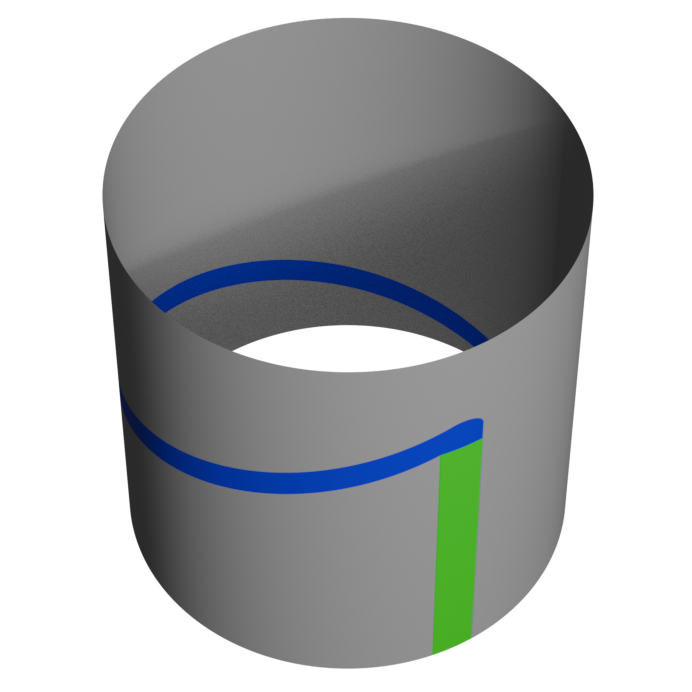}%

    \medskip
    
    \includegraphics[width=0.4\linewidth]{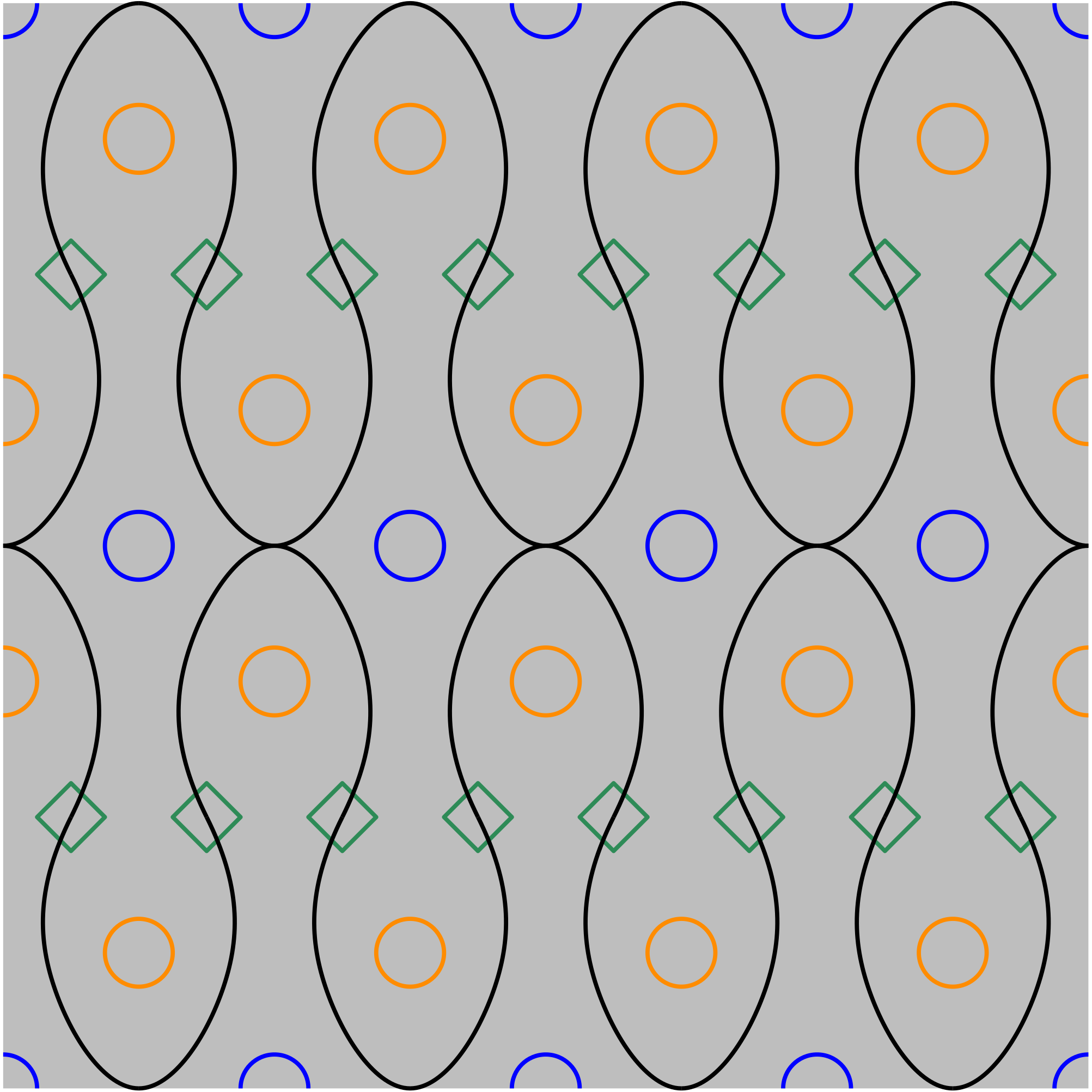}%
    \qquad\qquad%
    \includegraphics[width=0.4\linewidth]{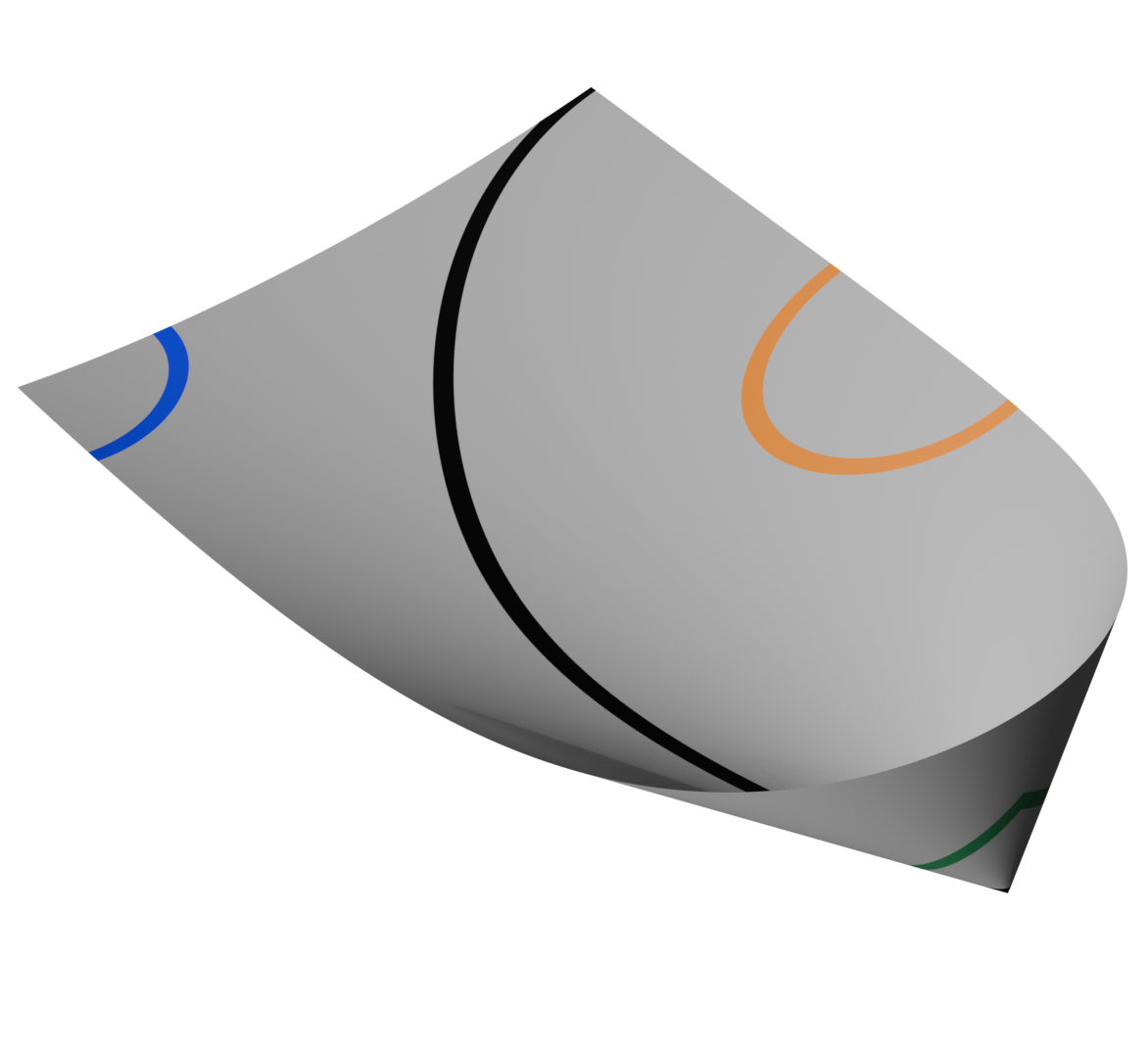}%

    \medskip
    
    \includegraphics[width=0.4\linewidth]{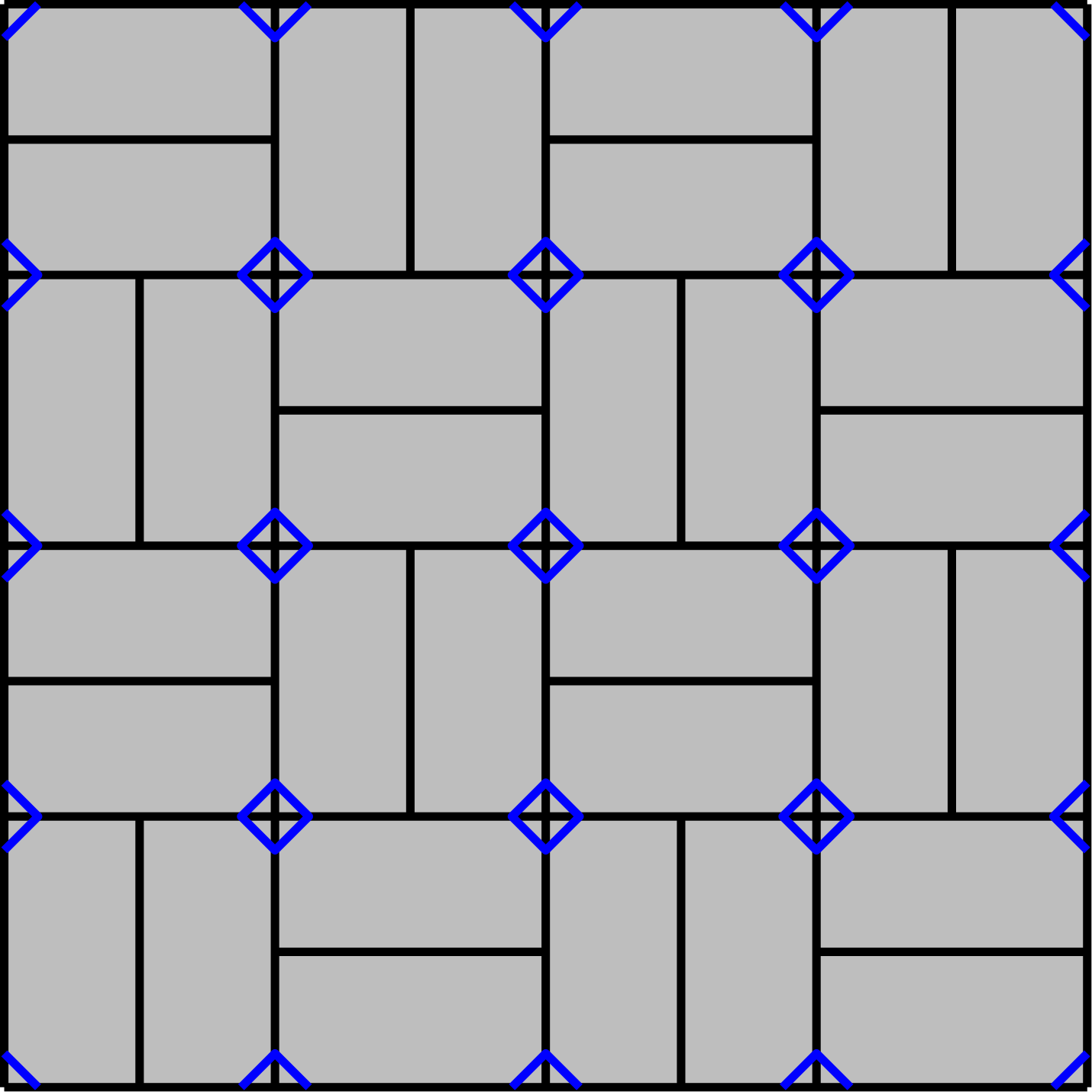}%
    \qquad\qquad%
    \includegraphics[width=0.4\linewidth]{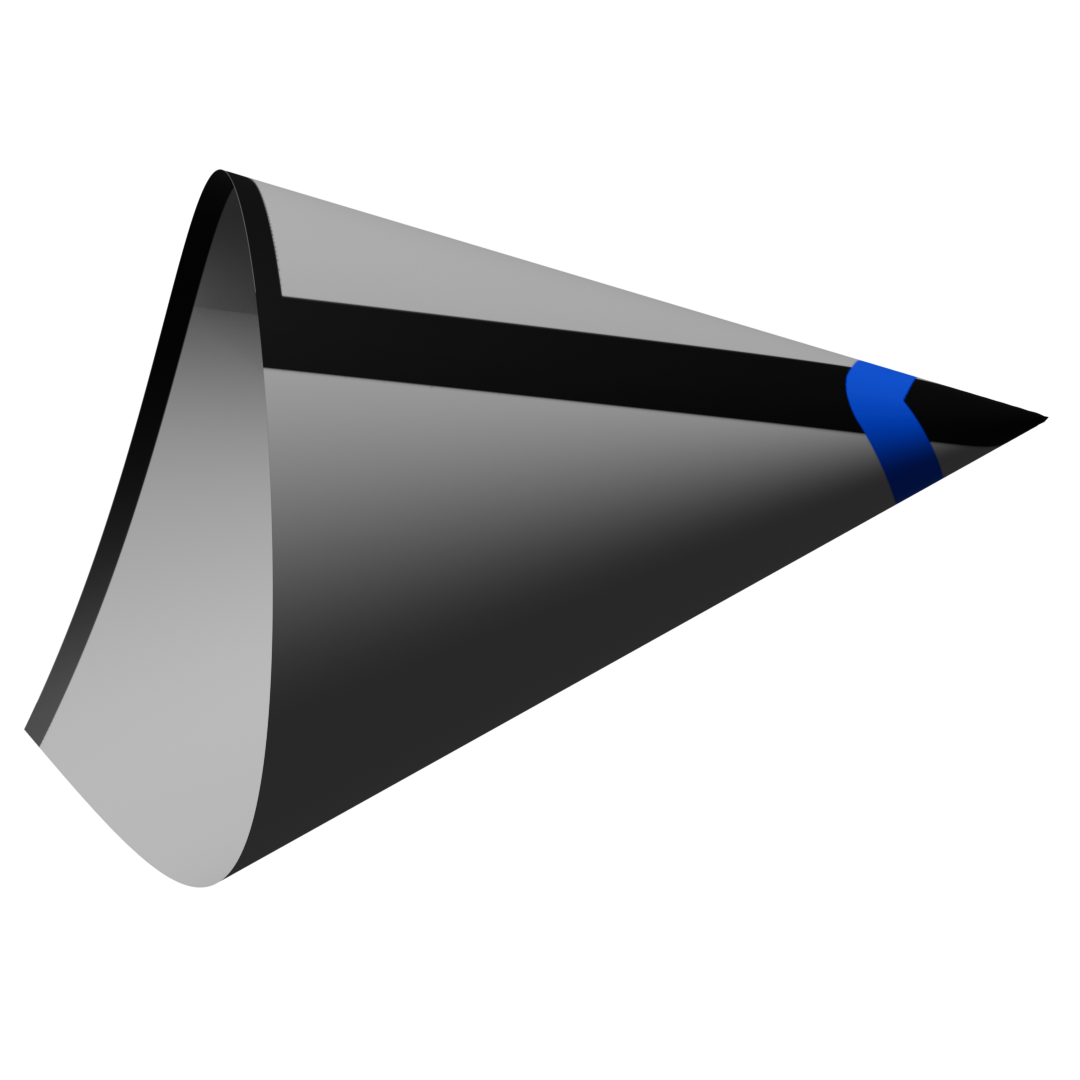}%
    
    \caption{\textbf{(left)} Wallpapers with symmetry groups of type ${\ast}{\ast}$, $2{\ast}22$, and $4{\ast}2$, respectively. \textbf{(right)} An optimal bilipschitz embedding of each wallpaper pattern quotiented by its symmetry group. The image of each embedding forms a portion of a circular cylinder or cone.}
    \label{fig:wallpaper-embeddings}
\end{figure}

\cref{thm.wallpaper groups} computes the exact Euclidean distortion of quotients of $\mathbb R^2$ by seven different types of wallpaper groups. 
Of these seven, four have a distortion of 1, optimally embedding back into $\mathbb R^2$ as a fundamental domain of the action. 
The three remaining quotients embed optimally in $\mathbb R^3$. 
The optimal embeddings implicit in the proof of \cref{thm.wallpaper groups} are illustrated in Figure~\ref{fig:wallpaper-embeddings}.

\subsection{Affine-linear actions on landmarks}
\label{sec.euclidean}

In this section, we consider a setting that is particularly relevant to \textit{morphometrics}, which quantitatively analyzes the effect of genetic or environmental factors on the size and shape of biological organisms.
To accomplish this analysis, it is common to scan a collection of specimens, and then painstakingly identify the locations of certain discrete homologous features known as \textit{landmarks}, e.g., the point on a beetle's abdomen that is furthest from its head.
After recording the spatial coordinates of $n$ landmarks in a $3$-dimensional scan, the result is an $n$-tuple of vectors in $\mathbb{R}^3$ that represents a biological specimen.
Since different rotations and translations of the specimen should be identified with each other, we mod out by the diagonal action of the special Euclidean group $\operatorname{SE}(3)=\SO(3)\ltimes\mathbb{R}^3$ on $(\mathbb{R}^3)^n$.
How shall we embed this orbit space into Euclidean space so as to facilitate data analysis?

In what follows, we apply \cref{thm.euclidean-subgroup} to treat a more general instance of this problem in which $K\leq\O(r)$ is compact and $K\ltimes\mathbb{R}^r$ acts on $(\mathbb{R}^r)^n$ by
\[
(A,b)\cdot (x_i)_{i=1}^n
=(Ax_i+b)_{i=1}^n.
\]
Note that $K\ltimes\mathbb{R}^r$ acts transitively on $\mathbb{R}^r$, and so $(\mathbb{R}^r)^n/(K\ltimes\mathbb{R}^r)$ consists of a single point in the degenerate case where $n=1$, i.e., its Euclidean distortion is undefined.


\begin{theorem}\label{thm.EG}
For any compact group $K\leq \O(r)$ and $n\geq 2$, it holds that
\[
c_2\big((\mathbb{R}^r)^n/(K\ltimes\mathbb{R}^r)\big) 
= c_2\big((\mathbb{R}^r)^{n-1}/K\big).
\]
\end{theorem}

\begin{proof}
In the context of \cref{thm.euclidean-subgroup}, we have $V=(\mathbb{R}^r)^n$.
Also, $G\leq\O(V)$ is the image of the representation $K\curvearrowright V$, namely, the group consisting of $\operatorname{Id}_n\otimes A\in\O(V)$ for every $A\in K$, while $T$ is the subspace of $(x_i)_{i=1}^n\in V$ for which all of the $x_i$'s are equal to each other.
Then $T^\perp$ is the subspace of $(x_i)_{i=1}^n\in V$ for which $\sum_{i=1}^n x_i=0$, and \cref{thm.euclidean-subgroup} gives 
\[
c_2\big((\mathbb{R}^r)^n/(K\ltimes\mathbb{R}^r)\big) 
=c_2\big(V/(G\ltimes T)\big)
=c_2(T^\perp/G).
\]
Next, if we identify each $(x_i)_{i=1}^n\in T^\perp$ with the $r\times n$ matrix whose $i$th column is $x_i$, then $T^\perp$ is the subspace of $X\in\mathbb{R}^{r\times n}$ with $X\mathbf{1}=0$.
Let $B\in\mathbb{R}^{n\times (n-1)}$ denote any matrix whose columns form an orthonormal basis for the orthogonal complement of $\mathbf{1}\in\mathbb{R}^n$.
Then $X\mapsto XB$ defines a surjective $K$-equivariant isometry $T^\perp\to(\mathbb{R}^r)^{n-1}$.
By the second equivariant embedding lemma (\cref{lem:quotient-distortion}), this descends to an isometric isomorphism of orbit spaces, and so
\[
c_2\big((\mathbb{R}^r)^n/(K\ltimes\mathbb{R}^r)\big) 
=c_2(T^\perp/G)
=c_2\big((\mathbb{R}^r)^{n-1}/K\big).
\tag*{\qedhere}
\]
\end{proof}

    


In so many words, the proof of \cref{thm.EG} indicates that the best way to embed $(\mathbb{R}^r)^n/(K\ltimes \mathbb{R}^r)$ is to first translate the tuple $(x_i)_{i=1}^n$ of landmarks to be centered at the origin (by subtracting the centroid), and then apply the appropriate $K$-invariant map as if one were embedding $(\mathbb{R}^r)^n/K$ (or more correctly, the sub-metric space corresponding to the origin-centered tuples in $(\mathbb{R}^r)^n$).
This approach allows us to embed $(\mathbb{R}^r)^n/\E(r)$ in a way that achieves the Euclidean distortion.

\begin{example}\label{cor.euclidean-orthogonal-isometry}
In the case where $K=\O(r)$, it holds that $K\ltimes\mathbb{R}^r$ is the Euclidean group $\E(r)$, in which case \cref{thm.EG,prop.gram bounds} together give
\[
c_2\big((\mathbb{R}^r)^n/\E(r)\big) = \left\{\begin{array}{cl}
\sqrt{2} &\text{if } n\geq 3\\
1 &\text{if } n=2.
\end{array}\right.
\]
\end{example}

The above example is useful in settings in which the practitioner is happy to mod out by chirality in addition to rotations and translations.
If chirality is an important feature to maintain, one should instead mod out by the special Euclidean group.

\begin{example}
In the case where $K=\SO(r)$, it holds that $K\ltimes\mathbb{R}^r$ is the special Euclidean group $\operatorname{SE}(r)$, in which case \cref{thm.EG,thm.so psi lower bound} together give
\[
\sqrt 2
\leq c_2\big((\mathbb{R}^r)^n/\operatorname{SE}(r)\big)
\leq 2\sqrt 2
\]
whenever $n-1\geq r\geq 2$.  
\end{example}


\subsection{Permutation actions on graphs and databases}
\label{sec.wednesday theorem}

While the previous sections showed how to embed various orbit spaces into Euclidean space with uniformly bounded distortion, in this section, we discuss real-world instances of the quotient embedding problem that are not so well behaved.
The examples we present exhibit a common form: 
The objects of interest are represented as a matrix only after selecting an arbitrary labeling of sorts, thereby introducing a permutation ambiguity.
Throughout this section, we use the notation $[n]:=\{1,\ldots,n\}$ for our label set.

As a precursor example, consider the set of \textit{point clouds} consisting of $n$ points in $\mathbb{R}^d$.
While one might be inclined to think of these objects as multisets of size $n$, one typically records an example of such an object by listing the constituent vectors in some order.
This results in an $n$-tuple of vectors, which you might represent as a function $f\colon [n]\to\mathbb{R}^d$, but note that the arbitrarily selected order introduces an ambiguity.
In particular, the symmetric group $S_n$ acts on the space $(\mathbb{R}^d)^{[n]}$ of such functions by $(\pi\cdot f)(i)=f(\pi^{-1}(i))$ for $\pi\in S_n$, and two functions should be identified if they reside in a common $S_n$-orbit.
As such, we may identify our space of point clouds with the orbit space $(\mathbb{R}^d)^{[n]}/S_n$, whose quotient metric coincides with the $2$-Wasserstein distance.
Unfortunately, the Euclidean distortions of these spaces are not uniformly bounded.


\begin{proposition}
\label{prop.unbounded growth}
For each $d\geq 3$, it holds that $\displaystyle\liminf_{n\to \infty}c_2\big((\mathbb{R}^d)^{[n]}/S_n\big) = \infty$.
\end{proposition}

The proof of \cref{prop.unbounded growth} follows directly from the proof of Theorem~42 in~\cite{CahillIM:24}, which in turn was adapted from the main ideas in~\cite{AndoniNN:16}.
Regarding the hypothesis $d\geq 3$, we note that when $d=1$, the Euclidean distortion is always $1$.
Meanwhile, the situation is open in the case where $d=2$.

Next, consider the set of \textit{weighted simple graphs} on $n$ vertices, where each edge is assigned some nonzero weight in $\mathbb{R}$.
Once we label the vertices by $[n]$, then each edge is determined by an unordered pair $\{i,j\}\in\binom{[n]}{2}$, and so one may represent the weighted graph as a function $f\colon \binom{[n]}{2}\to\mathbb{R}$.
In particular, if the vertices with labels $i$ and $j$ are adjacent, then $f(\{i,j\})$ equals the corresponding edge weight, and otherwise $f(\{i,j\})=0$.
Of course, since the labeling was arbitrary, this introduces an ambiguity by the action of $S_n$ on $\mathbb{R}^{\binom{[n]}{2}}$ defined by $(\pi\cdot f)(\{i,j\})=f(\{\pi^{-1}(i),\pi^{-1}(j)\})$.
The resulting orbit space $\mathbb{R}^{\binom{[n]}{2}}/S_n$ has a quotient metric that is $\mathsf{NP}$-hard to compute, as one may use this distance to detect whether a given unweighted graph contains a Hamiltonian cycle.
It turns out that this orbit space is also poorly behaved in the context of Euclidean distortion.

\begin{theorem}
It holds that $\displaystyle\liminf_{n\to \infty}c_2\big(\mathbb{R}^{\binom{[n]}{2}}/S_n\big) = \infty$.
\end{theorem}

\begin{proof}
We will leverage \cref{thm.wednesday} to reduce to the point cloud setting, and then apply \cref{prop.unbounded growth} to conclude the result.

Put $W:=\mathbb{R}^{\binom{[n]}{2}}$, and let $V$ denote the subspace of all $f$ supported on a set of pairs $\{i,j\}$ with $i,j>n-3$.
Let $U\leq V^\perp$ denote the subspace of all $f$ supported on a set of pairs $\{i,j\}$ with $i\leq n-3<j$, and let $L\colon U\to(\mathbb{R}^3)^{[n-3]}$ denote the linear map defined by
\[
(Lf)(i)
=\left[\begin{array}{l}
f(\{i,n-2\})\\
f(\{i,n-1\})\\
f(\{i,n\}))
\end{array}\right].
\]

Next, let $G\leq\O(W)$ denote the image of the orthogonal representation $S_n\curvearrowright W$.
By inspecting a generic member of $V$, it follows that the pointwise stabilizer $G_V\leq G$ of $V$ corresponds to the subgroup $S_{n-3}$ that fixes $n-2$, $n-1$, and $n$.
Notably, $U$ is invariant under $G_V$, and $L$ is a surjective $S_{n-3}$-equivariant isometry.
By the second equivariant embedding lemma (\cref{lem:quotient-distortion}), it follows that $U/G_V$ and $(\mathbb{R}^3)^{[n-3]}/S_{n-3}$ are isomorphic as metric spaces.
Then \cref{thm.wednesday} gives
\[
c_2\big(\mathbb{R}^{\binom{[n]}{2}}/S_n\big)
=c_2(W/G)
\geq c_2(V^\perp/G_V)
\geq c_2(U/G_V)
=c_2\big((\mathbb{R}^3)^{[n-3]}/S_{n-3}\big),
\]
and the result follows from \cref{prop.unbounded growth}.
\end{proof}

Note that in the setting where the vertices are also weighted, we obtain the metric space $\mathbb{R}^{\binom{[n]}{\leq 2}}/S_n$, of which $\mathbb{R}^{\binom{[n]}{2}}/S_n$ is a sub-metric space, and so the Euclidean distortion is also unbounded.



Finally, we consider the set of \textit{unlabeled real databases} with $m$ rows and $n$ columns.
Once we arbitrarily label the rows and columns by $[m]$ and $[n]$, respectively, we obtain a matrix $A\in\mathbb{R}^{m\times n}$.
In doing so, we also introduce an ambiguity by the action of $S_m\times S_n$ defined by $((\pi,\tau)\cdot A)_{i,j}=A_{\pi^{-1}(i),\tau^{-1}(j)}$.
Much like the previous examples in this section, the resulting orbit spaces have unbounded Euclidean distortions.

\begin{theorem}
For each $m\geq 3$, it holds that $\displaystyle\liminf_{n\to \infty}c_2\big(\mathbb{R}^{m\times n}/(S_m\times S_n)\big) = \infty$.
\end{theorem}

\begin{proof}
Take $W:=\mathbb{R}^{m\times n}$, and let $G\leq\O(W)$ denote the image of the orthogonal representation $S_m\times S_n\curvearrowright W$.
Let $V\leq W$ denote the subspace of matrices whose first $n-1$ columns are zero.
By inspecting a generic member of $V$, it holds that the pointwise stabilizer $G_V\leq G$ of $V$ corresponds to the subgroup $\{\operatorname{id}\}\times S_{n-1}$.
Then \cref{thm.wednesday} gives
\[
c_2\big(\mathbb{R}^{m\times n}/(S_m\times S_n)\big)
=c_2(W/G)
\geq c_2(V^\perp/G_V)
=c_2\big((\mathbb{R}^m)^{[n-1]}/S_{n-1}\big),
\]
and the result follows from \cref{prop.unbounded growth}.
\end{proof}


\section{Discussion}   
\label{sec.discussion} 
\subsection{Bilipschitz invariants from polynomial invariants}

In this paper, we presented several bilipschitz embeddings of quotient spaces, but we have yet to discuss one important aspect of our approach.
(This aspect is perhaps more philosophical in nature, but we found it particularly useful, so we document it here for the sake of scaffolding.)
In what follows, we present a chronological account of how we used \textit{polynomial invariants} to discover the bilipschitz embeddings in \cref{sec.root unity,sec.so,sec.alternating}.

\begin{enumerate}
\item 
We started by observing that the optimal embedding of the quotient space $\mathbb R^{r\times n}/\O(r)$, given in \cref{prop.gram bounds}, can be viewed as a Lipschitz modification of a polynomial map (namely, $X\mapsto X^TX$), whose coordinate functions generate the algebra of $\O(r)$-invariant polynomials. 
Interestingly, when all eigenvalues of $X^TX$ are either $0$ or $1$, this modification reduces to $\sqrt{X^TX}=X^TX$, and so the optimal bilipschitz embedding restricts to a polynomial embedding over all such $X$.

\item 
Next, we worked to bilipschitzly embed the quotient space $\mathbb R^{r\times n}/\SO(r)$.
In pursuit of parallelism, we first recalled that the entries of the Gram matrix, together with the Pl\"ucker coordinates, generate the algebra of polynomial invariants in this setting. 
This led us to hunt for a bilipschitz modification of these invariants.
On the $\SO(r)$-invariant Stiefel manifold, where $XX^T = \Id_r$, we found that the Pl\"ucker coordinates themselves already provide a bilipschitz embedding, albeit with poor distortion. 
After some experimentation, we arrived at $X\mapsto\sigma_{\min}(X)\cdot \Plu(V_X)$ and proved \cref{thm.so psi lower bound}.
    
\item 
At this point, we were inspired to consider other orientation-preserving actions, so we turned to the index-$2$ alternating subgroup of a reflection group.
We started by recalling that the Weyl chamber embedding is optimal for the full reflection group. 
Our idea was to augment this with (a Lipschitz modification of) a polynomial that is invariant under the alternating group action but not under the full reflection group. 
This led us to the product $\prod_{i=1}^m\langle x,\alpha_i\rangle$, where $\alpha_1,\ldots,\alpha_m$ denote the normalized positive roots associated with a fixed Weyl chamber $C$. 
This polynomial behaves as desired since each member of the alternating subgroup flips an even number of signs in the product. 
Inspired by our minimum singular value modification of the Pl\"{u}cker coordinates, we modified this product to $\varepsilon\cdot\min_{i\in [m]}|\langle x,\alpha_i\rangle|$, where $\varepsilon = \det(w)$ for a reflection group element $w$ such that $wx\in C$. 
Geometrically, this corresponds to mapping $x$ to $wx\in C$, recording the determinant sign $\varepsilon$, and measuring the distance from $wx$ to the chamber walls. 
This interpretation led us to view the quotient space as the glued space $C\sqcup_{\partial C} C$.
    
\item 
We then hunted for bilipschitz embeddings of quotients by root-of-unity scalar actions, culminating in \cref{thm.root unity}. 
In this setting, the coordinate functions of the tensor power map $u\mapsto u^{\otimes r}$ generate the complex algebra of invariant polynomials. 
We started by showing that the normalized tensor power already yields a bilipschitz embedding, albeit with poor distortion. 
By augmenting with the optimal embedding $u\mapsto u\otimes \overline u/\|u\|$ for the space modulo the full circle action (see Example~17 and Corollary~37 in~\cite{CahillIM:24}), we achieved the Euclidean distortion after an appropriate scaling. 
Notably, this optimal embedding does not stem from a complex polynomial, but rather a real polynomial. 
More generally, this suggests that low-distortion bilipschitz embeddings may be more naturally derived from invariant polynomials involving both $z$ and $\overline z$, rather than from the usual complex polynomial invariants that avoid complex conjugation.
    
\item 
After all of this success in converting polynomial invariants into bilipschitz embeddings of quotient spaces, we hunted for a unifying through-line in our approach, and we eventually converged on the general theory of quotient--orbit embeddings presented in \cref{sec.reflection like}.
While this effectively captures the technical aspects of our constructions, it fails to convey their underlying polynomial inspiration.
We have yet to find a general approach to convert polynomial invariants into (optimal) bilipschitz invariants.
\end{enumerate}

\subsection{Open problems}

We conclude with several open problems. 
In \cref{sec.so,sec.alternating,sec.tran klien}, we established two-sided bounds on the Euclidean distortions of quotients by special orthogonal, alternating, and wallpaper group actions. 
What is the exact Euclidean distortion in these cases?
Note that for the wallpaper groups with orbifold signature $333$, $3{\ast}3$, $442$, or $632$, the quotient metric is unique up to a scale factor, and yet \cref{thm.wallpaper groups} fails to establish the Euclidean distortion.

In light of \cref{thm.upsilon-bound}, it would be interesting to compute the Euclidean contortion $\Upsilon(G)$ for groups beyond those handled in \cref{lem.upsilon compute}. 
For example, what are the values of $\Upsilon(C_4)$ and $\Upsilon(C_2\times C_2)$?
More ambitiously, if one computed the contortions of all finite simple groups, Theorem~\ref{thm.upsilon subnormal} would produce bounds on the contortions of all finite groups in terms of their simple subquotients.
Should we expect such contortion bounds to be relatively tight?
Or do there exist many groups with low contortion whose simple subquotients have large contortion?
On the asymptotic side, what is the infimum of $\alpha$ for which $\Upsilon(G)=O(|G|^\alpha)$ for all finite groups $G$?
See the end of \cref{sec.23-groups} for further discussion.

Finally, what is the minimum Hilbert space dimension needed to achieve the Euclidean distortion?
For example, the codomain in \cref{prop.gram bounds} can be restricted to the $\binom{n+1}{2}$-dimensional subspace of symmetric matrices, but can we decrease this dimension any further? 
In general, we conjecture that for any finite-dimensional Hilbert space $V$ and any isometry group $\Gamma\leq \E(V)$ with closed orbits, the orbit space $V/\Gamma$ achieves its Euclidean distortion in a finite-dimensional Hilbert space.

\section*{Acknowledgments}

BBS was partially supported by ONR N00014-22-1-2126 and NSF CCF 2212457.
HD was partially supported by NSF DMS 2147769.
DGM was partially supported by NSF DMS 2220304.

\addtocontents{toc}{\protect\setcounter{tocdepth}{1}}
\appendix

\section{Previous results in bilipschitz invariant theory}
\label{app.previous results}

We first collect two existing general tools for computing Euclidean distortion. 
The first regards the Euclidean distortion of a product space and is based on Lemma~39 in \cite{CahillIM:24}. 
While \cite{CahillIM:24} states the result in a slightly different setting, the proof is nearly identical. 
Even so, we include a proof.

\begin{proposition}[Euclidean distortion of product spaces]
\label{prop.max-product}
    Given nonempty metric spaces $(X,d_X)$ and $(Y,d_Y)$, if we endow $X\times Y$ with the metric
    \[
    d_{X\times Y}^2\big((x,y),(x',y')\big) 
    := d_X^2(x,x') + d_Y^2(y,y'),
    \]
    then it holds that
    \[
    c_2(X\times Y) 
    = \max\{c_2(X),c_2(Y)\}.
    \]
\end{proposition}

\begin{proof}
    Let $y\in Y$ be arbitrary. Then $X$ is isometric to the subspace $X\times \{y\} \subseteq X\times Y$, so $c_2(X)\leq c_2(X\times Y)$ and similarly for $Y$.

    For the other inequality, we may assume that $c_2(X)$ and $c_2(Y)$ are finite, since otherwise the inequality is trivial. Let $f_X\colon X\to H_X$ and $f_Y\colon Y \to H_Y$ be bilipschitz embeddings into Hilbert spaces. Without loss of generality, we may assume the optimal lower Lipschitz constants for both maps are 1 by rescaling if necessary. Use $\beta_X$ and $\beta_Y$ to denote the respective optimal upper Lipschitz constants. We need only show that the product map
    \[
    f_X \times f_Y
    \colon X\times Y \to H_X\oplus H_Y
    \]
    has optimal upper Lipschitz constant at most $\max\{\beta_1,\beta_2\}$. Indeed,
    \begin{align*}
        \|(f_X\times f_Y)(x,y) - (f_X\times f_Y)(x',y')\|_{H_X\oplus H_Y}^2
        & = \|f_X(x)-f_X(x')\|_{H_X}^2 + \|f_Y(y)-f_Y(y')\|_{H_Y}^2
        \\ & \leq
        \beta_X^2\cdot d_X^2(x,x') + \beta_Y^2\cdot d_Y^2(y,y')
        \\ & \leq \max\{\beta_X^2,\beta_Y^2\} \cdot d_{X\times Y}^2\big((x,y),(x',y')\big),
    \end{align*}
    as desired.
\end{proof}

The second general tool for computing Euclidean distortion allows one to pass from a metric space to the finite subspaces thereof. 
Note that the proof given in \cite{CahillIM:24} uses an ultraproduct construction and hence relies on some nonconstructive axiom such as the axiom of choice.

\begin{proposition}[Euclidean distortion is finitely determined, Proposition~31 \cite{CahillIM:24}]
\label{prop.finitely-determined}
    Let $X$ be a metric space. Then
    \[c_2(X) = \sup_{\substack{B\subseteq X\\|B|<\infty}} c_2(B).\]
\end{proposition}

In addition to the above tools, we also use previously established Euclidean distortions of a few spaces. 
Some of these results are mentioned in Table~\ref{table1}, but we collect them here as well for easy reference.

\begin{proposition}\label{prop.known-distortions}
The following statements hold:
    \begin{enumerate}[label=(\alph*)]
        \item\label{it.known_distortions.pm_identity} \emph{(\cite[Corollary 36]{CahillIM:24})} Let $V$ be a real Hilbert space with $\dim(V)\geq 2$ and set $G:= \{\pm \Id_V\}$. Then $c_2(V/G) = \sqrt{2}$.
        \item\label{it.known_distortions.circle_action} \emph{(\cite[Corollary 37]{CahillIM:24})} Let $V$ be a complex Hilbert space with $\dim(V)\geq 2$ and set $G:= \{z\in\mathbb C:|z|=1\}$. Then $c_2(V/G) = \sqrt{2}$.
        \item\label{it.known_distortions.roots_of_unity} \emph{(\cite[Corollary 38]{CahillIM:24})} Let $G := \langle e^{2\pi i/r}\rangle \leq U(1)$. Then $c_2(\mathbb C/G) = r\sin\left(\frac{\pi}{2r}\right)$.
        \item\label{it.known_distortions.circle}\emph{(\cite[Theorem 6.1]{HeimendahlLVZ:22})} For each $c>0$, $c_2(\mathbb R/c\mathbb Z) = \frac{\pi}{2}$.
    \end{enumerate}
\end{proposition}

\section{Bilipschitz embedding of a compact quotient}
\label{app.finite distortion}
This section presents a streamlined proof of the following result:
\begin{proposition}
    \label{prop.finite distortion}
    Let $M$ be a compact Riemannian manifold equipped with an isometric action by a compact group $G$. Then $c_2(M/G)<\infty$. 
\end{proposition}

In particular, if $V$ is a finite-dimensional real inner product space with unit sphere $S$ and $G\leq \Oname(V)$ is compact, then $c_2(S/G)<\infty$, and so $c_2(V/G)<\infty$ by Section~4 of~\cite{CahillIM:24}.

The first key tool to the proof of \cref{prop.finite distortion} is the following \emph{gluing theorem}, which reduces the embedding problem to local neighborhoods. It is an immediate consequence of Theorem~1.1 in~\cite{MakarychevM:16}.
\begin{proposition}[Gluing theorem]
    \label{prop.gluing theorem}
    Let $X$ be a metric space, and let $A,B\subseteq X$ be subspaces. Then
    \[c_2(A\cup B)\leq 11\cdot c_2(A)\cdot c_2(B).\]
\end{proposition}

We will also use the fact that bilipschitz behavior is preserved under diffeomorphic changes over compact sets:

\begin{proposition}[Diffeomorphisms are bilipschitz over compact sets]
    \label{prop.diff local bilip}
    Let $f\colon N_1\to N_2$ be a diffeomorphism between Riemannian manifolds. Then for every relatively compact subspace $X\subseteq N_1$, the restriction $f|_{X}$ is bilipschitz.
\end{proposition}
\begin{proof}
    Suppose for contradiction that $f|_{X}$ is not bilipschitz for some relatively compact subspace $X\subseteq N_1$. Then there exist sequences $(x_k)_{k=1}^\infty$, $(y_k)_{k=1}^\infty$ in $X$ such that $x_k\neq y_k$ and
    \begin{equation}
    \label{eq.bilip zero}
    \frac{d_{N_2}(f(x_k),f(y_k))}{d_{N_1}(x_k,y_k)}\to 0.
    \end{equation}
    By replacing $X$ with its closure if necessary, we may assume without loss of generality that $X$ is compact. Thus, we may pass to subsequences and assume $x_k\to x$ and $y_k\to y$  for some $x,y\in X$.
    
    If $x\neq y$, then \eqref{eq.bilip zero} implies $f(x)=f(y)$, contradicting the injectivity of $f$. Meanwhile, if $x=y$, then $f$ fails to be locally bilipschitz at $x$, which contradicts the fact that diffeomorphisms of Riemannian manifolds are locally bilipschitz.
\end{proof}

We will also make essential use of a well-behaved $G$-equivariant diffeomorphism provided by the \emph{slice theorem}. In what follows, recall the notion of an \emph{isotropy representation} introduced at the beginning of \cref{sec.isotropy}. The following result commonly referred to as the \emph{slice theorem} or the \emph{tube theorem} is a reformulation of the main content of Section~3.2 in~\cite{AlexandrinoB:15}, adapted to the setting of an isometric action by a compact Lie group. 
\begin{proposition}[Slice theorem]
    \label{prop.slice theorem}
    Let $M$ be a compact Riemannian manifold equipped with an isometric action by a compact Lie group $G$, where $G$ is endowed with a bi-invariant Riemannian metric. Fix a point $x\in M$, let $G_x$ denote its (compact) pointwise stabilizer, and let $N_x$ denote the orthogonal complement of $T_x(G\cdot x)$ in $T_xM$. Then there exists $R>0$ such that for any open ball $B\subseteq N_x$ of radius less than $R$, it holds that
    \begin{itemize}
        \item[(i)] $B$ is $G_x$-invariant under the isotropy representation,
        \item[(ii)] $G\cdot \exp_x(B)$ is open in $M$,
        \item[(iii)] $G_x$ acts freely and isometrically on the product Riemannian manifold $G\times B$ via the action
        \[h\cdot (g,y):= (gh^{-1},hy),\]
        and
        \item[(iv)] letting $G\times_{G_x} B := (G\times B)/G_x$ denote the corresponding orbit space, equipped with the quotient Riemannian metric, and giving it an isometric left $G$-action via 
        \[k\cdot [(g,y)]:= [(kg,y)],\]
        it holds that the map 
    \[G\times_{G_x} B\to G\cdot \exp_x(B), \quad [(g,y)]\mapsto g\cdot \exp_x(y)\]
    is a $G$-equivariant diffeomorphism.
    \end{itemize}
\end{proposition}

With this setup, we are ready to give a proof of \cref{prop.finite distortion}.

\begin{proof}[Proof of \cref{prop.finite distortion}]
    We proceed by induction on the dimension of $M$. For the base case $\dim(M)=0$, the quotient $M/G$ is a finite metric space, and any injective embedding into a Hilbert space is automatically bilipschitz.
    
    Now assume that the result holds for all compact Riemannian manifolds of dimension less than $n$, and suppose that $\dim(M) = n$. By replacing $G$ with its compact image in the isometry group of $M$ if necessary, we may assume without loss of generality that $G$ is a compact Lie group, equipped with a bi-invariant Riemannian metric.
    
    Using the gluing theorem (\cref{prop.gluing theorem}) and the compactness of $M/G$, it suffices to fix $x\in M$ and to find a $G$-stable neighborhood $U$ of $x$ such that $c_2(U/G)<\infty$. Let $B$ be as given in the slice theorem (\cref{prop.slice theorem}). By \cref{prop.diff local bilip}, the slice theorem, and the second equivariant embedding lemma (\cref{lem:quotient-distortion}), it suffices to show that $c_2((G\times_{G_x} B)/G) < \infty$. Note that
    \[
    \begin{aligned}
    &d_{(G\times_{G_x} B)/G}(G\cdot[(g_1,y_1)] ,G\cdot[(g_2,y_2)])^2\\
    &=\min_{k_1,k_2\in G}d_{G\times_{G_x} B}(k_1\cdot[(g_1,y_1)] ,k_2\cdot[(g_2,y_2)])^2 && \text{(def.\ of $d_{(G\times_{G_x}B)/G}$)}\\
    &= \min_{k_1,k_2\in G}d_{G\times_{G_x} B}([(k_1g_1,y_1)] ,[(k_2g_2,y_2)])^2 && \text{(def.\ of $G$-action on $G\times_{G_x}B$)}\\
    &= \min_{k_1,k_2\in G}d_{G\times_{G_x} B}([(k_1,y_1)] ,[(k_2,y_2)])^2 && \text{(change of variables)}\\
    &= \min_{k_1,k_2\in G}\min_{h_1,h_2\in G_x}d_{G\times B}((k_1h_1^{-1},h_1y_1) ,(k_2h_2^{-1},h_2y_2))^2 && \text{(def.\ of $d_{G\times_{G_x}B}$)}\\
    &= \min_{h_1,h_2\in G_x} \min_{k_1,k_2\in G} d_{G\times B}((k_1h_1^{-1},h_1y_1) ,(k_2h_2^{-1},h_2y_2))^2 && \text{(switch minima)}\\
    &= \min_{h_1,h_2\in G_x} \Big(\min_{k_1,k_2\in G} d_G(k_1h_1^{-1},k_2h_2^{-1})^2 + d_B(h_1y_1,h_2y_2)^2\Big) && \text{(product metric)}\\
    &= \min_{h_1,h_2\in G_x} d_B(h_1y_1,h_2y_2)^2 && \text{(inner $\min=0$ via $k_i = h_i$)}\\
    &= d_{B/G_x}(G_x\cdot y_1,G_x\cdot y_2)^2 && \text{(def.\ of $d_{B/G_x}$).}
    \end{aligned}
    \]
    As such, $(G\times_{G_x} B)/G$ is isomorphic, as a metric space, to $B/G_x\subseteq N_x/G_x$, where the isometry is induced by the well-defined $G$-invariant surjection 
    \[G\times_{G_x}B\to B/G_x, \quad [(g,y)]\mapsto G_x\cdot y.\]
    Thus, it suffices to show $c_2(N_x/G_x)<\infty$. By \cite[Section~4]{CahillIM:24}, it further suffices to show that $c_2(S/G_x)<\infty$, where $S$ denotes that unit sphere in $N_x$. The result now follows by applying the induction hypothesis to $S$, whose dimension is at most $\dim(M)-1$.
\end{proof}

\section{Bochner spaces over Radon measures}
\label{app.bochner}

This section is primarily devoted to the proof of \cref{prop.hilbert L2}. 
While the definition of $L^2(G,H)$ given therein appears natural as it extends the familiar space when $H=\mathbb R$, the situation is more nuanced than it first appears.

Given a measure space $(Y,\mathcal A,\mu)$ and a Hilbert space $H$, consider the space of Borel-measurable functions $f\colon Y\to H$ such that $\int_Y\|f(y)\|_H^2\, d\mu(y) < \infty$. 
Identifying functions that agree $\mu$-almost everywhere, one might hope this space forms a vector space. 
Indeed, this identification recovers $L^2(Y,\mathbb R)$ when $H=\mathbb R$. 
However, for general $Y$ and $H$, the space may fail to be closed under addition. 
An explicit example is given in \cref{prop.fail vector} at the end of this section, where $(Y,\mathcal A,\mu)$ is a probability space.

To address this issue, we restrict attention to \textit{strongly measurable} functions. 
For a general Banach space $B$, this yields the well-defined \textit{Bochner space} $L^p(Y,B)$, which is a Banach space under the natural norm. 
Fortunately, when $\mu$ is a Radon measure (e.g., Haar measure on a locally compact group), Borel and strong measurability coincide for $p$-norm-integrable functions, thereby validating the \textit{a priori} natural definition in \cref{prop.hilbert L2}.

This section is divided into four subsections. 
The first two are devoted to the proof of \cref{prop.hilbert L2}, while the latter two offer further context and elaboration on related subtleties. 
In~\ref{app.sub.general def}, we consider a measure space $(Y,\mathcal A,\mu)$ and a Banach space $B$, and formally define key concepts such as strong measurability and the Bochner spaces $L^p(Y,B)$. 
We observe that, in the Radon setting, Borel and strong measurability coincide for $p$-norm-integrable functions, a fact that allows for an immediate proof of \cref{prop.hilbert L2}(a), as Haar measures are Radon. 
In~\ref{app.sub.proof b}, we prove \cref{prop.hilbert L2}(b) by leveraging the density of compactly supported continuous functions in $L^p(Y,B)$, under the assumption that $Y$ is a locally compact Hausdorff space and $\mu$ is a Radon measure. 

In~\ref{app.sub.dimension}, we provide additional insight into the structure of $L^2(G,H)$ by examining its Hilbert space dimension when $G$ is a locally compact group endowed with Haar measure and $H$ is a Hilbert space. 
Finally, in~\ref{app.sub.note noncompact}, we discuss the limitations of elementary set theory in determining whether the distinction between strong and Borel measurability remains necessary in the $\sigma$-finite non-Radon case.

\subsection{General definitions and proof of \cref{prop.hilbert L2}(a)}
\label{app.sub.general def}

In this subsection, we introduce relevant terminology, review foundational results, and then provide a proof of \cref{prop.hilbert L2}(a). 

\begin{definition}
Let $(Y,\mathcal A,\mu)$ be a measure space and $B$ a Banach space. 
A Borel-measurable function $f\colon Y\to B$ is said to be
\begin{itemize}
    \item[(a)] \textbf{$\mu$-simple} if it is of the form
\[
f(x)
=\sum_{i=1}^n x_i\cdot 1_{x\in A_i},
\]
where $x_i \in B$ and $A_i \in \mathcal A$ satisfy $\mu(A_i) < \infty$,
    \item[(b)] \textbf{$\mu$-strongly measurable} if there exists a sequence $(f_n)_{n=1}^\infty$ of $\mu$-simple functions such that $f_n\to f$ pointwise $\mu$-almost everywhere,
    \item[(c)] \textbf{$\mu$-essentially separably valued} if there exists a $\mu$-null set $N\subseteq Y$ such that $f(Y\setminus N)$ is separable, and
    \item[(d)] \textbf{$p$-norm-integrable}, with $p\in [1,\infty)$, if
    \[\int_Y \|f(y)\|_B^p\, d\mu < \infty.\]
\end{itemize}
In addition, we say that $f$ has \textbf{$\sigma$-finite support} if the measurable set $f^{-1}(B\setminus \{0\})$ is $\sigma$-finite with respect to $\mu$.
\end{definition}

Note that every $p$-norm-integrable function $f\colon Y\to B$ has $\sigma$-finite support. 
Indeed, the sets $Y_n:=\{y\in Y:\|f(y)\|_B^p>1/n\}$ form a countable cover of $f^{-1}(B\setminus\{0\})$, and each $Y_n$ has finite measure by Markov's inequality.

Next, we recall the notion of a Radon measure, which plays a central role in what follows.

\begin{definition}
    Let $Y$ be a Hausdorff topological space, and let $\mathcal A$ denote its Borel $\sigma$-algebra. 
    A measure $\mu$ on $(Y,\mathcal A)$ is called a \textbf{Radon measure} if it satisfies the following:
    \begin{itemize}
        \item[(a)] \textit{Locally finite}: For every $y\in Y$, there exists an open neighborhood $U$ of $y$ such that $\mu(U)<\infty$, and
        \item[(b)] \textit{Inner regular}: For all $A\in \mathcal A$,
            \[
            \mu(A) 
            = \sup\big\{\mu(K): K\subseteq A\ \text{compact}\big\}.
            \]
    \end{itemize}
\end{definition}

Note that local finiteness implies that $\mu$ is finite on compact sets. 
The converse holds when $Y$ is additionally locally compact. 
Also, every Haar measure on a locally compact group is a Radon measure; see~\cite[Thm.~8.12]{Salamon:16}. 

The following proposition summarizes several results from~\cite[Chapter~1]{HytonenNVW:16}.

\begin{proposition}
    \label{prop.banach L2}
    Let $(Y,\mathcal A,\mu)$ be a measure space, $B$ a Banach space, and $f\colon Y\to B$ a Borel-measurable function. 
    Then the following are equivalent:
    \begin{enumerate}[label=(\roman*)]
        \item $f$ is $\mu$-strongly measurable;
        \item $f$ is $\mu$-essentially separably valued and has $\sigma$-finite support.
    \end{enumerate}
    Furthermore, for $p\in[1,\infty)$, let $L^p(Y,B)$ denote the space of all $p$-norm-integrable, $\mu$-strongly measurable functions $f\colon Y\to B$, identified up to $\mu$-almost everywhere equality. 
    Then $L^p(Y,B)$ is a Banach space under the norm
    \[
    \|f\|_{L^p(Y,B)} 
    := \left(\int_Y\big\|f(y)\big\|^p_B\, d\mu(y)\right)^{1/p},
    \]
    and the $\mu$-simple functions form a dense subset.
\end{proposition}

The space $L^p(Y,B)$ is called a \textbf{Bochner space}. 
For $\mu$-strongly measurable functions $f,g\in L^p(Y,B)$, the sum $f+g$ need not be Borel-measurable. 
However, by \cref{prop.banach L2}(ii), it agrees $\mu$-almost everywhere with a Borel-measurable function, which we take to represent the sum in $L^p(Y,B)$; see~\cite[Prop.~2.6]{Folland:99}.

If $Y$ is a locally compact Hausdorff topological space equipped with a Radon measure, then the space $\mathcal C_c(Y,B)$ of compactly supported continuous functions from $Y$ to $B$ is dense in $L^p(Y,B)$. 
This follows from approximating $\mu$-simple functions, which reduces to the familiar case $B=\mathbb R$; see, for example,~\cite[Thm.~4.15]{Salamon:16}.

By definition, every $\mu$-strongly measurable function is Borel-measurable. 
In the special case of functions with $\sigma$-finite support over Radon measure spaces, the converse also holds.

\begin{proposition}
    \label{prop.compact measure}
    Let $(Y,\mathcal A,\mu)$ be a Radon measure space and $B$ a Banach space. 
    Then every Borel-measurable function $f\colon Y\to B$ with $\sigma$-finite support is $\mu$-essentially separably valued.
\end{proposition}

\begin{proof}
    First, observe that the restriction of $\mu$ to the measurable subspace $f^{-1}(B\setminus \{0\})\in\mathcal A$ remains a Radon measure. 
    Indeed, the subspace $f^{-1}(B\setminus \{0\})$ is Hausdorff, the restriction of $\mathcal A$ to this subspace coincides with its Borel $\sigma$-algebra (which remains in $\mathcal A$ by measurability), the restricted measure remains locally finite, and every compact subset of $f^{-1}(B\setminus \{0\})$ is a compact subset of $Y$.
    Thus, we may assume without loss of generality that $Y$ is $\sigma$-finite. 
    Let $(Y,\hat{\mathcal A},\hat\mu)$ denote the completion of $\mu$. 
    Since every $\hat \mu$-null set is contained in a $\mu$-null set, it suffices to show $f$ is $\hat \mu$-essentially separably valued.
    
    Note that $\hat \mu$ remains locally finite, as $\mathcal A\subseteq \hat {\mathcal A}$. 
    Moreover, $\hat \mu$ is inner regular: each $F\in \hat{\mathcal A}$ has the form $F=E\cup N$ where $E\in\mathcal A$ and $\hat\mu(N)=0$, so it suffices to approximate $\mu(E)$ from below by compact sets. 
    It then follows from~\cite[Prop.~213H(a)]{Fremlin:00} that $\hat \mu$ is a Radon measure in the sense of \cite[Def.~411H(b)]{Fremlin:00} (on $\sigma$-finite spaces, this definition requires a Radon measure to be complete). 
    The conclusion now follows by combining Proposition~416A, Proposition~418G and Theorem~451S from~\cite{Fremlin:00}. 
\end{proof}

With this, \cref{prop.hilbert L2}(a) follows immediately by combining \cref{prop.banach L2,prop.compact measure}, along with the facts that every Haar measure over a locally compact group is Radon and that every $p$-norm-integrable function has $\sigma$-finite support.

\subsection{Proof of \cref{prop.hilbert L2}(b)}
\label{app.sub.proof b}

We now prove \cref{prop.hilbert L2}(b). 
Let $H$ be a Hilbert space, and equip $L^2(G,H)$ with the linear left action of $G$ defined by
\[
(g\cdot f)(h)
= f(hg).
\]
Through this action, each $g\in G$ maps Borel-measurable functions to Borel-measurable functions.
Furthermore, the action is isometric and well defined on $L^2(G,H)$. 
Indeed, for all $g\in G$ and $f\in L^2(G,H)$, the right-invariance of $\mu$ yields
    \begin{align*}
        \|g\cdot f\|_{L^2(G,H)}^2
         =\int_G\|f(hg) \|_H^2\, d\mu(h)
         =\int_G \|f(h)\|_H^2\, d\mu(h) 
         = \|f\|_{L^2(G,H)}^2.
    \end{align*}
To verify strong continuity, observe first that, by \cref{prop.banach L2,prop.compact measure}, the space of $\mu$-simple functions is dense in $L^2(G,H)$. 
Since $G$ is locally compact Hausdorff and its Haar measure is Radon, every such function can be approximated in norm by compactly supported continuous functions; e.g., see~\cite[Thm.~4.15]{Salamon:16}. 
It follows that the space $\mathcal C_c(G,H)$ of compactly supported continuous functions from $G$ to $H$ is dense in $L^2(G,H)$. 
For any $f\in \mathcal C_c(G,H)$ and any sequence $g_n\to g$ in $G$, the dominated convergence theorem yields
    \begin{equation}
        \label{eq.DCT}
        \|g_n\cdot f - g\cdot f\|_{L^2(G,H)}^2 
        = \int_{G} \|f(hg_n)-f(hg)\|_{H}^2\, d\mu(h) 
        \to 0.
    \end{equation}
Now fix any $f\in L^2(G,H)$, $\varepsilon>0$, and a sequence $g_n\to g$ in $G$. 
By density, we may choose $k\in \mathcal C_c(G,H)$ with $\|f-k\|_{L^2(G,H)}<\frac{\varepsilon}{3}$. 
Then by \eqref{eq.DCT}, there exists $N> 0$ such that for all $n\geq N$, we have $\|g_n\cdot k-g\cdot k\|_{L^2(G,H)}<\frac{\varepsilon}{3}$, and so
    \[
    \|g_n\cdot f-g\cdot f\|_{L^2(G,H)}
    \leq 2\|f-k\|_{L^2(G,H)}+\|g_n\cdot k - g\cdot k\|_{L^2(G,H)} 
    < \varepsilon.
    \]
Since $\varepsilon>0$ was arbitrary, it follows that $g_n\cdot f\to g\cdot f$ in norm, as required.

\subsection{Note on dimension of \texorpdfstring{$L^2(G,H)$}{L2(G,H)}}
\label{app.sub.dimension}

We briefly remark on the dimension of $L^2(G,H)$, where $G$ is a locally compact group equipped with Haar measure and $H$ is a Hilbert space. 
If $G$ is finite and $H$ is finite dimensional, then $L^2(G,H)\cong H^{|G|}$ is also finite dimensional. 
Meanwhile, if $G$ is second countable and $H$ is separable, then the Borel $\sigma$-algebra of $G$ is countably generated, and the separability of $L^2(G,H)$ follows from \cite[Prop.~1.2.29]{HytonenNVW:16}.

More generally, using the axiom of choice, let $\mathfrak m(G)$ denote the minimal cardinality of a topological basis for $G$, and let $\dim(H)$ denote the Hilbert space dimension of $H$, i.e., the cardinality of any orthonormal basis. 
Then
\[
\dim(L^2(G,H)) 
= \mathfrak m(G)\cdot \dim(H).
\]
To see this, let $(e_i)_{i\in I}$ be an orthonormal basis for $H$. 
Then $L^2(G,H)$ decomposes as an orthogonal direct sum of strongly continuous unitary representations:
\[
L^2(G,H)
\cong \bigoplus_{i\in I} L^2(G, \mathbb Re_i).
\]
By Theorem~2 in~\cite{Vries:78}, each summand $L^2(G,\mathbb Re_i)$ has dimension $\mathfrak m (G)$, so the claimed formula follows.

\subsection{Note on non-Radon measures and the Banach--Ulam problem}
\label{app.sub.note noncompact}

In light of \cref{prop.hilbert L2,prop.banach L2,prop.compact measure}, and noting that every $p$-norm-integrable function has $\sigma$-finite support, we conclude with a remark on the equivalence between Borel and $\mu$-strong measurability in the setting of $\sigma$-finite measure spaces, which may not be Radon. 
According to \cref{prop.banach L2}, $\mu$-strongly measurability in this context is equivalent to being $\mu$-essentially separably valued. 
With this in mind, consider the following statement:
\begin{itemize}
\item[\textbf{($\mathsf{E}$)}]
For every $\sigma$-finite measure space $(Y,\mathcal A,\mu)$, Banach space $B$, and Borel-measurable function $f\colon Y\to B$, it holds that $f$ is $\mu$-strongly measurable.
\end{itemize}
We show this is equivalent to the following deep statement in set-theoretic measure theory:
\begin{itemize}
\item[\textbf{(${\mathsf{BU}}$)}]
There does not exist a discrete probability space $(X,2^X,\mu)$ such that $\mu(\{x\}) = 0$ for every $x\in X$.
\end{itemize}
The presence for a proof of \textbf{(${\mathsf{BU}}$)} in Zermelo--Fraenkel set theory with the Axiom of Choice ($\mathsf{ZFC}$) is commonly referred to as the \emph{Banach--Ulam problem}. 
While it is believed that \textbf{(${\mathsf{BU}}$)} is independent of $\mathsf{ZFC}$ \cite[Sec.~363S]{Fremlin:00}, the following can be stated with certainty: the consistency of $\mathsf{ZFC}$ implies the consistency of $\mathsf{ZFC}+\textbf{(${\mathsf{BU}}$)}$. 
Thus, the failure of $\textbf{(${\mathsf{BU}}$)}$ cannot be proved within $\mathsf{ZFC}$ itself~\cite[Sec.~2.1.6]{Federer:14}. 
The next proposition establishes that \textbf{(${\mathsf{BU}}$)} is equivalent to \textbf{(${\mathsf{E}}$)}. 
Hence, the failure of \textbf{(${\mathsf{E}}$)} is not provable within $\mathsf{ZFC}$, and it is plausible that \textbf{(${\mathsf{E}}$)} is independent of $\mathsf{ZFC}$.

\begin{proposition}
\label{prop.equiv E UC}
    The following equivalence holds:
    \[
    \normalfont \textbf{(${\mathsf{E}}$)} 
    \iff \textbf{(${\mathsf{BU}}$)}.
    \]
\end{proposition}

\begin{proof}
\textbf{(${\mathsf{BU}}$)}$\Rightarrow$\textbf{(${\mathsf{E}}$)} was established in~\cite[Prop.~438D]{Fremlin:00}.
Now suppose the negation of \textbf{(${\mathsf{BU}}$)} holds, that is, there exists a probability space $(X,2^X,\mu)$ such that $\mu(\{x\})=0$ for every $x\in X$. 
Let $(e_i)_{i\in X}$ be the standard orthonormal basis of the Hilbert space $\ell^2(X)$, and define a function $f\colon X\to \ell^2(X)$ by $f(w):= e_w$. 
While $f$ is Borel-measurable, it is not $\mu$-strongly measurable: for any subset $A\subseteq X$, the image $f(A)$ is separable if and only if $A$ is countable, in which case $\mu(A) = 0 < 1$.
\end{proof}
    
In fact, assuming the {continuum hypothesis} and the negation of {\normalfont\textbf{(${\mathsf{BU}}$)}}, the following proposition demonstrates why strong measurability is required in \cref{prop.banach L2} to ensure that addition in Bochner spaces is well-defined.

\begin{proposition}
\label{prop.fail vector}
    Assuming the continuum hypothesis and the negation of {\normalfont\textbf{(${\mathsf{BU}}$)}}, there exists a probability space $(Y,\mathcal A,\mu)$, a Hilbert space $H$, and Borel-measurable functions $f_1,f_2\colon Y\to H$ such that $\|f_i(y)\|_H = 1$ for each $y\in Y$, but $f_1+f_2$ is not $\mu$-almost everywhere equal to any Borel-measurable function.
\end{proposition}

\begin{proof}
By assumption, there exists a probability space $(X,2^X,\mu)$ such that $\mu(\{x\})=0$ for every $x\in X$. 
By Theorem~438C in~\cite{Fremlin:00} and the continuum hypothesis, we may take $X$ to be an \textit{atom}, that is, $\mu(A)\in\{0,1\}$ for every subset $A\subseteq X$.

Define $(Y, \mathcal A,\mu^{\otimes 2})$ to be the product probability space where $Y:= X\times X$ and $\mathcal A:= 2^X\otimes 2^X$. 
Put $H:=\ell^2(X)$ and denote its Borel $\sigma$-algebra by $\mathcal B$. 
Let $(e_x)_{x\in X}$ be the standard orthonormal basis of $H$, and define
\[
f_1(x,x')
:= e_x
\qquad
\text{and}
\qquad
f_2(x,x')
:= -e_{x'}.
\] 
Then $f_1$ and $f_2$ are $(\mathcal A,\mathcal B)$-measurable functions satisfying $\|f_i(y)\|_H=1$ for each $y\in Y$, and we have $\Delta = (f_1+f_2)^{-1}(0)$. 
If $f_1+f_2$ is $\mu$-almost everywhere equal to an $(\mathcal A,\mathcal B)$-measurable function, then $f_1+f_2$ is $(\mathcal A^*,\mathcal B)$-measurable, where $\mathcal A^*$ is the completion of $\mathcal A$. 
As such, it suffices to show $\Delta$ is not $\mathcal A^*$-measurable, or equivalently, that its inner and outer measures do not coincide.

To state this aim more precisely, recall that a \emph{rectangle} in $Y$ is a set of the form $C\times D$ where $C,D\subseteq X$, and let $\mathcal R_f\subseteq \mathcal A$ (resp.\ $\mathcal R_c\subseteq \mathcal A$) denote the algebra (resp.\ collection) of finite (resp.\ countable) disjoint unions of rectangles in $Y$. 
Then we wish to establish the inequality
\begin{equation}
\label{eq.inner outer}
\sup\Big\{\mu^{\otimes 2}(A): A\in \mathcal R_c, \ A\subseteq \Delta\Big\}
< \inf\left\{\sum_{i=1}^\infty\mu^{\otimes 2}(R_i): R_i\in \mathcal R_f,\, \Delta\subseteq \cup_{i=1}^\infty R_i \right\}.
\end{equation}
Since every rectangle within $\Delta$ is a singleton and $\mu^{\otimes 2}$ is countably additive and vanishing on singletons, the left-hand side of \eqref{eq.inner outer} is zero. 
On the other hand, let $K$ denote its right-hand side. 
The following shows that $K$ is nonzero:
\begin{align*}
    K 
    &= \inf\left\{\sum_{i=1}^\infty\mu(C_i)\mu(D_i): C_i,D_i\subseteq X,\, \Delta\subseteq \cup_{i=1}^\infty \big(C_i\times D_i\big) \right\} 
    &&\text{(definition of $\mathcal R_f$)}\\
    &=\inf\left\{\sum_{i=1}^\infty\mu(C_i)^2: C_i\subseteq X,\, \Delta\subseteq \cup_{i=1}^\infty \big(C_i\times C_i\big) \right\} 
    &&\text{(replace $C_i,D_i$ with $C_i\cap D_i$)}\\
    &=\inf\left\{\sum_{i=1}^\infty\mu(C_i)^2: C_i\subseteq X,\, X= \cup_{i=1}^\infty  C_i \right\} 
    &&\text{(definition of $\Delta$)}\\
    &= 1 
    &&\text{($X$ is an atom).} 
    \tag*{\qedhere}
\end{align*}
\end{proof}

\section{Extensions by subspaces of translations are split}
\label{app.split}

In this appendix, we use group cohomology to prove that in the situation of \cref{thm.euclidean-subgroup}, the extension of groups $0\rightarrow T\rightarrow \Gamma\rightarrow G\rightarrow 1$ is always split. 
Note that $\pi|_{\Gamma}\colon \Gamma \to G$ is a Lie group homomorphism since $\Gamma$ and $G$ are closed embedded Lie subgroups of $\E(V)$ and $\Oname(V)$, respectively.

\begin{proposition}\label{prop.extension-splits}
Under the hypotheses of \cref{thm.euclidean-subgroup}, the following short exact sequence of Lie group homomorphisms splits:
\[
0
\rightarrow T
\hookrightarrow \Gamma
\xrightarrow{\pi|_\Gamma} G
\rightarrow 1.
\]
In other words, there exists a Lie group homomorphism $\varphi\colon G\to \Gamma$ such that $\pi|_\Gamma\circ \varphi = \mathrm{id}_G$, whereupon it follows that $\Gamma = T\rtimes \varphi(G)$ is an internal semidirect product. 
Furthermore, there exists $p\in V$ such that $\varphi(g)\cdot p = p$, for each $g\in G$.
\end{proposition}

We provide two proofs for this proposition. 
The first is group cohomology--based while the second avoids group cohomology, at least explicitly. 
Recall that for a topological group $G$, a continuous $G$-module $M$ is a topological abelian group equipped with an action of $G$ such that the action map $G\times M\rightarrow M$ is continuous. 
Following \cite[Section~3]{Hu:52}, the continuous group cohomology of $G$ with coefficients in $M$ is defined as the cohomology $H_c^\bullet(G,M)$ of the cochain complex $\{\operatorname{C}(G^n,M)\}_{n=0}^\infty$, where $\operatorname{C}(G^n,M)$ is the $\mathbb Z$-module comprised of continuous maps from $G^n$ to $M$ and the coboundary map $d \colon \operatorname{C}(G^n,M)\to \operatorname{C}(G^{n+1},M)$ is defined by the usual formula
\[
\begin{aligned}
(d\varphi)(g_1,\dots,g_{n+1}) 
:= g_1\cdot \varphi(g_2,\dots,g_{n+1})
+\sum_{i=1}^n (-1)^i \varphi(g_1,&\dots,g_ig_{i+1},\dots,g_{n+1})\\
&+ (-1)^{n+1}\varphi(g_1,\dots,g_{n}).
\end{aligned}
\]
Specializing to $n\in\{0,1\}$, one sees that the $1$-cocycles are the \textit{continuous crossed homomorphisms}, i.e., continuous maps $b\colon G\to M$ that satisfy $b(gh) = b(g) + g\cdot b(h)$ for all $g,h\in G$, and that the $1$-coboundaries are the \textit{principal crossed homomorphisms}, i.e., $1$-cocycles of the form $g\mapsto gm-m$ for some $m\in M$. 
As such, $H_c^1(G,M)$ vanishes if and only if every continuous crossed homomorphism is principal. 
Meanwhile, by \cite[(5.3)]{Hu:52}, the group $H_c^2(G,M)$ classifies isomorphism classes of extensions in the category of topological groups 
\[
0
\to M
\hookrightarrow \Gamma 
\xrightarrow{\pi} G
\to 1
\]
with conjugation by $\Gamma$ inducing the given action of $G$ on $M$, such that there exists a continuous (but not necessarily homomorphic) section $\varphi\colon G\rightarrow\Gamma$ that splits $\pi$ (i.e., such that $\pi\varphi=1_G$). 
The class of the split extension $\Gamma:= M\rtimes G$ corresponds to $0\in H_c^2(G,M)$, so if $H_c^2(G,M)=0$, then every extension with a continuous section has a homomorphic section.

With this recap, we give an almost immediate group cohomology--based proof of \cref{prop.extension-splits}.

\begin{proof}[Cohomology-based proof]
The hypothesis of \cref{prop.extension-splits} states that $T$ is a real vector space, and also immediately implies that $G$ is compact. 
Let $W$ be any finite-dimensional real $G$-representation. 
By \cite[Theorem~2.8]{Hu:52}, $H_c^i(G,W) = 0$ for each $i\geq 1$. 
In particular, $H_c^2(G,T) = 0$ and $H_c^1(G,V) = 0$.
    
The Lie group surjection $\pi|_\Gamma\colon\Gamma \rightarrow G$ defines a fiber bundle with fiber $T$. 
Since $T$ is a real vector space, thus solid (see \cite[Section~12.1]{Steenrod:99}), and $G$ is compact Hausdorff, the bundle has a continuous section by \cite[Theorem~12.2]{Steenrod:99}. 
Then by the above discussion, $H_c^2(G,T) = 0$ implies the existence of the desired continuous homomorphic section $\varphi\colon G\to \Gamma$. 
    
Next, one can check that the map $k\colon G\to V$ defined by $k(g) := \varphi(g)(0)$ is a continuous crossed homomorphism with respect to the action of $G$ on $V$. 
Since $H_c^1(G,V) = 0$, it follows that $k$ is principal, so it has the form $k(g)=g\cdot q - q$ for some $q\in V$. 
It follows that $\varphi(g)(x) = g\cdot x + k(g) = g\cdot (x+q)-q$ for each $g\in G$ and $x\in V$, so $p:=-q$ is a fixed point of $\varphi(G)$.
\end{proof}

The following is an alternative proof of \cref{prop.extension-splits} that avoids group cohomology, at least explicitly. 
(The map $b\colon G\rightarrow T^\perp$ constructed in the proof turns out to be a $1$-cocycle, so cohomology lurks beneath the surface.)

\begin{proof}[Prima facie cohomology-free proof]
For $c\in V$, let $t_c\in \E(V)$ denote translation by $c$. 
For each $h\in \Gamma$, we have a unique expression $h(x)= \pi(h)\cdot x+h(0)$ for $x\in V$, where $\pi(h)\in G\leq \O(V)$ and $h(0)\in V$. 
We claim that there exists a unique function $b\colon G \to T^\perp$ such that for each $A\in G$, it holds that
\[
\{h(0):h \in \pi^{-1}(A)\} 
= b(A)+T.
\]
To see this, note that the left-hand side is a nonempty union of $T$-cosets since $\pi(T\circ h) = \pi(h)$ and $T\circ h(0) = h(0)+T$ for each $h\in \Gamma$. 
It is a unique coset since if $\pi(h_1)=\pi(h_2)=A$, then $h_1\circ h_2^{-1} = t_{h_1(0)-h_2(0)}$, and so $h_1(0)\in h_2(0)+T$. 
This proves the claim.
    
Now consider the map $\varphi\colon G\to \Gamma$, given by $\varphi(A)(x) := Ax+b(A)$. 
Then $\pi\circ \varphi(A) = A$ for each $A\in G$, that is, $\varphi$ is a (set-map) section of $\pi$. 
Moreover, we have that
\[
\varphi(A)\circ \varphi(B)(x) 
= ABx + b(A)+A\cdot b(B) 
= ABx+b(AB) 
= \varphi(AB)(x),
\]
where the middle equality follows from the $A$-invariance of $T^\perp$ (noted at the beginning of the section) and the uniqueness of $b(AB)$ in $T^\perp$ as established in the aforementioned claim. 
So $b(A)+A\cdot b(B)=b(AB)$, i.e., $b$ is a crossed homomorphism, and $\varphi$ is a homomorphism. 
    
We are done once we show that $\varphi$ is smooth and that $\varphi(G)$ has a fixed point. 
Assuming smoothness, the latter follows by taking $v\in V$ any vector and defining
\[
p
:= \int_G \varphi(g)v d\mu_G(g),
\]
where the integral is with respect to the normalized (left-invariant) Haar measure $\mu_G$ on the compact Lie group $G$. 
This point is evidently fixed by every element of $\varphi(G)\leq \Gamma$ (in particular, $b(A) = p-Ap$ for each $A\in G$). 

We are left to prove smoothness. 
Since $H$ is a closed Lie subgroup of $\Gamma$, and $\pi|_\Gamma$ is a surjective Lie group homomorphism with kernel $H$, Noether's first isomorphism theorem for Lie groups entails that $\pi|_\Gamma$ is a submersion, so it possesses smooth local sections. 
Take an arbitrary $g\in G$, and let $\varphi_0\colon U\to \Gamma$ denote a smooth local section of $\pi|_\Gamma$ over an open neighborhood $U\subseteq G$ of $g$. 
Let $P\colon V\to T^\perp$ denote the linear orthogonal projection of $V$ onto $T^\perp$. 
Then the map $\Psi\colon \Gamma\to \Gamma$ given by $\Psi(h)(x):= \pi(h)\cdot x + P(h(0))$ is smooth and satisfies $\pi|_\Gamma\circ \Psi = \pi|_\Gamma$, as well as $\Psi(h)(0) = b(\pi(h))$ by the definition of $b$. 
This entails that $\varphi|_U = \Psi\circ \varphi_0$, and so $\varphi|_U$ is smooth. 
Overall, it follows that $\varphi$ is smooth, as desired.
\end{proof}

\section{Proof of Proposition~\ref{prop.so max filter}}
\label{app.calc proof}
Since $\SO(r)$ is compact, the function $M\mapsto \max_{Q\in \SO(r)}\operatorname{Tr}(QM)$ is continuous over $M \in \mathbb R^{r\times r}$. 
By continuity, we may pass to the open and dense subset where $M$ has distinct singular values. 
We proceed under this assumption for the rest of the proof.
    
Let $M = U_0\Sigma V_0$ be a singular value decomposition with $U_0,V_0\in \O(r)$. 
By the cyclic property of the trace, 
\[
\max_{Q\in \SO(r)}\operatorname{Tr}(QM)
= \max_{Q\in \SO(r)}\operatorname{Tr}(Q\Sigma_{\varepsilon}),
\]
where $\varepsilon = \mathbf{1}_{\{\det(M) < 0\}}$, $\Sigma_{0} = \Sigma$, and $\Sigma_{1} = J\Sigma = \Sigma J$, with $J$ denoting the identity matrix with its last diagonal entry negated.
    
If $\varepsilon = 0$, then the maximum is achieved at $Q = \Id_r$, since $|q_{ii}|\leq 1$ for all $Q\in \SO(r)$ and all $i\in [r]$. 
So we now assume $\varepsilon = 1$.

Let $U\in \arg\max_{Q\in \SO(r)}\operatorname{Tr}(Q\Sigma J)$. 
Then $U$ is a critical point of the map $Q\mapsto \operatorname{Tr}(Q\Sigma J)$, defined on $\SO(r)$. 
Taking directional derivatives along the tangent space
\[
T_{U}\SO(r) 
= \{A^TU:A^T = -A \in \mathbb R^{n\times n}\},
\]
we obtain $\operatorname{Tr}(A^TU\Sigma J) = 0$ for all skew-symmetric matrices $A \in \operatorname{Skew}^{r\times r}$. 
In other words, $U\Sigma J \in (\operatorname{Skew}^{r\times r})^\perp = \operatorname{Sym}^{r\times r}$, i.e., $U\Sigma J$ is symmetric.
    
This implies that both $U\Sigma J = J \Sigma U^T$ are valid singular value decompositions of the same matrix. 
Since $\Sigma$ has distinct diagonal entries, it follows that $U$ is diagonal.  
Considering we also have $U\in \SO(r)$, the diagonal entries of $U$ are necessarily $\pm 1$, with an even number of $-1$s. 
Thus,
\[
\Tr(U\Sigma J) 
= \max_{\substack{\varepsilon \in \{-1,1\}^r\\ \prod_{i=1}^r \varepsilon_i = -1}} \sum_{i=1}^r \varepsilon_i\sigma_i(M),
\]
which yields the desired formula.


\begin{thebibliography}{WW}
    \bibitem{AgarwalRT:20}
    I.\ Agarwal, O.\ Regev, Y.\ Tang,
    Nearly optimal embeddings of flat tori,
    (2020) arXiv:2005.00098.

    \bibitem{AlexandrinoB:15}
    M.\ M.\ Alexandrino, R.\ G.\ Bettiol,
    Lie groups and geometric aspects of isometric actions,
    Springer, 2015.

    \bibitem{Alharbi:22}
    W.\ Alharbi, S.\ Alshabhi, D.\ Freeman, D.\ Ghoreishi,
    Locality and stability for phase retrieval,
    Sampl. Theo. Sig. Proc. Dat. Anal. 22(1) (2024) 10.

    \bibitem{AmirBDE:25}
    T.\ Amir, T.\ Bendory, N.\ Dym, D.\ Edidin,
    The stability of generalized phase retrieval problem over compact groups,
    (2025) arXiv:2505.04190.
    
    \bibitem{AndoniNN:16}
    A.\ Andoni, A.\ Naor, O.\ Neiman, 
    Impossibility of sketching of the 3D transportation metric with quadratic cost, 
    ICALP 2016.

    \bibitem{BalanCE:06}
    R.\ Balan, P.\ Casazza, D.\ Edidin,
    On signal reconstruction without phase,
    Appl.\ Comput.\ Harmon.\ Anal.\ 20 (2006) 345--356.

    \bibitem{BalanD:22}
    R.\ Balan, C.\ B.\ Dock,
    Lipschitz analysis of generalized phase retrievable matrix frames,
    SIAM J.\ Mat.\ Anal.\ Appl.\ 43(3) (2022) 1518--1571.

    \bibitem{BalanHS:22}
    R.\ Balan, N.\ Haghani, M.\ Singh,
    Permutation invariant representations with applications to graph deep learning,
    arXiv:2203.07546.
    
    \bibitem{BalanT:23}
    R.\ Balan, E.\ Tsoukanis,
    G-invariant representations using coorbits: Bi-lipschitz properties,
    arXiv:2308.11784.

    \bibitem{BalanW:15}
    R.\ Balan, Y.\ Wang,
    Invertibility and robustness of phaseless reconstruction,
    Appl.\ Comput.\ Harmon.\ Anal.\ 38 (2015) 469--488.

    \bibitem{BandeiraCMN:14}
    A.\ S.\ Bandeira, J.\ Cahill, D.\ G.\ Mixon, A.\ A.\ Nelson,
    Saving phase:\ Injectivity and stability for phase retrieval,
    Appl.\ Comput.\ Harmon.\ Anal.\ 37 (2014) 106--125.

    \bibitem{Burnside:04}
    W.\ Burnside, 
    On groups of order $p^\alpha q^\beta$, 
    Proc.\ London\ Math.\ Soc. 2 (1904) 388-392
    
    \bibitem{CahillCD:16}
    J.\ Cahill, P.\ Casazza, I.\ Daubechies,
    Phase retrieval in infinite-dimensional Hilbert spaces,
    Trans.\ Amer.\ Math.\ Soc., Ser.\ B 3 (2016) 63--76.

    \bibitem{CahillIM:24}
    J.\ Cahill, J.\ W.\ Iverson, D.\ G.\ Mixon,
    Towards a bilipschitz invariant theory,
    Appl. Comp. Harm. Anal. 72 (2024) 101669.

    \bibitem{CahillIMP:22}
    J.\ Cahill, J.\ W.\ Iverson, D.\ G.\ Mixon, D.\ Packer,
    Group-invariant max filtering,
    Found. Comp. Math. (2024) 1--38.




    \bibitem{ConcaEHV:15}
    A.\ Conca, D.\ Edidin, M.\ Hering, C.\ Vinzant,
    An algebraic characterization of injectivity in phase retrieval,
    Appl.\ Comput.\ Harmon.\ Anal.\ 38 (2015) 346--356.

    \bibitem{Conway:02}
    J.\ H.\ Conway, D.\ H.\ Huson,
    The Orbifold Notation for Two-Dimensional Groups,
    Structural Chemistry 13, 247–257 (2002).

    \bibitem{Conway:16}
    J.\ H.\ Conway, H. Burgiel, C. Goodman-Strauss,
    The Symmetries of Things,
    CRC Press (2016).

    \bibitem{Dadok:85}
    J.\ Dadok,
    Polar Coordinates Induced by Actions of Compact Lie Groups,
    Trans. Amer. Math. Soc. 288, 1 (1985) 125--137.

    
    \bibitem{Eriksson:18}
    S.\ Eriksson-Bique,
    Quantitative bi-Lipschitz embeddings of bounded-curvature manifolds and orbifolds,
    Geom. and Top. 22 (2018) 1961--2026.

    \bibitem{Federer:14}
    H.\ Federer,
    Geometric measure theory,
    Springer (2014).

    \bibitem{Folland:99}
    G.\ B.\ Folland,
    Real analysis: modern techniques and their applications,
    John Wiley \& Sons (1999).

    \bibitem{Fremlin:00}
    D.\ H.\ Fremlin,
    Measure theory,
    Vol 4 \& 5 Torres Fremlin (2000).

    
    \bibitem{GroveB:96}
    L.\ C.\ Grove, C.\ T.\ Benson,
    Finite reflection groups,
    Springer Sci. \& Busi. Med. (1996) Vol. 99.

    \bibitem{GroveK:02}
    K.\ Grove,
    Geometry of, and via, symmetries,
    Uni. Lect. Series-AMS, 27, 31--51 (2002).

    \bibitem{HavivR:10}
    I.\ Haviv, O.\ Regev,
    The Euclidean Distortion of Flat Tori,
    Approx.\ Rand.\ Comb.\ Opt.\ Alg.\ \&\ Tech. (2010) 232--245.

    \bibitem{HeimendahlLVZ:22}
    A.\ Heimendahl, M.\ L\"ucke, F.\ Vallentin, M.\ C.\ Zimmermann,
    A semidefinite program for least distortion embeddings of flat tori into Hilbert spaces,
    arXiv:2210.11952.


    \bibitem{Hu:52}
    S.\ Hu,
    Cohomology theory in topological groups,
    Mich. Math. 1 (1952) 11--59.

    \bibitem{HytonenNVW:16}
    T.\ Hyt\"onen, J.\ V.\ Neerven, M.\ Veraar, L.\ Weis,
    Analysis in Banach spaces,
    Vol 12 Berlin: Springer (2016).
    
    \bibitem{JonesOR:11}
    P.W.\ Jones, A.\ Osipov, V.\ Rokhlin,
    Randomized approximate nearest neighbors algorithm.
    ONAS\ 108.38 (2011): 15679-15686.

    \bibitem{Kapovich:24}
    M.\ Kapovich,
    A note on properly discontinuous actions,
    S\~{a}o Paulo J. Math. Sci. 18, 2 (2024) 807--836.
    
    \bibitem{KhotN:06}
    S.\ Khot, A.\ Naor,
    Nonembeddability theorems via Fourier analysis,
    Mathematische Annalen 334 (2006) 821--852.

    \bibitem{Kramer:22}
    L.\ Kramer,
    Some remarks on proper actions, proper metric spaces, and buildings,
    Adv. Geom. 22, 4 (2022) 541--559.
    

    \bibitem{Lee:18}
    J.\ M.\ Lee,
    Introduction to Riemannian manifolds,
    Springer, 2018.

    \bibitem{MacBeath:67}
    M.\ MacBeath,
    The classification of non-euclidean plane crystallographic groups,
    Canad. J. Math. 19 (1967) 1192--1205.

    \bibitem{MakarychevM:16}
    K.\ Makarychev, Y.\ Makarychev,
    A union of Euclidean metric spaces is Euclidean,
    Discrete Analysis Journal
    (2016).
    

    \bibitem{MirandaT:93}
    H.\ Miranda, R.\ C.\ Thompson,
    A trace inequality with a subtracted term,
    Lin. Alg. and its Appl. 185 (1993) 165--172.

    \bibitem{MixonP:22}
    D.\ G.\ Mixon, D.\ Packer,
    Max filtering with reflection groups,
    Adv. Comp. Math. 49(6) (2023) 82.

    \bibitem{MixonQ:22}
    D.\ G.\ Mixon, Y.\ Qaddura,
    Injectivity, stability, and positive definiteness of max filtering,
    Constructive Approximation (2025).
    
    \bibitem{Qaddura:25}
    Y.\ Qaddura,
    A max filtering local stability theorem with application to weighted phase retrieval and cryo-EM,
    arxiv:2403.14042.

    \bibitem{Salamon:16}
    D.\ Salamon,
    Measure and integration,
    London EMS (2016).
    


    \bibitem{Steenrod:99}
    N.\ Steenrod,
    The topology of fibre bundles,
    Princeton university press, vol 14 (1999).

    
    \bibitem{VallentinM:23}
    F.\ Vallentin, P.\ Moustrou,
    Least distortion Euclidean embeddings of flat tori,
    ISSAC (2023) 13--23.

    \bibitem{Vries:78}
    J.\ De Vries,
    The local weight of an effective locally compact transformation group and the dimension of $L^2(G)$,
    Coll. Math. Vol. 39 IMPAN (1978) 319--323.
    
    \bibitem{XiaXX:24}
    Y.\ Xia, Z.\ Xu, Z.\ Xu,
    Stability in phase retrieval: Characterizing condition numbers and the optimal vector set,
    Math. Comp. (2024).

    \bibitem{Zolotov:19}
    V.\ Zolotov,
    Bi-lipschitz embeddings of SRA-free spaces into Euclidean spaces,
    arXiv:1906.02477 (2019).
\end{thebibliography}
\end{document}